\documentclass[12pt]{article}
\usepackage[margin =1in]{geometry}
\usepackage{amssymb,amsthm,amsmath}
\usepackage{enumerate}
\usepackage{enumitem}
\usepackage{hyperref}
\usepackage{cite}
\usepackage{cleveref}
\usepackage{color}
\usepackage[dvipsnames]{xcolor}

\usepackage{graphicx}
\usepackage{subcaption}

\numberwithin{equation}{section}

\newcommand{\cT}{\mathcal{T}}
\newcommand{\cS}{\mathcal{S}}
\newcommand{\tr}{\mathrm{Tr}}
\newcommand{\vect}{\mathbf{vec}}
\newcommand{\cI}{\mathcal{I}}

\newcommand{\cD}{\mathcal{D}}
\newcommand{\lr}{\kappa_1}
\newcommand{\ur}{\kappa_2}
\newcommand{\ir}{\kappa_3}
\newcommand{\muo}{\mu_0}
\newcommand{\po}{p_0}
\newcommand{\co}{\alpha}
\newcommand{\cX}{\mathcal{X}}

\newcommand{\norm}[1] {\left \| #1 \right \|}
\newcommand{\inclu}[0] {\ar@{^{(}->}}

\newcommand{\rank}{\text{Rank }\,}

\newcommand{\proj}{\mathrm{proj}}

\newcommand{\dist}{{\rm dist}}
\newcommand{\R}{{\bf R}}

\newcommand{\cA}{\mathcal{A}}
\newcommand{\cE}{\mathcal{E}}

\newcommand{\cN}{\mathcal{N}}
\newcommand{\EE}{\mathbb{E}}
\newcommand{\EEE}{{\bf E}}
\newcommand{\YY}{\bf{Y}}
\newcommand{\trace}{\mathrm{Tr}}
\newcommand{\sign}{\mathrm{sign}}

\newcommand{\RR}{{\bf R}}
\newcommand{\PP}{\mathbb{P}}
\renewcommand{\SS}{\mathbb{S}}
\newcommand{\pfail}{p_{\mathrm{fail}}}
\newcommand{\lv}{p}
\newcommand{\rv}{q}
\newcommand{\lM}{P}

\newcommand{\op}{\mathrm{op}}

\newcommand{\abs}[1]{\left| #1 \right|}

\newcommand{\argmin}{\operatornamewithlimits{argmin}}

\newcommand\eqd{\stackrel{\mathit{d}}{=}} 

\newtheorem{thm}{Theorem}[section]

\newtheorem{proposition}[thm]{Proposition}
\newtheorem{lem}[thm]{Lemma}
\newtheorem{lemma}[thm]{Lemma}

\newtheorem{corollary}[thm]{Corollary}
\newtheorem{conjecture}[thm]{Conjecture}

\newtheorem{claim}{Claim}

\theoremstyle{definition}
\newtheorem{definition}[thm]{Definition}
\newtheorem{assumption}{Assumption}

\usepackage{mathtools}
\DeclarePairedDelimiter{\dotp}{\langle}{\rangle}
\usepackage[boxruled]{algorithm2e}

\makeatletter
\newcommand{\opnorm}{\@ifstar\@opnorms\@opnorm}
\newcommand{\@opnorms}[1]{%
	\left|\mkern-1.5mu\left|\mkern-1.5mu\left|
	#1
	\right|\mkern-1.5mu\right|\mkern-1.5mu\right|
}
\newcommand{\@opnorm}[2][]{%
	\mathopen{#1|\mkern-1.5mu#1|\mkern-1.5mu#1|}
	#2
	\mathclose{#1|\mkern-1.5mu#1|\mkern-1.5mu#1|}
}
\makeatother

\addtocontents{toc}{\protect\setcounter{tocdepth}{2}}

\begin{document}

	\title{Low-rank matrix recovery with composite optimization: \\ good conditioning and rapid convergence}

	\author{Vasileios Charisopoulos\thanks{School of ORIE, Cornell University, Ithaca, NY 14850,
			USA; \texttt{people.orie.cornell.edu/vc333/}}\qquad Yudong Chen\thanks{School of ORIE, Cornell University,
			Ithaca, NY 14850, USA;
			\texttt{people.orie.cornell.edu/yudong.chen/}}\qquad Damek Davis\thanks{School of ORIE, Cornell University,
Ithaca, NY 14850, USA;
\texttt{people.orie.cornell.edu/dsd95/}.} \\ Mateo D\'iaz\thanks{CAM, Cornell University. Ithaca, NY 14850, USA;
	\texttt{people.cam.cornell.edu/md825/}} \qquad Lijun Ding\thanks{School of ORIE, Cornell University,
	Ithaca, NY 14850, USA;
	\texttt{people.orie.cornell.edu/ld446/}.}\qquad Dmitriy Drusvyatskiy\thanks{Department of Mathematics, U. Washington,
Seattle, WA 98195; \texttt{www.math.washington.edu/{\raise.17ex\hbox{$\scriptstyle\sim$}}ddrusv}. Research of Drusvyatskiy was supported by the NSF DMS   1651851 and CCF 1740551 awards.}}

\date{}
\maketitle

\begin{abstract}
	The task of recovering a low-rank matrix from its noisy linear measurements plays a central role in computational science. Smooth formulations of the problem often exhibit an undesirable phenomenon: the condition number, classically defined, scales poorly with the dimension of the ambient space. In contrast, we here show that in a variety of concrete circumstances,  nonsmooth penalty formulations do not suffer from the same type of ill-conditioning.
	 Consequently, standard algorithms for nonsmooth optimization, such as  subgradient and prox-linear methods, converge at a rapid dimension-independent rate when initialized within constant relative error of the solution. Moreover, nonsmooth formulations are naturally robust against outliers. Our framework subsumes such important computational tasks as phase retrieval, blind deconvolution, quadratic sensing,  matrix completion, and robust PCA. Numerical experiments on these problems illustrate the benefits of the  proposed approach.
\end{abstract}

\newpage
\tableofcontents
\newpage

\section{Introduction}
\label{sec:intro}
Recovering a low-rank matrix from noisy linear measurements has become an increasingly central task in data science. Important and well-studied examples include phase retrieval \cite{7078985,wirt_flow,ma2017implicit}, blind deconvolution \cite{ahmed2014blind,li2016rapid,MR3424852,proc_flow}, matrix completion \cite{rec_exa,davenport2016overview,MR3565131}, covariance matrix estimation \cite{MR3367819,li2018nonconvex}, and robust principal component analysis \cite{chand,rob_cand}. Optimization-based approaches for low-rank matrix recovery naturally lead to nonconvex formulations, which are NP hard in general. To overcome this issue, in the last two decades researchers have developed convex relaxations that succeed with high probability under appropriate statistical assumptions.  Convex techniques, however, have a well-documented limitation: the parameter space describing the relaxations is usually much larger than that of the target problem.  Consequently, standard algorithms applied on convex relaxations may not scale well to the large problems. Consequently, there has been a renewed interest in directly optimizing nonconvex formulations with iterative methods within the original parameter space of the problem. Aside from a few notable exceptions on specific problems \cite{mat_comp_min,bhojanapalli2016global,ge2017unified},  most algorithms of this type proceed in  two-stages. The first stage---{\em initialization}---yields a rough estimate of an optimal solution, often using spectral techniques.
The second stage---{\em local refinement}---uses a local search algorithm that
rapidly converges to an optimal solution,  when initialized at the output of
the initialization stage.

This work focuses on developing provable low-rank matrix recovery algorithms
based on nonconvex problem formulations. We focus primarily on local refinement
and describe a set of unifying sufficient conditions leading to rapid local
convergence of iterative methods. In contrast to the current literature on the
topic, which typically relies on smooth problem formulations and gradient-based
methods, our primary focus is on \emph{nonsmooth formulations} that exhibit
sharp growth away from the solution set. Such formulations are well-known in
the nonlinear programming community to be amenable to rapidly convergent
local-search algorithms. Along the way, we will observe an apparent benefit of
nonsmooth formulations over their smooth counterparts. All nonsmooth
formulations analyzed in this paper are ``well-conditioned," resulting in fast
``out-of-the-box" convergence guarantees. In contrast, standard smooth
formulations for the same recovery tasks can be poorly conditioned, in the
sense that classical convergence guarantees of nonlinear programming are overly
pessimistic.
Overcoming the poor conditioning typically requires nuanced problem and algorithmic specific analysis (e.g. \cite{proc_flow,ma2017implicit,MR3025133,rand_quad}), which nonsmooth formulations manage to avoid for the problems considered here.


Setting the stage, consider a rank $r$ matrix $M_{\sharp}\in \R^{d_1\times d_2}$ and a linear  map $\cA\colon \R^{d_1\times d_2} \to \R^m$ from the space of matrices to the space of measurements.
The goal of  low-rank matrix recovery  is to recover $M_{\sharp}$
from the image vector $b=\mathcal{A}(M_{\sharp})$, possibly corrupted by noise.
Typical nonconvex approaches
proceed by choosing some penalty function $h(\cdot)$ with which to measure the residual $\cA(M)-b$ for a trial solution $M$. Then, in the case that $M_{\sharp}$ is symmetric and positive semidefinite, one may focus on the
formulation
\begin{equation}\label{eqn:target_problem}
\min_{X\in\R^{d\times r}}~ f(X):=h\left(\mathcal{A}(XX^\top)-b\right)\qquad \textrm{subject to }X\in \cD,
\end{equation}
or when  $M_{\sharp}$ is rectangular, one may instead use the formulation
\begin{equation}\label{eqn:target_problem_asymm}
\min_{X\in \R^{d_1\times r}, ~Y\in \R^{r\times d_2}}~ f(X,Y):=h\left(\mathcal{A}(XY)-b\right)\qquad\textrm{subject to }(X,Y)\in \cD.
\end{equation}
Here, $\cD$ is a convex set that incorporates prior knowledge about
$M_{\sharp}$ and is often used to enforce favorable structure on the decision
variables. The penalty $h$ is chosen specifically to penalize measurement
misfit and/or enforce structure on the residual errors.

\subsection*{Algorithms and conditioning for smooth formulations} Most
widely-used penalties $h(\cdot)$ are smooth and convex. Indeed,  the
\emph{squared} $\ell_2$-norm $h(z)=\tfrac{1}{2}\|z\|^2_2$ is ubiquitous in this
context. With such penalties, problems \eqref{eqn:target_problem} and
\eqref{eqn:target_problem_asymm} are smooth and thus are amenable to
gradient-based methods. The linear rate of convergence of gradient descent is
governed by the ``local condition number'' of  $f$. Indeed, if the estimate,
$
\mu I \preceq \nabla^2 f(X) \preceq L I,
$ holds
for all $X$ in a neighborhood of the solution set, then gradient descent converges to the solution set at the linear rate $1 - \mu/L$. It is known that for several widely-studied problems including phase retrieval, blind deconvolution, and matrix completion, the ratio $\mu/L$ scales inversely with the problem dimension.  Consequently, generic nonlinear programming guarantees yield efficiency estimates that are far too pessimistic.  Instead, near-dimension independent guarantees can be obtained by arguing that $\nabla^2 f$ is well conditioned along  the ``relevant'' directions or that $\nabla^2 f$ is well-conditioned within a restricted region of space that the iterates never escape (e.g. \cite{proc_flow,ma2017implicit,MR3025133}). Techniques of this type have been elegantly and successfully used over the past few years to obtain algorithms with near-optimal sample complexity. One byproduct of such techniques, however, is that the underlying arguments are finely tailored to each particular problem and algorithm at hand. We refer the reader to the recent surveys \cite{chi2018nonconvex} for details.

\subsection*{Algorithms and conditioning for nonsmooth formulations}

The goal of our work is to justify the following principle:
\begin{quote}  Statistical assumptions for common  recovery problems guarantee that  \eqref{eqn:target_problem} and \eqref{eqn:target_problem_asymm} {\em are well-conditioned} when $h$ is an appropriate {\em nonsmooth convex penalty}.
\end{quote}

To explain what we mean by ``good conditioning," let  us treat  \eqref{eqn:target_problem}
and \eqref{eqn:target_problem_asymm} within the broader \emph{convex composite} problem class:
\begin{equation}\label{eqn:comp_prob}
\min_{x\in\cX}~ f(x):=h(F(x)),
\end{equation}
where $F(\cdot)$ is a smooth map on the space of matrices and $\cX$ is a closed
convex set.
Indeed, in the symmetric and positive semidefinite case, we identify $x$ with
matrices $X$ and  define $F(X)=\cA(XX^\top)-b$, while in the asymmetric case,
we identify $x$ with pairs  of matrices $(X,Y)$ and define $F(X,Y)=\cA(XY)-b$.
Though compositional problems \eqref{eqn:comp_prob} have been well-studied in
nonlinear programming \cite{burke_desc,burke_gauss,fletcher_model},  their
computational promise in data science has only begun recently to emerge. For
example, the papers \cite{duc_ruan_stoch_comp,davis2019stochastic,eff_paquette}
discuss stochastic and inexact algorithms on composite problems,
while the papers \cite{duchi_ruan_PR,davis2017nonsmooth},
\cite{charisopoulos2019composite}, and \cite{li2018nonconvex_robust}
investigate applications to phase retrieval, blind deconvolution, and matrix
sensing, respectively.

A number of algorithms are available for problems of the form
\eqref{eqn:comp_prob}, and hence for \eqref{eqn:target_problem}
and \eqref{eqn:target_problem_asymm}. Two most notable ones are the projected subgradient\footnote{Here, the subdifferential is
	formally obtained through the chain rule $\partial f(x)=\nabla F(x)^*\partial
	h(F(x))$, where $\partial h(\cdot)$ is the subdifferential in the sense of
	convex analysis.}
method~\cite{davis2018subgradient,goffin}
$$x_{t+1}=\proj_{\cX}(x_t-\alpha_t v_t)\qquad\textrm{with}\qquad v_t\in \partial f(x_t),$$
and the prox-linear algorithm~\cite{burke_desc,prox,drusvyatskiy2018error}
$$x_{t+1}=\argmin_{x\in\cX}~ h\Big(F(x_t)+\nabla F(x_t)(x-x_t)\Big)+\frac{\beta}{2}\|x-x_t\|^2_2.$$
Notice that each iteration of the subgradient method is relatively cheap,
requiring access only to the subgradients of $f$ and the nearest-point
projection onto $\cX$.  The prox-linear method in contrast requires solving a
strongly convex problem in each iteration. That being said, the prox-linear
method has much stronger convergence guarantees than the subgradient method, as
we will review shortly.

The local convergence guarantees of both methods are straightforward to describe, and underlie what we mean by ``good conditioning''. Define $\cX^*:=\argmin_{\cX} f$, and for any $x\in\cX$ define the convex model $f_x(y)=h(F(x)+\nabla F(x)(y-x))$.
Suppose there exist constants $\rho,\mu>0$ satisfying the two properties:
\begin{itemize}
	\item {\bf (approximation)} $\left|f(y)-f_x(y)\right|\leq \frac{\rho}{2}\|y-x\|^2_2$ for all $x,y\in \cX$,
	\item {\bf (sharpness)}  $f(x)-\inf f\geq \mu\cdot \dist(x,\cX^*)$  for all $x\in \cX$.
\end{itemize}
The approximation and sharpness properties have intuitive meanings.
The former says that the nonconvex function $f(y)$ is well approximated by the
convex model $f_x(y)$, with quality that degrades quadratically as $y$ deviates
from $x$. In particular, this property guarantees that the quadratically
perturbed function $x\mapsto f(x)+\frac{\rho}{2}\|x\|^2_2$ is convex on $\cX$. Yet another consequence of the approximation property is that the epigraph of $f$ admits a supporting concave quadratic with amplitute $\rho$ at each of its points. Sharpness, in turn, asserts that $f$ must grow at least linearly as $x$
moves away from the solution set. In other words, the function values should
robustly distinguish between optimal and suboptimal solutions. In statistical
contexts, one can interpret sharpness as  strong identifiability of
the statistical model. The three figures below  illustrate the approximation and sharpness properties for idealized objectives in phase retrieval, blind deconvolution, and robust PCA problems.

\begin{figure}[!h]
	\centering
	\begin{subfigure}[b]{0.32\textwidth}
		\centering\includegraphics[width=\textwidth]{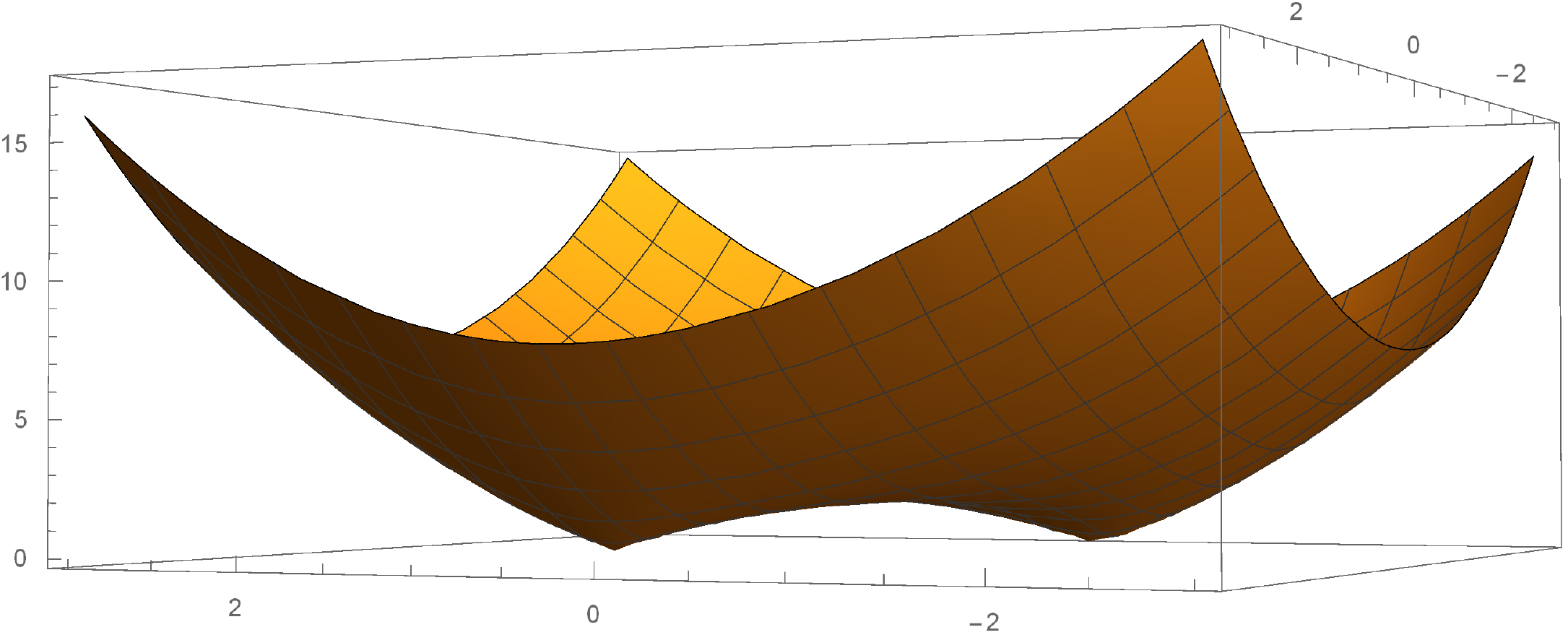}
	\end{subfigure}~
	\begin{subfigure}[b]{0.32\textwidth}
		\centering\includegraphics[width=\textwidth]{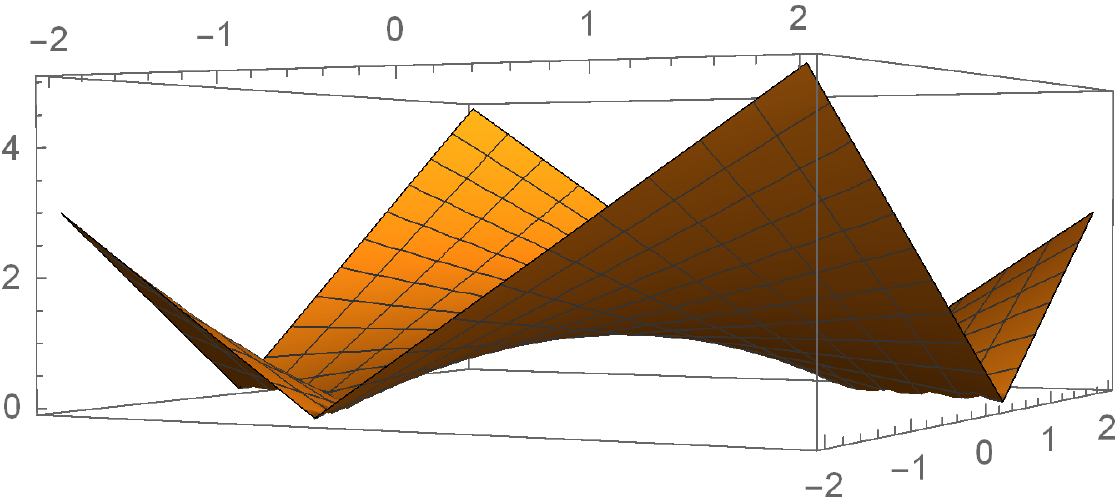}
	\end{subfigure}
	\begin{subfigure}[b]{0.32\textwidth}
		\centering\includegraphics[width=\textwidth]{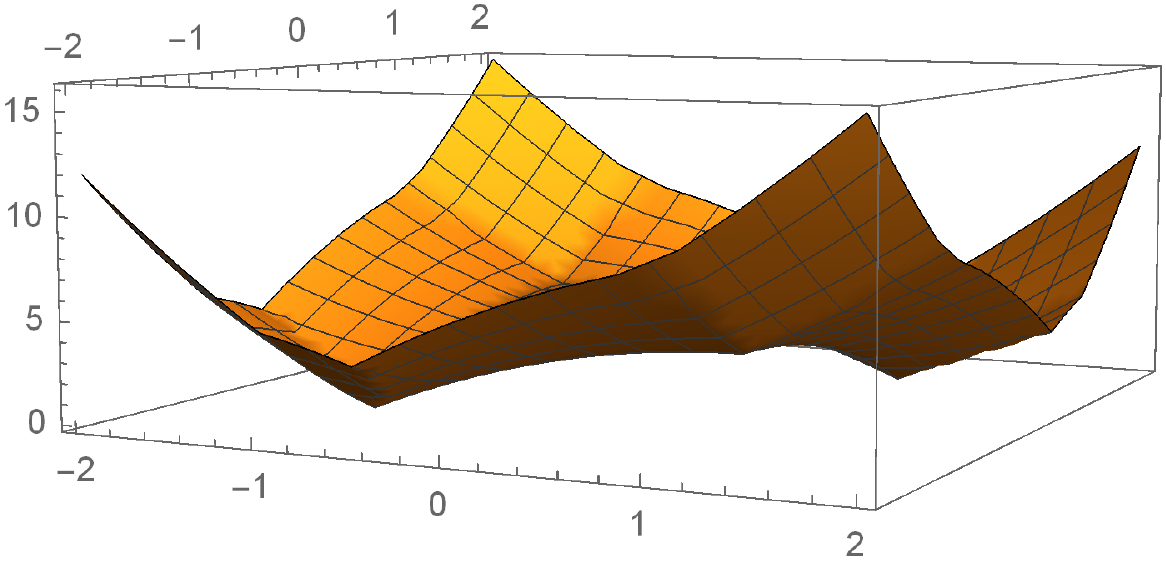}
	\end{subfigure}
	\begin{subfigure}[b]{0.32\textwidth}
		\centering\includegraphics[width=\textwidth]{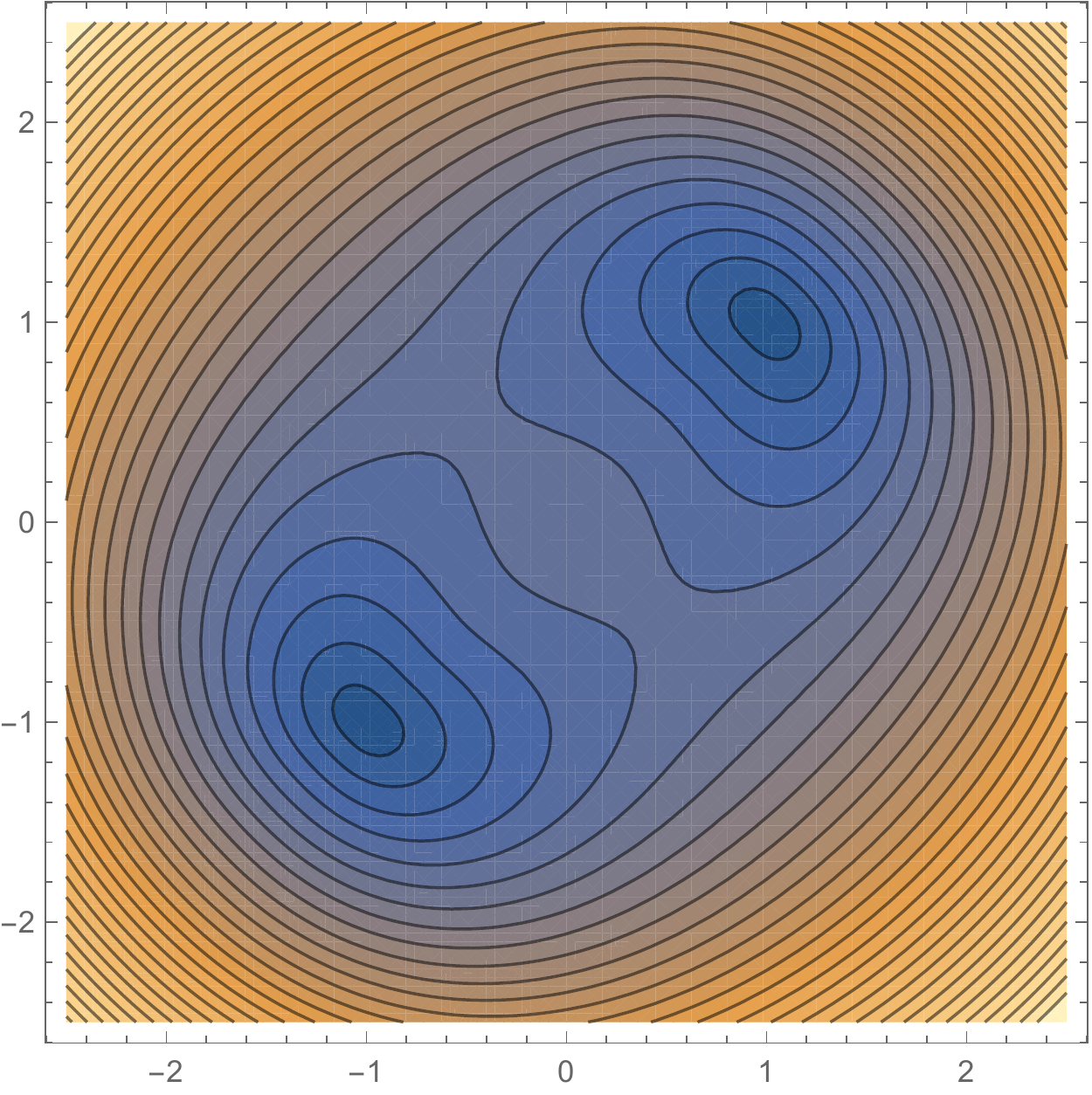}
		\caption{$f(x)=\EE |(a^\top x)^2-(a^\top {\bf 1})^2|$\\
			\centering(phase retrieval)}
		\label{fig:contour1}

	\end{subfigure}~
	\begin{subfigure}[b]{0.32\textwidth}
		\centering\includegraphics[width=\textwidth]{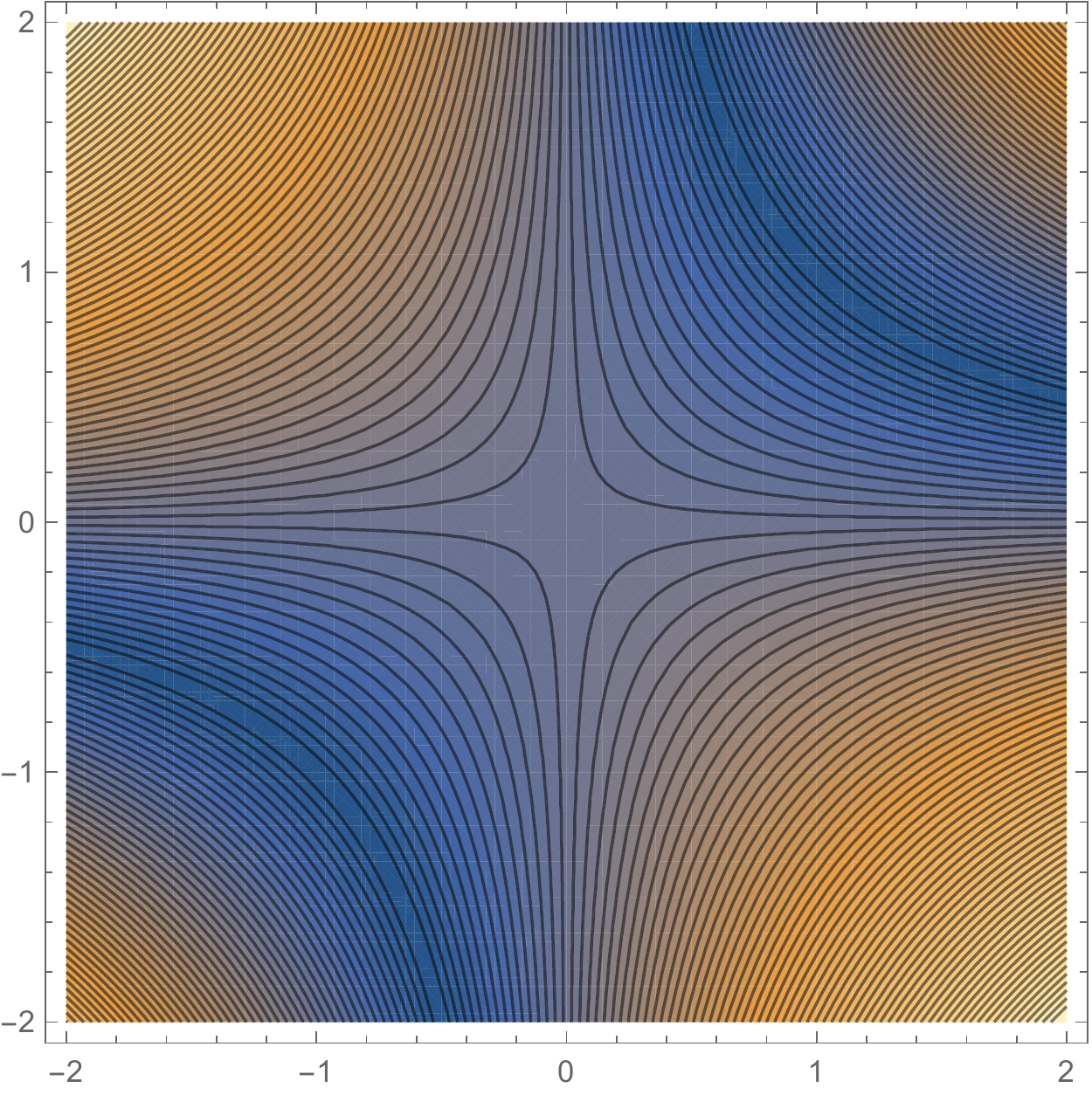}
		\caption{$f(x,y)=|xy-1|$\\ \centering(blind deconvolution)}
		\label{fig:contour2}
	\end{subfigure}
	\begin{subfigure}[b]{0.32\textwidth}
		\centering\includegraphics[width=\textwidth]{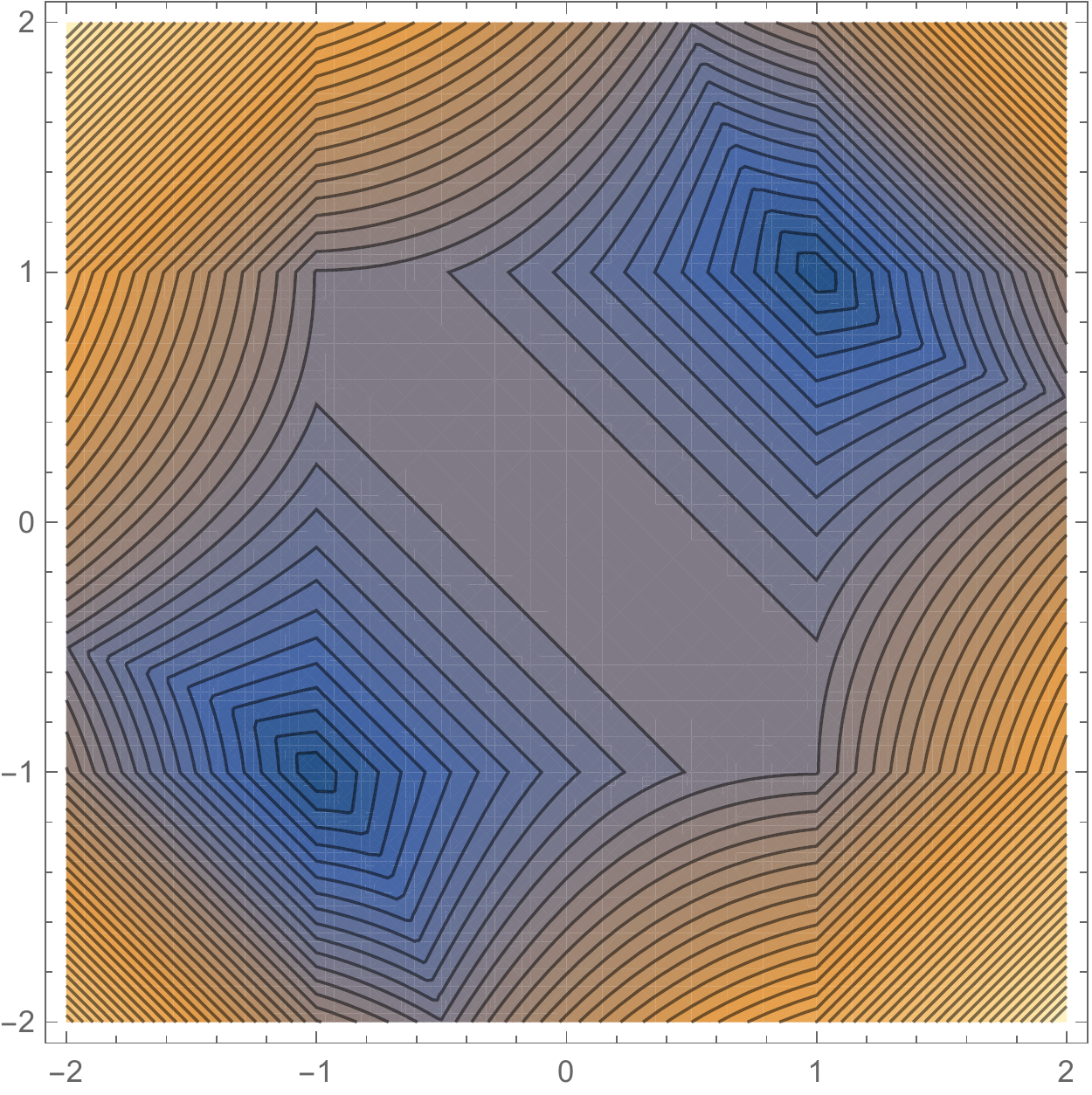}
		\caption{$f(x)=\|xx^\top-{\bf 1}{\bf 1}^\top\|_1$\\
			\centering(robust PCA)}
		\label{fig:contour3}
	\end{subfigure}
\end{figure}

Approximation and sharpness, taken together,  guarantee rapid convergence of numerical methods when initialized within the tube:
$$\mathcal{T}= \Big\{x\in \cX:\dist(x,\cX^*)\leq \frac{\mu}{\rho} \Big\}.$$
For common low-rank  recovery problems,  $\mathcal{T}$ has an intuitive interpretation: it consists of those matrices that are within constant relative error of the solution. We note that standard spectral initialization techniques, in turn, can generate such matrices with nearly optimal sample complexity. We refer the reader to the survey \cite{chi2018nonconvex}, and references therein, for details.

\paragraph{Guiding strategy.} The following is the guiding algorithmic  principle of this work:
\begin{quote}
	When initialized at $x_0\in\mathcal{T}$, the prox-linear algorithm converges quadratically to the solution set $ \cX^* $;  the subgradient method, in turn, converges linearly with a rate governed by ratio $\frac{\mu}{L}\in (0,1)$, where $L$ is the Lipschitz constant of $f$ on $\mathcal{T}$.\footnote{Both the parameters $\alpha_t$ and $\beta$ must be properly chosen for these guarantees to take hold.}
\end{quote}
In light of this observation, our strategy can be succinctly summarized as follows. We will show that for a variety of low-rank recovery problems, the parameters $\mu,L,\rho>0$ (or variants) are dimension independent under standard statistical assumptions. Consequently, the  formulations \eqref{eqn:target_problem} and \eqref{eqn:target_problem_asymm} are ``well-conditioned'', and subgradient and prox-linear methods converge rapidly when initialized within constant relative error of the optimal solution.

\subsection*{Approximation and sharpness via the Restricted Isometry Property}
We begin verifying our thesis by showing that the composite problems, \eqref{eqn:target_problem} and \eqref{eqn:target_problem_asymm}, are well-conditioned under the following Restricted Isometry Property (RIP): there exists a norm $\opnorm{\cdot}$ and numerical constants $\lr,\ur>0$ so that
\begin{equation}\label{eqn:RIP_intro}
\lr\|W\|_F\leq \opnorm{\cA(W)}\leq \ur\|W\|_F,
\end{equation}
for all matrices $W\in \R^{d_1\times d_2}$ of rank at most $2r$. We argue that under RIP,
\begin{quote} the \emph{nonsmooth} norm $h = \opnorm{\cdot }$ is a natural penalty function to use.
\end{quote}
Indeed, as we will show, the composite loss $h(F(x))$ in the symmetric setting
admits
constants $\mu,\rho,L$ that depend only on the RIP parameters and the extremal singular values of $M_{\sharp}$:
$$\mu=0.9\kappa_1 \sqrt{\sigma_r(M_{\sharp})}, \qquad \rho=\kappa_2, \qquad L=0.9\kappa_1\sqrt{\sigma_r(M_{\sharp})}+2\ur\sqrt{\sigma_1(M_{\sharp})}.$$
In particular, the initialization ratio scales as
$\frac{\mu}{\rho}\asymp\frac{\lr}{\ur}\sqrt{\sigma_r(M_{\sharp})}$ and the
condition number scales as $\frac{L}{\mu}\asymp
1+\frac{\ur}{\lr}\sqrt{\frac{\sigma_1(M_{\sharp})}{\sigma_r(M_{\sharp})}}$.
Consequently,
the rapid local convergence guarantees previously described immediately
take-hold. The asymmetric setting is slightly more nuanced since the objective
function is sharp only on bounded sets. Nonetheless, it can be analyzed in a
similar way leading to analogous rapid convergence guarantees. Incidentally,
we show that the prox-linear method converges rapidly without any
modification; this is in contrast to smooth methods, which typically require
incorporating an auxiliary regularization term into the objective (e.g.
\cite{proc_flow}). We note that similar results in the symmetric setting were
independently obtained in the complimentary work
\cite{li2018nonconvex_robust}, albeit with a looser estimate of
$L$; the two
treatments of the asymmetric setting are distinct, however.\footnote{The authors of \cite{li2018nonconvex_robust} provide a bound on
	$L$ that scales with the Frobenius norm $\sqrt{\|M_{\sharp}}\|_F$. We instead derive a sharper
	bound that scales as $\sqrt{\|M_{\sharp}\|_{\rm op}}$. As a byproduct, the
	linear rate of convergence for the subgradient method  scales only with the
	condition number $\sigma_1(M_{\sharp})/\sigma_r(M_{\sharp})$ instead of $\|M_{\sharp}\|_F/\sigma_r(M_{\sharp})$.}

After establishing basic properties of the composite loss, we turn our attention to verifying RIP in several concrete scenarios. We note that the seminal works \cite{guaran,candes2011tight} showed that if $\cA(\cdot)$ arises from a Gaussian ensemble, then in the regime $m\gtrsim r(d_1+d_2)$ RIP holds with high probability for the scaled $\ell_2$ norm $\opnorm{z}=m^{-1/2}\|z\|_2$. 
More generally when $\cA$ is highly structured, RIP may be most naturally measured in a non-Euclidean norm. For example,  RIP with respect to the scaled $\ell_1$ norm $\opnorm{z}=m^{-1}\|z\|_1$ holds for phase retrieval \cite{eM, duchi_ruan_PR}, blind deconvolution \cite{charisopoulos2019composite}, and quadratic sensing \cite{MR3367819}; in contrast, RIP relative to the scaled $\ell_2$ norm fails for all three problems.
In particular, specializing our results to the aforementioned recovery tasks yields solution methodologies with best known sample and computational complexity guarantees. Notice that while one may ``smooth-out" the $\ell_2$ norm by squaring it, we argue that it may be more natural to optimize the $\ell_1$ norm directly as a nonsmooth penalty. Moreover, we show that $\ell_1$ penalization enables exact recovery even if a constant fraction of measurements is corrupted by outliers.

\subsection*{Beyond RIP: matrix completion and robust PCA}

The RIP assumption provides a nice vantage point for analyzing the problem parameters $\mu,\rho, L>0$. There are, however, a number of important problems, which do not satisfy RIP. Nonetheless, the general paradigm based on the interplay of sharpness and approximation is still powerful. We consider two such settings, matrix completion and robust principal component analysis (PCA), leveraging some intermediate results from \cite{chen2015fast}.

The goal of the matrix completion problem \cite{rec_exa} is to recover a low rank matrix $M_{\sharp}$ from its partially observed entries. We focus on the formulation
$$\argmin_{X\in\cX}~f(X)=\|\Pi_{\Omega}(XX^\top)-\Pi_{\Omega}(M_{\sharp})\|_2,$$
where $\Pi_\Omega$ is the projection onto the index set of observed entries
$\Omega$ and
$$\mathcal{X} = \left\{X\in \mathbb{R}^{d\times r}: \|X\|_{2,\infty} \leq \sqrt{\frac{\nu r\|M_{\sharp}\|_{\rm op}}{d}}\,\right\}$$
is the set of incoherent matrices. To analyze the conditioning of this formulation, we assume that the indices in $\Omega$ are chosen as i.i.d.\ Bernoulli with parameter $p\in (0,1)$ and that all nonzero singular values of $M_{\sharp}$ are equal to one.
Using results of \cite{chen2015fast}, we quickly deduce sharpness  with high probability. The error in approximation, however, takes the following nonstandard form. In the regime $p\geq \frac{c}{\epsilon^2}(\frac{\nu^2 r^2}{d}+ \frac{\log d}{d})$ for some constants $c>0$ and $\epsilon \in (0,1)$, the estimate holds with high probability:
$$|f(Y)-f_X(Y)|\leq \sqrt{1+\epsilon}\|Y-X\|^2_2+\sqrt{\epsilon}\|X-Y\|_F\qquad \textrm{for all }X,Y\in \cX.$$ The following modification of the prox-linear method therefore arises naturally:
$$X_{k+1}=\argmin_{X\in \cX}~ f_{X_k}(X)+\sqrt{1+\epsilon}\|X-X_k\|^2_F+\sqrt{\epsilon}\|X-X_k\|_F.$$
We show that subgradient methods and the prox-linear method, thus modified,
both converge at a dimension independent linear rate when initialized near the
solution. Namely, as long as $\epsilon$ and $\dist(X_0,\cX^*)$ are below some constant thresholds, both
the subgradient and the modified
prox-linear methods converge linearly with high probability:
$$\dist(X_k,\cX^*)\lesssim
\begin{cases}
\left(1-\frac{c}{\nu r}\right)^{k/2} & \textrm{subgradient} \\
2^{-k} &  \textrm{prox-linear}
\end{cases}.$$
Here $c>0$ is a numerical constant.
Notice that the prox-linear method enjoys a much faster rate of convergence that is independent of any unknown constants or problem parameters---an observation fully supported by our numerical experiments.

As the final example, we consider the problem of robust PCA \cite{rob_cand,chand}, which aims to decompose a given matrix $W$ into a sum of a low-rank and a sparse matrix. We consider two different problem formulations:
\begin{equation}\label{eqn:l2_formul_intro}
\min_{(X,S)\in \cD_1}~ F((X,S))=\|XX^{\top} +S -W\|_F,
\end{equation}
and
\begin{equation}\label{eqn:intro_nonsmooth}
\min_{X\in \cD_2} f(X)=\|XX^\top-W\|_1,
\end{equation}
where $\cD_1$ and $\cD_2$ are appropriately defined convex regions. Under standard
incoherence assumptions, we show that the formulation
\eqref{eqn:l2_formul_intro} is well-conditioned, and therefore subgradient and
prox-linear methods are applicable. Still,
formulation~\eqref{eqn:l2_formul_intro} has a major drawback in that one must
know properties of the optimal sparse matrix $S_\sharp$ in order to define the
constraint set $\cD_1$, in order to ensure good conditioning. Consequently, we analyze
formulation~\eqref{eqn:intro_nonsmooth} as a more practical alternative.

The analysis of \eqref{eqn:intro_nonsmooth} is more challenging than that of~\eqref{eqn:l2_formul_intro}. Indeed, it appears that we must replace the Frobenius norm $\|X\|_F$ in the approximation/sharpness conditions with the sum of the row norms $\|X\|_{2,1}$. With this set-up, we verify the convex approximation property in general:
$$|f(Y)-f_X(Y)|\leq \|Y-X\|_{2,1}^2\qquad \textrm{for all }X,Y$$
and sharpness only when $r=1$. We conjecture, however, that an analogous sharpness bound holds for all $r$. It is easy to see that the quadratic convergence guarantees for the prox-linear method do not rely on the Euclidean nature of the norm, and the algorithm becomes applicable. To the best of our knowledge, it is not yet known how to adapt linearly convergent subgradient methods to the non-Euclidean setting.

\subsection*{Robust recovery with sparse outliers and dense noise}

The aforementioned guarantees lead to exact recovery of $M_\sharp$ under
noiseless or sparsely corrupted measurements $b$. A more realistic noise model
allows for further corruption by a dense noise vector $e$ of small norm. Exact
recovery is no longer possible with such errors. Instead, we should only expect
to recover $M_\sharp$ up to a tolerance proportional to the size of $e$.
Indeed, we show that appropriately modified subgradient and prox-linear
algorithms converge linearly and quadratically, respectively, up to the
tolerance $\delta = O( \opnorm{e} /\mu)$ for an appropriate norm
$\opnorm{\cdot}$. Finally, we discuss in detail the case of recovering a low
rank PSD matrix $M_\sharp$ from the corrupted measurements $\cA(M_\sharp) +
\Delta + e$, where $\Delta$ represents sparse outliers and $e$ represents small
dense noise. To the best of our knowledge, theoretical guarantees for this error model have not been previously established
in the nonconvex low-rank recovery literature. Surprisingly, we show it is
possible to recover the matrix $M_\sharp$ up to a tolerance \emph{independent}
of the norm or location of the outliers $\Delta$.

\subsection*{Numerical experiments}
We conclude with an experimental evaluation of our theoretical findings on quadratic and bilinear matrix sensing, matrix completion, and robust PCA problems.
In the first set of experiments, we test the robustness of the proposed methods against varying combinations
of rank/corruption level by
reporting the empirical recovery rate across independent runs of synthetic
problem instances. All the aforementioned model problems exhibit sharp phase
transitions, yet our methods succeed for more than moderate levels of
corruption (or unobserved entries in the case of matrix completion). For
example, in the case of matrix sensing, we can corrupt almost half of the
measurements $\mathcal{A}_i(M)$ and still retain perfect recovery rates.
Interestingly, our experimental findings indicate that the prox-linear method
can tolerate slightly higher levels of corruption compared to the subgradient
method, making it the method of choice for small-to-moderate dimensions.

We then demonstrate that the convergence rate analysis is fully supported by empirical evidence. In particular, we test the subgradient and prox-linear methods for different
rank/corruption configurations. In the case of quadratic/bilinear sensing and robust PCA, we observe
that the subgradient method converges linearly and the prox-linear method
converges quadratically, as expected. In particular, our numerical experiments appear to support our sharpness conjecture for the robust PCA problem.
In the case of matrix completion, both  algorithms converge linearly. The prox-linear method in particular, converges extremely quickly, reaching high accuracy solutions in under $25$ iterations for reasonable values of $p$.

In the noiseless setting, we compare against  gradient descent with constant step-size on smooth formulations of each problem (except for robust PCA). We
notice that the Polyak subgradient method outperforms gradient descent in all cases. That being said, one can heuristically equip gradient descent with the Polyak step-size  as well. To the best of our knowledge, the gradient method with Polyak step-size has has not been investigated on smooth problem formulations we consider here. Experimentally, we see that the Polyak (sub)gradient methods on smooth and nonsmooth formulations perform comparably in the noiseless setting.

\subsection*{Outline of the paper} The outline of the paper is as follows.
Section~\ref{sec:prelim} records some basic notation  we will use.
Section~\ref{sec:all_algos_reg_cond} informally discusses the sharpness and
approximation properties, and their impact on convergence of the subgradient
and prox-linear methods. Section~\ref{sec:ass_model} analyzes the parameters
$\mu,\rho, L$ under RIP. Section~\ref{sec:conv_guarant} rigorously discusses
convergence guarantees of numerical methods under regularity conditions.
Section~\ref{sec:all_examples} reviews examples of problems satisfying RIP and
deduces convergence guarantees for subgradient and prox-linear algorithms.
Sections~\ref{sec:mat_comp} and \ref{sec:robust_pca} discuss the matrix completion and
robust PCA problems, respectively. Section~\ref{sec:recovery_tol}
discusses robust recovery up to a noise tolerance. The final
Section~\ref{sec:num_exp} illustrates the developed theory and algorithms with
numerical experiments on quadratic/bi-linear sensing, matrix completion, and
robust PCA problems.

\section{Preliminaries}\label{sec:prelim}
In this section, we summarize the basic notation we will use throughout the paper.
Henceforth, the symbol $\EEE$ will denote a Euclidean space with inner product $\langle\cdot,\cdot \rangle$ and the induced norm $\|x\|_2=\sqrt{\langle x,x\rangle}$.  The closed unit ball in $\EEE$ will be denoted by $\mathbb{B}$, while a closed ball of radius $\epsilon>0$ around a point $x$ will be written as $B_{\epsilon}(x)$.
For any point $x\in \EEE$ and a set  $Q\subset\EEE$,  the distance and the nearest-point projection in $ \ell_2 $-norm are defined by
\begin{align*}
\dist(x;Q)=\inf_{y\in Q} \|x-y\|_2\qquad \textrm{and}\qquad\proj_Q(x)=\argmin_{y\in Q} \|x-y\|_2,
\end{align*}
respectively. For any pair of functions $f$ and $g$ on $\EEE$, the notation $f\lesssim g$ will mean that there exists a numerical constant $C$ such that $f(x)\leq C g(x)$ for all $x\in \EEE$.
Given a linear map between Euclidean spaces, $\cA\colon\EEE\to{\bf Y}$, the adjoint map will be written as $\cA^*\colon {\bf Y}\to \EEE$. We will use $I_d$ for the $d$-dimensional identity matrix and $\mathbf{0}$ for the zero matrix with variable sizes. The symbol $[m]$ will be shorthand for the set $\{1, \dots, m\}.$

We will always endow the Euclidean space of vectors $\R^d$ with the  usual dot-product $\langle x,y \rangle=x^{\top}y$ and the induced $\ell_2$-norm. More generally, the $\ell_p$ norm of a vector $x$ will be denoted by $\|x\|_p=(\sum_{i}|x_i|^p)^{1/p}$. Similarly, we will equip the space of rectangular matrices $\R^{d_1\times d_2}$ with the trace product $\langle X,Y\rangle=\tr(X^\top Y)$ and the induced Frobenius norm $\|X\|_F=\sqrt{\tr(X^\top X)}$.
The operator norm of a matrix $X\in \R^{d_1\times d_2}$ will be written as $\|X\|_{\textrm{op}}$. The symbol $\sigma(X)$ will denote the vector of singular values of a matrix $X$ in nonincreasing order. 
We also define the row-wise matrix norms $\|X\|_{b,a}=\|(\|X_{1\cdot}\|_b,\|X_{2\cdot}\|_b\,\ldots,\|X_{d_1\cdot}\|_b)\|_a$.
The symbols $\mathcal{S}^{d}$, $\mathcal{S}^{d}_+$, $O(d)$, and $GL(d)$ will denote the sets of symmetric, positive semidefinite, orthogonal, and invertible matrices, respectively.

Nonsmooth functions will play a central role in this work. Consequently, we will require some basic constructions of generalized differentiation, as described for example in the monographs \cite{RW98,Mord_1,cov_lift}.
Consider a function $f\colon\EEE\to\R\cup\{+\infty\}$ and a point $x$, with $f(x)$ finite. The {\em subdifferential} of $f$ at $x$, denoted by $\partial f(x)$, is the set of all vectors $\xi\in\EEE$ satisfying
\begin{equation}\label{eqn:subgrad_defn}
f(y)\geq f(x)+\langle \xi,y-x\rangle +o(\|y-x\|_2)\quad \textrm{as }y\to x.
\end{equation}
Here $o(r)$ denotes any function satisfying $o(r)/r\to 0$ as $r\to 0$.
Thus, a vector $\xi$ lies in the subdifferential $\partial f(x)$ precisely when the linear function $y\mapsto f(x) + \dotp{\xi, y- x}$  lower-bounds $f$ up to first-order around $x$. Standard results show that for a convex function $f$ the subdifferential $\partial f(x)$ reduces to the subdifferential in the sense of convex analysis, while for a differentiable function it consists only of the gradient: $\partial f(x)=\{\nabla f(x)\}$. For any closed convex functions $h\colon\YY\to\R$ and $g\colon\EEE\to\R\cup\{+\infty\}$ and $C^1$-smooth map $F\colon\EEE\to\YY$, the chain rule holds \cite[Theorem 10.6]{RW98}:
$$\partial (h\circ F+g)(x)=\nabla F(x)^*\partial h(F(x))+\partial g(x).$$ We say that a point $x$ is
\emph{stationary} for  $f$ whenever the inclusion $0 \in \partial f(x)$ holds. Equivalently, stationary points are precisely those that satisfy first-order necessary conditions for minimality: the directional derivative is nonnegative in every direction.

We say a that a random vector $X$ in $\R^d$ is {\em $\eta$-sub-gaussian} whenever $\EE \exp\left( \frac{\dotp{u, X}^2}{\eta^2} \right) \leq 2$ for all unit vectors $u\in \R^d$. The {\em sub-gaussian norm} of a real-valued random variable $X$ is defined to be $\|X\|_{\psi_2}=\inf\{t>0:\EE\exp\left( \frac{X^2}{t^2} \right) \leq 2\}$, while the {\em sub-exponential norm} is defined by $\|X\|_{\psi_1}=\inf\{t>0:\EE\exp\left( \frac{|X|}{t} \right) \leq 2\}$.

\section{Regularity conditions and algorithms (informal)}\label{sec:all_algos_reg_cond}

As outlined in Section~\ref{sec:intro}, we consider the low-rank matrix recovery problem within the framework of compositional optimization:
\begin{equation}
	\label{eqn:target_comp}
	\min_{x\in \cX} f(x):=h(F(x)),
\end{equation}
where $\cX\subset\EEE$ is a closed convex set, $h\colon\YY\to\R$ is a finite
convex function and $F\colon\EEE\to\YY$ is a $C^1$-smooth map. We depart from
previous work on low-rank matrix recovery by allowing $h$ to be nonsmooth. We primary focus on those algorithms
for~\eqref{eqn:target_comp} that converge rapidly (linearly or faster) when
initialized sufficiently close to the solution set.

Such rapid convergence guarantees rely on some regularity of the optimization problem. In the compositional setting, regularity conditions take the following appealing form.
\begin{assumption}\label{ass:bb}
	Suppose that the following properties hold for  the composite optimization problem~\eqref{eqn:target_comp} for some real numbers $\mu,\rho,L>0$.
	\begin{enumerate}
		\item {\bf (Approximation accuracy)} The convex models $f_x(y):=h(F(x)+\nabla F(x)(y-x))$ satisfy the estimate
		\[ |f(y)-f_x(y)|\leq \frac{\rho}{2}\|y-x\|^2_2\qquad \forall x,y\in \cX. \]
		\item {\bf (Sharpness)} The set of minimizers $\displaystyle \cX^*:=\argmin_{x\in\cX} f(x)$ is nonempty and we have
		\[ f(x) - \inf_{\cX} f \geq \mu \cdot  \dist\left(x,\cX^\ast \right) \qquad \forall  x \in \cX. \]
		\item {\bf (Subgradient bound)} The bound, $\sup_{\zeta\in \partial f(x)} \|\zeta\|_2 \leq L$, holds for any $x$ in the tube \[\mathcal{T}:=\left\{x\in \cX:\dist(x,\cX)\leq \frac{\mu}{\rho}\right\}. \]
	\end{enumerate}
\end{assumption}

As pointed out in the introduction, these three properties are quite intuitive: The approximation accuracy guarantees that the objective function $f$ is well approximated by the convex model $f_x$, up to a quadratic error relative to the basepoint $x$. Sharpness stipulates that the objective function should grow at least linearly as one moves away from the solution set. The subgradient bound, in turn, asserts that the subgradients of $f$ are bounded in norm by $L$ on the tube $\mathcal{T}$. In particular, this property is implied by Lipschitz continuity on $\mathcal{T}$.

\begin{lem}[Subgradient bound and Lipschitz continuity {\cite[Theorem 9.13]{RW98}}]\label{lem:Lipschitz_subgradient} {\hfill \\ }
Suppose a function $f\colon\EEE\to\R$ is $L$-Lipschitz on an open set
$U\subset\EEE$. Then the estimate $\sup_{\zeta\in \partial f(x)} \|\zeta\|_2
\leq L$ holds for all $x\in U$.
\end{lem}

 The definition of the tube $\mathcal{T}$ might look unintuitive at first. Some thought, however, shows that it arises naturally since it provably contains no extraneous stationary points of the problem. In particular, $\mathcal{T}$ will serve as a basin of attraction of numerical methods; see the forthcoming Section~\ref{sec:conv_guarant} for details.
The following general principle has recently emerged
\cite{davis2018subgradient,
duchi_ruan_PR,davis2017nonsmooth,charisopoulos2019composite}.
 Under Assumption~\ref{ass:bb}, basic numerical methods converge rapidly when
initialized within the tube $\mathcal{T}$. Let us consider three such
procedures and briefly describe their convergence properties.
Detailed convergence guarantees are deferred to Section~\ref{sec:conv_guarant}.

\bigskip
\setcounter{algocf}{1}

	\begin{algorithm}[H]
		\KwData{$x_0 \in \RR^d$}

		{\bf Step $k$:} ($k\geq 0$)\\
		$\qquad$ Choose $\zeta_k \in \partial f(x_k)$. {\bf If} $\zeta_k=0$, then exit algorithm.\\
		$\qquad$ Set $\displaystyle x_{k+1}=\proj_{\cX}\left(x_{k} - \frac{f(x_k)-\min_{\mathcal{X}} f}{\|\zeta_k\|^2_2}\zeta_k\right)$.
		\caption{Polyak Subgradient Method}
		\label{alg:polyak}
	\end{algorithm}

		\begin{algorithm}[H]
		\KwData{Real $\lambda>0$ and $q\in (0,1)$.}

		{\bf Step $k$:}  $(k\geq 0)$ \\
		\qquad Choose $\zeta_k \in \partial g(x_k)$. {\bf If} $\zeta_k=0$, then exit algorithm.

		\qquad Set stepsize $\alpha_{k} = \lambda\cdot q^{k}$.\\
		\qquad Update iterate  $x_{k+1}=\proj_{\cX}\left(x_{k} - \alpha_k \frac{\zeta_k}{\norm{\zeta_k}_2}\right)$.\\

		\caption{Subgradient method with geometrically decreasing stepsize }
		\label{alg:geometrically_step}
	\end{algorithm}

		\begin{algorithm}[H]
			\KwData{Initial point $x_0 \in \RR^{d}$, proximal parameter $\beta>0$.}

			{\bf Step $k$:}  $(k\geq 0)$ \\
			\qquad Set  $\displaystyle x_{k+1} \leftarrow \argmin_{x \in \cX} \left\{h\left(F(x_k)+\nabla F(x_k)(x-x_k)\right) + \frac{\beta}{2} \|x-x_k\|^2_2\right\}.$

			\caption{Prox-linear algorithm}
			\label{alg:prox_lin}
		\end{algorithm}

\bigskip

Algorithm~\ref{alg:polyak} is the so-called Polyak subgradient method. In each iteration $k$, the method travels in the negative direction of a subgradient $\zeta_k\in \partial f(x_k)$, followed by a nearest-point projection onto $\cX$. The step-length is governed by the current functional gap $f(x_k)-\min_{\cX} f$. In particular, one must have the value $\min_{\cX} f$ explicitly available to implement the procedure. This value is sometimes known; case in point, the minimal value of the penalty formulations~\eqref{eqn:target_problem} and~\eqref{eqn:target_problem_asymm} for low-rank recovery is zero when the linear measurements are exact.  When the minimal value $\min_{\cX} f$ is not known, one can instead use Algorithm~\ref{alg:geometrically_step}, which replaces the step-length $( f(x_k)-\min_{\cX} f ) / \|\zeta_k\|_2 $ with a preset geometrically decaying sequence. Notice that the per iteration cost of both subgradient methods is dominated by a single subgradient evaluation and a projection onto $\cX$. Under appropriate parameter settings, Assumption~\ref{ass:bb} guarantees that both methods converge at a linear rate governed by the ratio $\frac{\mu}{L}$, when initialized within $\mathcal{T}$. The prox-linear algorithm (Algorithm~\ref{alg:geometrically_step}), in contrast, converges quadratically to the optimal solution, when initialized within $\mathcal{T}$. The caveat is that each iteration of the prox-linear method requires solving a strongly convex subproblem. Note that for low-rank recovery problems \eqref{eqn:target_problem} and \eqref{eqn:target_problem_asymm}, the size of the subproblems is proportional to the size of the factors and not the size of the matrices.

In the subsequent sections, we  show that Assumption~\ref{ass:bb} (or a close variant) holds with favorable parameters $\rho,\mu, L>0$ for common low-rank matrix recovery problems.

\section{Regularity under RIP}\label{sec:ass_model}

In this section, we consider the low-rank recovery
problems~\eqref{eqn:target_problem} and \eqref{eqn:target_problem_asymm}, and
show that restricted isometry properties of the  map $\cA(\cdot)$
naturally yield well-conditioned compositional formulations.\footnote{The guarantees we develop in the symmetric setting are similar to those in the recent preprint \cite{li2018nonconvex_robust}, albeit we obtain a sharper bound on $L$; the two sets of results were obtained independently.
	The guarantees for the asymmetric setting are different and are complementary to each other: we analyze the conditioning of the basic problem formulation \eqref{eqn:target_problem_asymm}, while \cite{li2018nonconvex_robust} introduces a regularization term  $ \| X^\top X - YY^\top \|_F$ that improves the basin of attraction for the subgradient method by a factor of the condition number of $M_{\sharp}$.}
The arguments are short and elementary, and yet apply to such important problems as
phase retrieval, blind deconvolution, and covariance matrix estimation.

Setting the stage, consider a linear map $\cA\colon\R^{d_1\times d_2}\to \R^m$, an arbitrary rank $r$ matrix $M_{\sharp}\in \R^{d_1\times d_2}$, and a vector $b\in\R^m$ modeling a corrupted estimate of the measurements $\cA(M_{\sharp})$. Recall that the goal of low-rank matrix recovery is to determine $M_{\sharp}$ given $ \cA $ and $b$. By the term {\em symmetric setting}, we mean that $M_{\sharp}$ is symmetric and positive semidefinite, whereas by {\em asymmetric setting} we mean that $M_{\sharp}$ is an arbitrary rank $r$ matrix. We will treat the two settings in parallel. In the symmetric setting, we use  $X_{\sharp}$ to denote any fixed $d\times r$ matrix for which the factorization $M_{\sharp}=X_{\sharp}X_{\sharp}^\top$ holds. Similarly, in the asymmetric case, $X_{\sharp}$ and $Y_{\sharp}$ denote any fixed $d_1\times r$ and $ r\times d_2$ matrices, respectively, satisfying $M_{\sharp}=X_{\sharp}Y_{\sharp}$.

We are interested in the set of all possible factorization of $M_{\sharp}$. Consequently, we will often appeal to the following representations:
\begin{align}
\{X\in \R^{d_1\times r}: XX^\top=M_{\sharp}\}&=\{X_{\sharp}R: R\in O(r)\},\label{eqn:opt_soln_set1}\\
\{(X,Y)\in \R^{d_1\times r}\times \R^{r\times d_2}: XY=M_{\sharp}\}&=\{(X_{\sharp}A,A^{-1}Y_{\sharp}): A\in GL(r)\}.\label{eqn:opt_soln_set2}
\end{align}
Throughout, we will let $\cD^*(M_{\sharp})$ refer to the set \eqref{eqn:opt_soln_set1} in the symmetric case and to \eqref{eqn:opt_soln_set2} in the asymmetric setting.

Henceforth, fix an arbitrary norm $\opnorm{\cdot}$ on $\R^{m}$. The following property, widely used in the literature on low-rank recovery, will play a central role in this section.
\begin{assumption}[Restricted Isometry Property (RIP)]\label{assump:RIP}
	There exist constants $\lr, \ur > 0$ such that for all matrices $W \in \RR^{d_1 \times d_2}$ of rank at most $2r$ the following bound holds:
	\begin{equation*}
	\lr \|W\|_F \leq \opnorm{\cA(W)} \leq \ur \|W\|_F.
	\end{equation*}
\end{assumption}

Assumption~\ref{assump:RIP} is classical and is satisfied in various important problems with the rescaled $\ell_2$-norm $\opnorm{\cdot}=\frac{1}{\sqrt{m}}\|\cdot\|_2$ and $\ell_1$-norm $\opnorm{\cdot}=\frac{1}{m}\|\cdot\|_1$.\footnote{In the latter case, RIP also goes by the name of Restricted Uniform Boundedness (RUB) \cite{ROP_matrix_recovery_rank_one}.} In Section~\ref{sec:all_examples} we discuss a number of such examples including matrix sensing under (sub-)Gaussian design, phase retrieval, blind deconvolution, and quadratic/bilinear sensing. We summarize the RIP properties for these examples in Table~\ref{table:nonlin} and refer the reader to Section~\ref{sec:all_examples} for the precise statements.

\begin{table}[ht]
	\centering
	\begin{tabular}{c c c c }
		\hline\hline
		Problem  & Measurement $\cA(M)_i$ & $(\lr,\ur)$ & Regime \\ [0.5ex]
		\hline
		(sub-)Gaussian sensing   & $\langle P_i,M\rangle$ & $(c,C)$ & $m \succsim \frac{r d}{(1-2\pfail)^2} \ln(1+\frac{1}{1-2\pfail})$ \\
		Quadratic sensing I &  $p_i^\top M p_i$ & $(c,C\sqrt{r})$ & $m \succsim \frac{ r^2 d}{(1 - 2 \pfail)^2}\ln(1+\frac{\sqrt{r}}{1-2\pfail})$\\
		Quadratic sensing II & $p_{i}^\top M p_{i}-\tilde{p}_{i}^\top M \tilde{p}_{i}$ & $(c,C)$ & $m \succsim \frac{ r d}{(1 - 2 \pfail)^2}\ln \left( 1 + \frac{1 }{1 - 2\pfail}\right)$\\ [1ex] 
		Bilinear sensing  &  $p_i^\top M q_i$ & $(c, C)$ & $m \succsim \frac{ r d}{(1 - 2 \pfail)^2}\ln \left( 1 + \frac{1 }{1 - 2 \pfail}\right)$\\
		\hline \hline
	\end{tabular}
	\caption{
		Common problems satisfying $\ell_1/\ell_2$ RIP in Assumption~\ref{assump:RIP}. The table summarizes the $\ell_1/\ell_2$ RIP for (sub-)Gaussian sensing, quadratic sensing (e.g., phase retrieval), and bilinear sensing (e.g., blind deconvolution) under standard (sub-)Gaussian assumptions on the data generating mechanism. In all cases, we set $ \protect\opnorm{\cdot}=\frac{1}{m}\|\cdot\|_1$ and assume for simplicity $d_1=d_2=d$. The symbols $c$ and $C$ refer to numerical constants, $\pfail$ refers to the proportion of corrupted measurements, $\ir$ is a constant multiple of $\left(1-2\pfail\right)$. See Section~\ref{sec:all_examples} for details. }
	\label{table:nonlin}
\end{table}

In light of Assumption~\ref{assump:RIP}, it it natural to take the norm $\opnorm{\cdot}$ as the penalty $h(\cdot)$ in \eqref{eqn:target_problem} and \eqref{eqn:target_problem_asymm} .
Then the symmetric problem \eqref{eqn:target_problem} becomes \begin{equation}\label{eqn:comp_prob_RIP}
\min_{X\in\R^{d\times r}} f(X):=\opnorm{\cA(XX^\top)-b},
\end{equation}
while the asymmetric formulation \eqref{eqn:target_problem_asymm} becomes
\begin{equation}\label{eqn:comp_prob_RIP_asym}
\min_{X\in \R^{d_1\times r},~ Y\in \R^{r\times d_2}}~ f(X,Y):=\opnorm{\mathcal{A}(XY)-b}.
\end{equation}

Our immediate goal is to show that under Assumption~\ref{assump:RIP}, the problems~\eqref{eqn:comp_prob_RIP}  and \eqref{eqn:comp_prob_RIP_asym} are well-conditioned in the sense of~Assumption~\ref{ass:bb}.
We note that the asymmetric setting is more nuanced that its symmetric counterpart because Assumption~\ref{ass:bb} can only
be guaranteed to hold on bounded sets. Nonetheless, as we discuss in
Section~\ref{sec:conv_guarant}, a localized version of Assumption~\ref{ass:bb}
suffices to guarantee rapid local convergence of subgradient and prox-linear
methods. In particular, our analysis of the local sharpness in the asymmetric
setting is new and illuminating; it shows that the regularization technique
suggested in \cite{li2018nonconvex_robust}  is not needed at all for the prox-linear method.
This conclusion contrasts with known techniques in the smooth setting, where regularization is often used.

\subsection{Approximation and Lipschitz continuity}
We begin with the following elementary proposition, which estimates the subgradient bound $L$ and the approximation modulus $\rho$ in the symmetric setting.   In what follows, we will use the expressions
\begin{align*}
f_X(Z)&=\opnorm{\cA(XX^\top+X(Z-X)^\top+(Z-X)X^\top)-b},\\
f_{(X, Y)}(\widehat X, \widehat Y)&=\opnorm{\cA(XY+X(\widehat Y-Y)+(\widehat X-X)Y) -
	b}.
\end{align*}

\begin{proposition}[Approximation accuracy and Lipschitz continuity (symmetric)]\label{prop:approx_acc_sym} {\hfill \\ }
	Suppose Assumption~\ref{assump:RIP} holds.
	Then for all $X,Z\in\R^{d\times r}$ the following estimates hold:
	\begin{align*}
	|f(Z) - f_X(Z)| &\leq \ur \|Z-X\|^2_F,\\
	|f(X)-f(Z)|&\leq  \ur\|X+Z\|_{op}\|X-Z\|_F.
	\end{align*}
\end{proposition}

\begin{proof}
	To see the first estimate, observe that
	\begin{align}
	|f(Z) - f_X(Z)|&=\opnorm{\cA(ZZ^\top)-b}-\opnorm{\cA(XX^\top+X(Z-X)^\top+(Z-X)X^\top)-b} \notag\\
	&\leq \opnorm{\cA(ZZ^\top- XX^\top-X(Z-X)^\top-(Z-X)X^\top}  \label{eqn:rev_triangle_rip}\\
	&=  \opnorm{ \cA\big((Z-X)(Z-X)^\top \big) }\notag\\
	&\leq  \ur \left\| (Z-X)(Z-X)^\top \right\|_F \label{eqn:randorip} \\
	&\leq \ur  \|Z-X\|_F^2,\notag
	\end{align}
	where \eqref{eqn:rev_triangle_rip} follows from the reverse triangle inequality and \eqref{eqn:randorip} uses Assumption~\ref{assump:RIP}.
	Next, for any $X,Z\in \cX$ we successively compute:
	\begin{align}
	|f(X)-f(Z)|&=\left|\opnorm{\cA(XX^\top)-b} -\opnorm{\cA(ZZ^\top)-b}\right|\notag\\
	&\leq \opnorm[\big]{\cA(XX^\top-ZZ^\top)}\label{eqn:rev_tri2}\\
	&\leq \ur\|XX^\top-ZZ^\top\|_F\label{eqn:useassrip}\\
	&=\frac{\ur}{2} \|(X + Z)(X-Z)^\top + (X-Z)(X+Z)^\top\|_F\notag\\
	&\leq \ur\|(X+Z)(X-Z)\|_F\notag \\
	&\leq \ur\|X+Z\|_{op}\|X-Z\|_F,\notag
	\end{align}
	where \eqref{eqn:rev_tri2} follows from the reverse triangle inequality and \eqref{eqn:useassrip} uses  Assumption~\ref{assump:RIP}.
	The proof is complete.
\end{proof}

The estimates of $L$ and $\rho$ in the asymmetric setting are completely analogous; we  record them in the following proposition.

\begin{proposition}[Approximation accuracy and Lipschitz continuity (asymmetric)]\label{prop:approx_acc_asym}  {\hfill \\
	}
	Suppose Assumption~\ref{assump:RIP} holds. Then for all $X,\widehat X \in
	\RR^{d_1\times r}$ and $Y,\widehat Y \in \RR^{r\times d_2}$ the following estimates hold:
	\begin{align*}
	|f(\widehat X, \widehat Y) - f_{(X, Y)}(\widehat X, \widehat Y)| &\leq \frac{\ur}{2}\cdot
	\|(X,Y) - (\widehat X, \widehat Y)\|^2_F,\\
	|f(X,Y)-f(\widehat X,\widehat Y)|&\leq \tfrac{\ur\max\{\|X+\widehat{X}\|_{\rm op},\|Y+\widehat{Y}\|_{\rm op}\}}{\sqrt{2}}\cdot \|(X,Y)-(\widehat X,\widehat
	Y)\|_F.
	\end{align*}
\end{proposition}

\begin{proof}
	To see the first estimate, observe that
	\begin{align*}
	|f(\widehat X, \widehat Y) - f_{(X,Y)}(\widehat X, \widehat Y)|&=\left|\opnorm{\cA(\widehat X\widehat Y) - b } -
	\opnorm{\cA(XY+X(\widehat Y-Y)+(\widehat X-X)Y) -
		b}\right| \notag \\
	&\leq \opnorm{ \cA(\widehat X\widehat Y -XY-X(\widehat Y-Y)-(\widehat X-X)Y)}  \notag\\
	&= \opnorm{ \cA\big((X-\widehat X)(Y-\widehat Y) \big)
	}\notag\\
	&\leq  \ur \left\| (X-\widehat X)(Y-\widehat Y) \right\|_F  \\
	&\leq \frac{\ur}{2} \left(\|X-\widehat X\|_F^2 + \|Y - \widehat Y\|_F^2\right),
	\end{align*}
	where the last estimate follows from Young's inequality $2ab \leq a^2+b^2.$
	Next, we successively compute:
	\begin{align*}
	|f(X,Y)-f(\widehat X,\widehat Y)|
	&\leq \opnorm{\cA(XY-\widehat X\widehat Y)}
	\leq \ur\|XY-\widehat X\widehat Y\|_F\\
	&=\frac{\ur}{2} \|(X + \widehat X)(Y-\widehat Y)^\top + (X - \widehat X)(Y+\widehat Y)^\top\|_F\notag\\
	&\leq \frac{\ur\max\{\|X+\widehat{X}\|_{\rm op},\|Y+\widehat{Y}\|_{\rm
			op}\}}{2} (\|Y-\widehat Y\|_F+\|X-\widehat X\|_F).
	\end{align*}
	The result follows by noting that $ a+b \le \sqrt{2(a^2 + b^2)}$ for all $a,b\in \RR $.

\end{proof}

\subsection{Sharpness}
We next move on to estimates of the sharpness constant $\mu$. We first deal with the noiseless setting $b=\cA(M_{\sharp})$ in Section~\ref{sec:sharp_noiseless}, and then move on to the general case when the measurements are corrupted by outliers in Section~\ref{subsec:sharpwithnoise}.

\subsubsection{Sharpness in the noiseless regime}\label{sec:sharp_noiseless}
We begin with with the symmetric setting in the noiseless case  $b=\cA(M_{\sharp})$. By Assumption~\ref{assump:RIP}, we have the estimate
\begin{equation}\label{eqn:lower-bound_est}
f(X)=\opnorm{\cA(XX^\top)-b}=\opnorm{\cA(XX^\top-X_{\sharp}X_{\sharp}^{\top})}\geq \lr\|XX^\top-X_{\sharp}X_{\sharp}^{\top}\|_F.
\end{equation}
It follows that the set of minimizers $\argmin_{X\in\R^{d\times r}}f(X)$
coincides with the set of minimizers of the function $X\mapsto \|XX^\top-X_\sharp X^\top_\sharp\|_F$, namely
$$  \cD^*(M_{\sharp}) := \{X_{\sharp}R: R\in O(r)\}.$$
Thus to argue sharpness of $f$ it suffices to estimate the sharpness constant of the function $X\mapsto \|XX^\top-X_{\sharp} X^\top_{\sharp}\|_F$. Fortunately, this calculation was already done in \cite[Lemma 5.4]{proc_flow}.

\begin{proposition}[{\hspace{1sp}\cite[Lemma 5.4]{proc_flow}}]\label{prop:l2sharpness}
	For any matrices $X, Z\in \R^{d\times r}$, we have the  bound
	\begin{align*}
	\|X X^\top - Z Z^\top\|_F  \geq \sqrt{2(\sqrt{2}-1)}\sigma_r(Z)\cdot\min_{R\in O(r)} \|X-ZR\|_F.
	\end{align*}
	Consequently if Assumption~\ref{assump:RIP} holds in the noiseless setting $b=\cA(M_{\sharp})$, then the bound holds:
	\begin{align*}
	f(X)  \geq  \lr\sqrt{2(\sqrt{2}-1)\sigma_r(M_{\sharp})}\cdot\dist(X,\cD^*(M_\sharp)) \qquad \text{for all $X\in \R^{d\times r}$}.
	\end{align*}
\end{proposition}

We next consider the asymmetric case.
By exactly the same reasoning as before, the set of minimizers of $f(X,Y)$
coincides with the set of minimizers of the function
$(X,Y)\mapsto \|XY-X_{\sharp} Y_{\sharp}\|_F$, namely
$$  \cD^*(M_{\sharp}) := \{(X_{\sharp}A,A^{-1}Y_{\sharp}): A\in GL(r)\}.$$
Thus to argue sharpness of $f$ it suffices to estimate the sharpness constant of the function $(X,Y)\mapsto \|XY-X_{\sharp} Y_{\sharp}\|_F$.
Such a sharpness guarantee in the rank one case was recently shown in \cite[Proposition 4.2]{charisopoulos2019composite}.

\begin{proposition}[\hspace{1sp}{\cite[Proposition 4.2]{charisopoulos2019composite}}]
	\label{prop:l2sharpnessrank1as}
	Fix a rank $1$ matrix $M_{\sharp}\in\R^{d_1\times d_2}$ and a constant $\nu\geq 1$. Then for any $x\in \R^{d_1}$ and $w\in \R^{d_{2}}$ satisfying
	$$\|w\|_2,\|x\|_2\leq \nu\sqrt{\sigma_1(M_{\sharp})},$$
	the following estimate holds:
	\begin{align*}
	\|xw^\top - M_{\sharp}\|_F  \geq \frac{\sqrt{\sigma_1(M_{\sharp})}}{2\sqrt{2}(\nu+1)} \cdot
	\dist \big((x,w), \cD^*(M_{\sharp})\big).
	\end{align*}
\end{proposition}

Notice that in contrast to the symmetric setting, the sharpness estimate is only valid on bounded sets. Indeed, this is unavoidable even in the setting $d_1=d_2=2$. To see this, define $M_{\sharp}=e_2e_2^\top$ and for any $\alpha>0$ set $x=\alpha e_1$ and $w=\tfrac{1}{\alpha} e_1$. It is routine to compute
$$\frac{\|xw^\top-M_{\sharp}\|_F}{\dist((x,w),\cD^*(M_{\sharp}))} =\sqrt{\frac{2}{2+\alpha^2+\frac{1}{\alpha^2}}}.$$
Therefore letting $\alpha$ tend to zero (or infinity) the quotient tends to zero.

The following corollary is a higher rank extension of Proposition~\ref{prop:l2sharpnessrank1as}.

\begin{thm}[Sharpness (asymmetric and noiseless)] Fix a constant $\nu>0$ and define $X_{\sharp}:=U\sqrt{\Lambda}$ and $Y_{\sharp}=\sqrt{\Lambda}V^\top$, where $M_{\sharp}=U\Lambda V^\top$ is any compact singular value decomposition of $M_{\sharp}$.
	Then for all $X\in\R^{d_1\times r}$ and $Y\in \R^{r\times d_2}$ satisfying
	\begin{equation}\label{eqn:local_neighborhood}
	\begin{aligned}
	\max\{\|X-X_{\sharp}\|_F,\|Y-Y_{\sharp}\|_F\}&\leq \nu\sqrt{\sigma_r(M_{\sharp})}\\
	\dist((X,Y),\cD^*(M_{\sharp}))&\leq \frac{\sqrt{\sigma_r(M_{\sharp})}}{1+2(1+\sqrt{2})\nu}
	\end{aligned},
	\end{equation}
	the estimate holds:
	$$\|XY-M_{\sharp}\|_F\geq \frac{\sqrt{\sigma_r(M_{\sharp})}}{2+4(1+\sqrt{2})\nu}\cdot\dist((X,Y),\cD^*(M_{\sharp})).$$
\end{thm}
\begin{proof}
	Define $\delta:=\frac{1}{1+2(1+\sqrt{2})\nu}$ and consider a pair of matrices $X$ and $Y$ satisfying \eqref{eqn:local_neighborhood}. Let $A \in GL(r)$ be an invertible matrix satisfying
	\begin{equation}\label{eqn:defineA}
	A \in \argmin_{A\in GL(r)} \left\{\|X - X_{\sharp} A\|_F^2 + \|Y -  A^{-1}Y_{\sharp}\|_F^2 \right\}.
	\end{equation}
	As a first step, we successively compute
	\begin{equation}\label{eqn:first_ineqs_est}
	\begin{aligned}
	&\|XY - X_{\sharp}Y_{\sharp}\|_F \\
	&= \| (X -  X_{\sharp}A) (A^{-1} Y_{\sharp}) +  X_{\sharp}A(Y - A^{-1}Y_{\sharp}) + (X - X_{\sharp}A)(Y - A^{-1} Y_{\sharp})\|_F\\
	&\geq \| (X -  X_{\sharp}A) (A^{-1} Y_{\sharp}) +  X_{\sharp}A(Y - A^{-1}Y_{\sharp})\|_F -\| (X - X_{\sharp}A)(Y - A^{-1} Y_{\sharp})\|_F\\
	&\geq \| (X -  X_{\sharp}A) (A^{-1} Y_{\sharp}) +  X_{\sharp}A(Y - A^{-1}Y_{\sharp})\|_F -\|X - X_{\sharp}A\|_F\cdot\|Y - A^{-1} Y_{\sharp}\|_F\\
	&\geq \| (X -  X_{\sharp}A) (A^{-1} Y_{\sharp}) +  X_{\sharp}A(Y - A^{-1}Y_{\sharp})\|_F -\frac{1}{2}(\|X - X_{\sharp}A\|_F^2+\|Y - A^{-1} Y_{\sharp}\|_F^2)\\
	&= \| (X -  X_{\sharp}A) (A^{-1} Y_{\sharp}) +  X_{\sharp}A(Y - A^{-1}Y_{\sharp})\|_F -\frac{1}{2}\dist^2((X,Y),\cD^*(M_{\sharp})) \\
	&\geq  \| (X -  X_{\sharp}A) (A^{-1} Y_{\sharp}) +  X_{\sharp}A(Y - A^{-1}Y_{\sharp})\|_F -\frac{\delta \sqrt{\sigma_r(M_{\sharp})}}{2}\cdot \dist((X,Y),\cD^*(M_{\sharp})).
	\end{aligned}
	\end{equation}
	We next aim to lower bound the first term on the right. To this end,  observe
	\begin{equation}\label{eqn:gettingthembounds}
	\begin{aligned}
	\| (X -  X_{\sharp}A) (A^{-1} Y_{\sharp}) &+  X_{\sharp}A(Y - A^{-1}Y_{\sharp})\|_F^2\\
	&= \|(X -  X_{\sharp}A) (A^{-1} Y_{\sharp})\|_F^2+ \|X_{\sharp}A(Y - A^{-1}Y_{\sharp})\|_F^2\\
	&\quad+2\tr((X -  X_{\sharp}A) (A^{-1} Y_{\sharp})(Y - A^{-1}Y_{\sharp})^\top(X_{\sharp}A)^\top).
	\end{aligned}
	\end{equation}
	We claim that the cross-term is non-negative. To see this, observe that first order optimality conditions in \eqref{eqn:defineA} directly imply that $A$ satisfies the equality
	$$
	A^\top X_{\sharp}^\top(X - X_{\sharp}A ) = (Y - A^{-1}Y_{\sharp})Y_{\sharp}^\top A^{-\top}.
	$$
	Thus we obtain
	\begin{align*}
	\trace( (X -  X_{\sharp}A) (A^{-1} Y_{\sharp}) (Y - A^{-1}Y_{\sharp})^\top(X_{\sharp}A)^\top)
	&=\trace( A^\top X_{\sharp}^\top(X -  X_{\sharp}A) (A^{-1} Y_{\sharp}) (Y - A^{-1}Y_{\sharp})^\top)\\
	&=\trace( (Y - A^{-1}Y_{\sharp} )Y_{\sharp}^\top A^{-T} (A^{-1} Y_{\sharp}) (Y - A^{-1}Y_{\sharp})^\top)\\
	&= \|(A^{-1}Y_{\sharp})(Y - A^{-1}Y_{\sharp} )\|_F^2.
	\end{align*}
	Therefore, returning to \eqref{eqn:gettingthembounds} we conclude that
	\begin{equation}\label{eqn:weird_ineqs_key}
	\begin{aligned}
	& \| (X -  X_{\sharp}A) (A^{-1} Y_{\sharp}) +  X_{\sharp}A(Y - A^{-1}Y_{\sharp})\|_F \\
	&\geq \sqrt{\|(X -  X_{\sharp}A) (A^{-1} Y_{\sharp})\|_F^2 +  \|X_{\sharp}A(Y - A^{-1}Y_{\sharp})\|_F^2}\\
	&\geq   \sqrt{\sigma_r(M_{\sharp})}\cdot\min\{\sigma_r(A^{-1}), \sigma_r(A)\}\cdot\dist((X,Y),\cD^*(M_{\sharp})).
	\end{aligned}
	\end{equation}
	Combining \eqref{eqn:first_ineqs_est} and \eqref{eqn:weird_ineqs_key}, we obtain
	\begin{equation}\label{eqn:prefinal_est}
	\|XY - M_{\sharp}\|_F\geq  \sqrt{\sigma_r(M_{\sharp})}\cdot\left(\min\{\sigma_r(A^{-1}), \sigma_r(A)\}-\frac{\delta}{2}\right)\cdot\dist((X,Y),\cD^*(M_{\sharp}))
	\end{equation}
	Finally, we estimate $\min\{\sigma_r(A^{-1}), \sigma_r(A)\}$. To this end, first note that
	\begin{equation}\label{eqn:distance_upper}
	\begin{aligned}
	\|X_{\sharp}-X_{\sharp}A\|_F+\|Y_{\sharp}-A^{-1}Y_{\sharp}\|_F&\leq \|X_{\sharp}-X\|_F+\|Y_{\sharp}-Y\|_F+\sqrt{2}\cdot\dist((X,Y),\cD^*(M_{\sharp}))\\
	&\leq 2\nu\sqrt{\sigma_r(M_{\sharp})}\cdot(1+\sqrt{2}).
	\end{aligned}
	\end{equation}
	We now aim to lower bound the left-hand-side in terms of $\min\{\sigma_r(A^{-1}), \sigma_r(A)\}$.
	Observe $$\|X_{\sharp}-X_{\sharp}A\|_F\geq\|X_{\sharp}-X_{\sharp}A\|_{\rm op}\geq   \sqrt{\sigma_r(M_{\sharp})}\cdot\|I-A\|_{\rm op}\geq \sqrt{\sigma_r(M_{\sharp})} \cdot(\sigma_{1}(A) -1).$$
	Similarly, we have
	\begin{align*}
	\|Y_{\sharp}-A^{-1}Y_{\sharp}\|_F\geq \|Y_{\sharp}-A^{-1}Y_{\sharp}\|_{\rm op}&\geq \sqrt{\sigma_r(M_{\sharp})}\cdot \|I-A^{-1}\|_{\rm op}\geq \sqrt{\sigma_r(M_{\sharp})}\cdot (\sigma_{1}(A^{-1}) -1).
	\end{align*}
	Hence using \eqref{eqn:distance_upper}, we obtain the estimate
	$$\min\{\sigma_r(A^{-1}),\sigma_r(A)\}\geq \left(1+2\nu\cdot(1+\sqrt{2})\right)^{-1}=\delta.$$
	Using this estimate in \eqref{eqn:prefinal_est} completes the proof.
\end{proof}

\subsubsection{Sharpness in presence of outliers}\label{subsec:sharpwithnoise}

The most important example of the norm $\opnorm{\cdot}$ for us is the scaled $\ell_1$-norm $\opnorm{\cdot}=\frac{1}{m}\|\cdot\|_1$. Indeed, all the examples in the forthcoming Section~\ref{sec:all_examples} will satisfy RIP relative to this norm. In this section, we will show that the $\ell_1$-norm has an added advantage. Under reasonable RIP-type conditions, sharpness will hold even if up to a half of the measurements are grossly corrupted.

Henceforth, for any set $\mathcal{I}$, define the restricted map $\cA_{\cI}:=\left(\cA(X)\right)_{i\in \cI}$. We interpret the set $\cI$ as corresponding to (arbitrarily) outlying measurements, while its complement corresponds to exact measurements.  Motivated by the work \cite{duchi_ruan_PR} on robust phase retrieval, we make the following assumption.

\begin{assumption}[$\cI$-outlier bounds]\label{assump:outlier}
	There exists a set $\cI \subset\{1, \ldots, m\}$ and a constant $\ir > 0$ such that the following hold.
	\begin{enumerate}[label = $\mathrm{(C\arabic*)}$]
		\item Equality holds $b_i=\cA(M_{\sharp})_i$ for all $i\notin \cI$.
		\item \label{item:assump:rip_outliers}  For all matrices $W$ of rank at most $2r$, we have
		\begin{equation}\label{eq:Upper_Bound}
		\ir \|W\|_F \leq  \frac{1}{m}\|\cA_{\mathcal{I}^c}(W)\|_1 -  \frac{1}{m}\|\cA_{\cI}(W)\|_1.
		\end{equation}
	\end{enumerate}
\end{assumption}
The assumption is simple to interpret. To elucidate the bound \eqref{eq:Upper_Bound}, let us suppose that the restricted maps $\cA_{\cI}$ and $\cA_{\cI^c}$ satisfy Assumption~\ref{assump:RIP} (RIP) with constants $\hat \lr$, $\hat \ur$ and $ \lr$, $ \ur$, respectively. Then for any rank $2r$ matrix $X$ we immediately deduce the estimate
$$\frac{1}{m}\|\cA_{\mathcal{I}^c}(W)\|_1 -  \frac{1}{m}\|\cA_{\cI}(W)\|_1\geq \left((1-\pfail)\lr-\pfail \hat \ur\right)\|W\|_F,$$
where $\pfail=\frac{|\cI|}{m}$ denotes the corruption frequency. In particular, the right-hand side is positive as long as the corruption frequency is below the threshold $\pfail <\frac{\lr}{\lr+\hat \ur}$.

Combining Assumption~\ref{assump:outlier}  with
Proposition~\ref{prop:l2sharpness} quickly yields sharpness of the objective
even in the noisy setting.

\begin{proposition}[Sharpness with outliers (symmetric)]\label{prop:noisy_sharp}
	Suppose that Assumption \ref{assump:outlier} holds.
	Then
	\begin{align*}
	f(X) - f(X_{\sharp})  \geq  \ir\left(\sqrt{2(\sqrt{2}-1)}\sigma_r(X_{\sharp})\right) \dist \big(X, \cD^\ast(M_{\sharp})\big) \qquad \text{for all $X\in\R^{d\times r}$}.
	\end{align*}
\end{proposition}
\begin{proof}
	Defining $\Delta := \cA(X_{\sharp}X_{\sharp}^\top) - b$, we have the following bound:
	\begin{align*}
	m\cdot(f(X) - f(X_{\sharp}))  &=
	\|\cA\left(XX^\top - X_{\sharp} X_{\sharp}^\top \right)  + \Delta \|_1
	-
	\|\Delta\|_1  \\
	&=  \|\cA_{\cI^c}( XX^\top - X_{\sharp} X_{\sharp}^\top)\|_1 +
	\sum_{i \in \mathcal{I}} \left(|\left(\cA(XX^\top - X_{\sharp} X_{\sharp}^\top)\right)_i
	+\Delta_i|
	- |\Delta_i|\right)\\
	&\geq  \|\cA_{\cI^c}(XX^\top - X_{\sharp}
	X_{\sharp}^\top)\|_1 - \|\cA_{\cI}(XX^\top - X_{\sharp}
	X_{\sharp}^\top)\|_1\\
	&\geq \ir m\|XX^\top - X_{\sharp}X_{\sharp}^\top\|_F \geq \ir m\left(\sqrt{2(\sqrt{2}-1)}\sigma_r(X_{\sharp})\right)\dist \big(X, \cD^*(M_{\sharp})\big),
	\end{align*}
	where the first inequality follows by the reverse triangle inequality, the second inequality follows by Assumption~\ref{item:assump:rip_outliers}, and the final inequality follows from Proposition~\ref{prop:l2sharpness}. The proof is complete.
\end{proof}

The argument in the asymmetric setting is completely analogous.
\begin{proposition}[Sharpness with outliers (asymmetric)]
	Suppose that Assumption  \ref{assump:outlier} holds.
	Fix a constant $\nu> 0$ and define $X_{\sharp}:=U\sqrt{\Lambda}$ and $Y_{\sharp}=\sqrt{\Lambda}V^\top$, where $M_{\sharp}=U\Lambda V^\top$ is any compact singular value decomposition of $M_{\sharp}$. Then for all $X\in\R^{d_1\times r}$ and $Y\in \R^{r\times d_2}$ satisfying
	\begin{align*}
	\max\{\|X-X_{\sharp}\|_F,\|Y-Y_{\sharp}\|_F\}&\leq \nu\sqrt{\sigma_r(M_{\sharp})}\\
	\dist((X,Y),\cD^*(M_{\sharp}))&\leq \frac{\sqrt{\sigma_r(M_{\sharp})}}{1+2(1+\sqrt{2})\nu}
	\end{align*}
	The estimate holds:
	$$f(X,Y) - f(X_{\sharp},Y_{\sharp})\geq \frac{\ir \sqrt{\sigma_r(M_{\sharp})}}{2+4(1+\sqrt{2})\nu}\cdot\dist((X,Y),\cD^*(M_{\sharp})).$$
\end{proposition}

\section{General convergence guarantees for subgradient \& prox-linear methods}\label{sec:conv_guarant}

In this section, we formally develop convergence guarantees for
Algorithms~\ref{alg:polyak}, \ref{alg:geometrically_step}, and
\ref{alg:prox_lin} under Assumption~\ref{ass:bb}, and deduce performance guarantees in the RIP setting.
To this end, it will be useful to first consider a broader class than the compositional problems \eqref{eqn:target_comp}. We say that  a function $f\colon\EEE \rightarrow \RR\cup\{+\infty\}$ is {\em $\rho$-weakly convex}\footnote{Weakly convex functions also go by other names such as lower-$C^2$, uniformly prox-regularity, paraconvex, and semiconvex. We refer the reader to the seminal works on the topic \cite{fav_C2,prox_reg,Nurminskii1973,paraconvex,semiconcave}.} if the perturbed function $x \mapsto f(x) + \frac{\rho}{2}\|x\|^2_2$ is convex. In particular, a composite function $f=h\circ F$ satisfying the approximation guarantee
$$|f_{x}(y)-f(y)|\leq \frac{\rho}{2}\|y-x\|_2^2 \qquad \forall x,y$$
is automatically $\rho$-weakly convex \cite[Lemma 4.2]{eff_paquette}.
Subgradients of weakly convex functions are very well-behaved. Indeed, notice that in general the little-o term in the expression  \eqref{eqn:subgrad_defn} may depend on the basepoint $x$, and may therefore be nonuniform. The subgradients of weakly convex functions, on the other hand,  automatically satisfy a uniform type of lower-approximation property. Indeed, a lower-semicontinuous function $f$ is $\rho$-weakly convex if and only if it satisfies:
\[f(y) \geq f(x) + \dotp{\xi, y-x} - \frac{\rho}{2} \|y-x\|^2_2 \qquad \forall x,y \in \EEE, \xi\in\partial f(x).\]

Setting the stage, we introduce the following assumption.
\begin{assumption}\label{ass:aa}
	Consider the optimization problem,
	\begin{equation}\label{eqn:target_weak_conv}
	\min_{x\in \cX}~ f(x).
	\end{equation}
	Suppose that the following properties hold for some real numbers $\mu,\rho>0$.
	\begin{enumerate}
		\item {\bf (Weak convexity)} The set $\cX$ is closed and convex, while the function $f\colon\EEE\to\R$ is $\rho$-weakly convex.
		\item {\bf (Sharpness)} The set of minimizers $\displaystyle \cX^*:=\argmin_{x\in\cX} f(x)$ is nonempty and the following inequality holds:
		\[f(x) - \inf_{\cX} f \geq \mu \cdot  \dist\left(x,\cX^\ast \right) \qquad \forall x \in \cX. \]
	\end{enumerate}
\end{assumption}

In particular, notice that Assumption~\ref{ass:bb} implies Assumption~\ref{ass:aa}. Taken together, weak convexity and sharpness provide an appealing framework for deriving local rapid convergence guarantees for numerical methods. In this section, we specifically focus on two such procedures: the subgradient and prox-linear algorithms. We aim to estimate both the radius of rapid converge around the solution set and the rate of convergence. Note that both of the algorithms, when initialized at a stationary point could stay there for all subsequent iterations. Since we are interested in finding global minima, we therefore estimate the neighborhood of the solution set that has no extraneous stationary points.
This is the content of the following simple lemma.

\begin{lem}[{\hspace{1sp}\cite[Lemma 3.1]{davis2018subgradient}}]\label{lem:no_extr_stat}
	Suppose that Assumption~\ref{ass:aa} holds.
	Then the problem \eqref{eqn:target_weak_conv} has no stationary points $x$ satisfying
	$$0<\dist(x;\cX^*)<\frac{2\mu}{\rho}.$$
\end{lem}

It is worthwhile to note that the estimate $\frac{2\mu}{\rho}$ of the radius in Lemma~\ref{lem:no_extr_stat} is tight \cite[Section 3]{charisopoulos2019composite}.
Hence, let us define for any $\gamma>0$ the tube
\begin{align}\label{eq:tube_region}
\cT_{\gamma} := \left\{ z \in \cX \colon \dist(z, \cX^\ast) \leq \gamma \cdot \frac{\mu}{\rho}\right\}.
\end{align}
Thus we would like to search for algorithms whose basin of attraction is a tube $\cT_{\gamma}$ for some numerical constant $\gamma>0$. Such a basin of attraction is in essence optimal.

The rate of convergence of the subgradient methods (Algorithms~\ref{alg:polyak} and \ref{alg:geometrically_step})  relies on the subgradient bound and the condition measure:
\begin{equation*}
L:=\sup\{\|\zeta\|_2:\zeta\in \partial f(x),x\in \cT_1\}\qquad \textrm{and} \qquad \tau:=\frac{\mu}{L}.
\end{equation*}
A straightforward argument \cite[Lemma 3.2]{davis2018subgradient} shows
$\tau\in [0,1]$. The following theorem appears as \cite[Theorem
4.1]{davis2018subgradient}, while its application to phase retrieval was
investigated in \cite{davis2017nonsmooth}.

\begin{thm}[Polyak subgradient method]\label{thm:qlinear}
	Suppose that Assumption~\ref{ass:aa} holds and fix a real number $\gamma \in (0,1)$.
	Then  Algorithm~\ref{alg:polyak} initialized at any point $x_0\in \mathcal{T}_{\gamma}$ produces iterates that converge $Q$-linearly to $\cX^*$, that is
	\begin{equation*} 
	\dist^2(x_{k+1},\cX^*) \leq \left(1-(1-\gamma) \tau^2\right)\dist^2(x_{k},\cX^*)\qquad \forall k\geq 0.
	\end{equation*}
\end{thm}
The following theorem appears as \cite[Theorem 6.1]{davis2018subgradient}. The convex version of the result dates back to Goffin \cite{goffin}.

\begin{thm}[Geometrically decaying subgradient method] \label{thm:geometric} Suppose that Assumption~\ref{ass:aa} holds, fix a real number  $\gamma \in (0,1)$,  and suppose  $\tau  \le \sqrt{ \frac{1}{2-\gamma} }$. Set $\lambda:=\frac{\gamma \mu^2}{\rho L} \textrm{ and } q:=\sqrt{1-(1-\gamma) \tau^2} $ in Algorithm~\ref{alg:geometrically_step}.
	Then the iterates $x_k$ generated by Algorithm~\ref{alg:geometrically_step}, initialized at any point $x_0 \in \mathcal{T}_{\gamma}$, satisfy:
	\begin{equation*} 
	\dist^2(x_k;\cX^*) \leq \frac{\gamma^2 \mu^2}{\rho^2}
	\left(1-(1-\gamma)\tau^2\right)^{k}\qquad \forall k\geq 0.
	\end{equation*}
\end{thm}

Let us now specialize to the composite setting under Assumption~\ref{ass:bb}.
Since Assumption~\ref{ass:bb} implies Assumption~\ref{ass:aa}, both subgradient
Algorithms~\ref{alg:polyak} and \ref{alg:geometrically_step} will enjoy a
linear rate of convergence when initialized sufficiently close the solution
set.  The following theorem, on the other hand, shows that the prox-linear
method will enjoy a quadratic rate of convergence (at the price of a higher
per-iteration cost). Guarantees of this type have appeared, for example, in
\cite{duchi_ruan_PR,burke_gauss,drusvyatskiy2018error}.

\begin{thm}[Prox-linear algorithm]
	Suppose Assumption~\ref{ass:bb} holds. Choose any $\beta \geq \rho$ in Algorithm~\ref{alg:prox_lin} and set $\gamma:=\rho/\beta$. Then Algorithm~\ref{alg:prox_lin} initialized at any point $x_0 \in \cT_{\gamma}$ converges quadratically:
	$$\dist(x_{k+1},\cX^*)\leq  \tfrac{\beta}{\mu}\cdot\dist^2(x_{k},\cX^*)\qquad \forall k\geq 0.$$
\end{thm}

We now apply the results above to the low-rank matrix factorization problem
under RIP, whose regularity properties were verified in
Section~\ref{sec:ass_model}. In particular, we have the following efficiency
guarantees of the subgradient and prox-linear methods applied to this problem.

\begin{corollary}[Convergence guarantees  under RIP (symmetric)]\label{cor:generic_conv_iso_sym}
	Suppose Assumptions~\ref{assump:RIP} and \ref{assump:outlier} are valid with $\opnorm{\cdot}=\frac{1}{m}\|\cdot\|_1$ and
	consider the optimization problem
	$$\min_{X\in \R^{d\times r}}~ f(X)=\frac{1}{m}\|\cA(XX^\top) - b\|_1.$$
	Choose any matrix $X_0$  satisfying
	$$\frac{\dist(X_0,\cD^\ast(M_\sharp))}{\sqrt{\sigma_r(M_{\sharp})}}\leq 0.2\cdot\frac{\ir}{\ur}.$$
	Define the condition number $\chi:=\sigma_1(M_{\sharp})/\sigma_r(M_{\sharp})$.			Then the following are true.
	\begin{enumerate}
		\item {\bf (Polyak subgradient)} Algorithm~\ref{alg:polyak} initialized at $X_0$  produces iterates that
		converge linearly to $\cD^\ast(M_\sharp)$, that is
		\begin{equation*}
		\frac{\dist^2(X_k,\cD^\ast(M_\sharp))}{\sigma_r(M_{\sharp})}\leq \left(1-\frac{0.2}{1+\frac{4\ur^2\chi}{\ir^2}}\right)^{k}\cdot \frac{\ir^2}{100\ur^2}\qquad \forall k\geq 0.
		\end{equation*}
		\item {\bf (geometric subgradient)}
		Algorithm~\ref{alg:geometrically_step} with
		$\lambda=\frac{0.81\ir^2\sqrt{\sigma_r(M_{\sharp})}}{2\ur (\ir+2\ur\sqrt{\chi})}$, $q=\sqrt{1-\frac{0.2}{1+4\ur^2\chi/\ir^2}}$
		and initialized at $X_0$ converges linearly:
		\begin{equation*}
		\frac{\dist^2(X_k,\cD^\ast(M_\sharp))}{\sigma_r(M_{\sharp})}\leq \left(1-\frac{0.2}{1+\frac{4\ur^2\chi}{\ir^2}}\right)^{k}\cdot \frac{\ir^2}{100\ur^2}\qquad \forall k\geq 0.
		\end{equation*}
		\item {\bf (prox-linear)}  Algorithm~\ref{alg:prox_lin} with $\beta = \rho$ and initialized at $X_0$ converges quadratically:
		$$\frac{\dist(X_k,\cD^\ast(M_\sharp)))}{\sqrt{\sigma_r(M_{\sharp})}}\leq 2^{-2^{k}}\cdot \frac{0.45\ir}{\ur}\qquad \forall k\geq 0.$$
	\end{enumerate}
\end{corollary}

\addtocontents{toc}{\protect\setcounter{tocdepth}{1}}
\subsection{Guarantees under local regularity}\label{sec:asymetric_algorithm}
\addtocontents{toc}{\protect\setcounter{tocdepth}{2}}

As explained in Section~\ref{sec:ass_model}, Assumptions~\ref{ass:bb} and \ref{ass:aa} are reasonable in the symmetric setting under RIP.  The asymmetric setting is more nuanced. Indeed, the solution set is unbounded, while uniform bounds on the sharpness and subgradient norms are only valid on bounded sets. 
One remedy, discussed in \cite{li2018nonconvex_robust}, is to modify the optimization formulation by introducing a form of regularization:
$$\min_{X,Y}~ \opnorm{\cA(XY)-y}+\lambda\|X^\top X-YY^\top\|_F.$$
In this section, we take a different approach that requires no modification to
the optimization problem nor the algorithms. The key idea is to show that if
the problem is well-conditioned only on a neighborhood of a particular
solution, then the iterates will remain in the neighborhood provided the initial point is sufficiently close to the solution.  In fact, we will see that the iterates themselves must
converge.
The proofs of the results in this section (Theorems~\ref{thm:polyak_local}, \ref{thm:geo_desc_loc}, and \ref{thm:prox_lin_loc}) are deferred to Appendix~\ref{sec:app_convloc}.

We begin with the following localized version of Assumption~\ref{ass:aa}.

\begin{assumption}\label{ass:target_weak_conv_loc}
	Consider the optimization problem,
	\begin{equation}\label{eqn:target_weak_conv_loc}
	\min_{x\in \cX}~ f(x).
	\end{equation}
	Fix an arbitrary point $\bar x\in \cX^*$ and suppose that the following properties hold for some real numbers $\epsilon,\mu,\rho>0$.
	\begin{enumerate}
		\item {\bf (Local weak convexity)} The set $\cX$ is closed and convex, and the bound  holds:
		$$f(y)\geq f(x)+\langle \zeta,y-x\rangle -\frac{\rho}{2}\|y-x\|^2_2\qquad \forall x,y \in \cX\cap B_{\epsilon}(\bar x), \zeta\in \partial f(x).$$
		\item {\bf (Local sharpness)} The inequality holds:
		\[f(x) - \inf_{\cX} f \geq \mu \cdot  \dist\left(x,\cX^\ast \right) \qquad \forall x\in \cX \cap B_{\epsilon}(\bar x). \]
	\end{enumerate}
\end{assumption}

The following two theorems establish convergence guarantees of the two subgradient methods under Assumption~\ref{ass:target_weak_conv_loc}.	Abusing notation slightly, we define the local quantities:
$$L:=\sup_{\zeta\in \partial f(x)}\{\|\zeta\|_2: x\in \cX\cap B_{\epsilon}(\bar x)\}\quad \textrm{and}\quad\tau:=\frac{\mu}{L}.$$

\begin{thm}[Polyak subgradient method (local regularity)]\label{thm:polyak_local}
	Suppose Assumption~\ref{ass:target_weak_conv_loc} holds and
	fix  an arbitrary point $x_0\in B_{\epsilon/4}(\bar x)$ satisfying
	$$\dist(x_0,\cX^*)\leq \min\left\{\frac{3\epsilon\mu^2}{64 L^2},\frac{\mu}{2\rho}\right\}.$$ Then Algorithm~\ref{alg:polyak} initialized at $x_0$ produces iterates $x_k$ that always lie in $B_{\epsilon}(\bar x)$ and satisfy
	\begin{equation}\label{eqn:subgrad_lin_loca}
	\dist^2(x_{k+1},\cX^*)\leq \left(1-\tfrac{1}{2}\tau^2\right)\dist^2(x_{k},\cX^*), \qquad \textrm{for all } k\geq 0.
	\end{equation}
	Moreover the iterates converge to some point $x_{\infty}\in\cX^*$ at the R-linear rate
	$$\|x_{k}-x_{\infty}\|_2\leq \frac{16L^3 \cdot \dist(x_0,\cX^*)}{3\mu^3}\cdot \left(1-\tfrac{1}{2}\tau^2\right)^{\frac{k}{2}} \qquad \textrm{for all } k\geq 0.$$
\end{thm}

\begin{thm}[Geometrically decaying subgradient method (local regularity)]\label{thm:geo_desc_loc}
	Suppose that Assumption~\ref{ass:target_weak_conv_loc} holds and that $\tau\leq \frac{1}{\sqrt{2}}$. Define $\gamma=\frac{\epsilon\rho}{4L+\epsilon\rho}$, $\lambda=\frac{\gamma\mu^2}{\rho L}$, and $q=\sqrt{1-(1-\gamma)\tau^2}$.  Then  Algorithm~\ref{alg:geometrically_step} initialized at any point $x_0\in B_{\epsilon/4}(\bar x)\cap \mathcal{T}_{\gamma}$ generates iterates $x_k$ that always lie in $B_{\epsilon}(\bar x)$ and satisfy
	\begin{equation} \label{eq:geometric_rate_loc}
	\dist^2(x_k;\cX^*) \leq \frac{\gamma^2 \mu^2}{\rho^2}
	\left(1-(1-\gamma)\tau^2\right)^{k}\qquad \textrm{for all } k\geq 0.
	\end{equation}
	Moreover, the iterates converge to some point $x_{\infty}\in \cX^*$ at the R-linear rate
	$$
	\|x_k-x_{\infty}\|_2\leq \tfrac{\lambda}{1-q}\cdot q^k \qquad \textrm{for all } k\geq 0.
	$$
\end{thm}

We end the section by specializing to the composite setting and analyzing the prox-linear method. The following is the localized version of Assumption~\ref{ass:bb}.

\begin{assumption}\label{ass:bb_loc}
	{\rm
		Consider the optimization problem,
		\begin{equation*}
		\min_{x\in \cX} f(x):=h(F(x)),
		\end{equation*}
		where the function $h(\cdot)$ and the set $\cX$ are convex and $F(\cdot)$ is differentiable.
		Fix a point $\bar x\in \cX^*$ and suppose that the following properties holds for some real numbers $\epsilon, \mu,\rho>0$.
		\begin{enumerate}
			\item {\bf (Approximation accuracy)} The convex models $f_x(y):=h(F(x)+\nabla F(x)(y-x))$ satisfy the estimate:
			$$|f(y)-f_x(y)|\leq \frac{\rho}{2}\|y-x\|^2_2\qquad \forall x\in \cX\cap B_{\epsilon}(\bar x), y\in \cX.$$
			\item {\bf (Sharpness)} The inequality holds:
			\[f(x) - \inf_{\cX} f \geq \mu \cdot  \dist\left(x,\cX^\ast \right) \qquad \forall x\in \cX\cap B_{\epsilon}(\bar x). \]
	\end{enumerate}}
\end{assumption}

The following theorem provides convergence guarantees for the prox-linear method under Assumption~\ref{ass:bb_loc}.

\begin{thm}[Prox-linear (local)]\label{thm:prox_lin_loc}
	Suppose Assumption~\ref{ass:bb_loc} holds, choose any $\beta\geq \rho$, and
	fix  an arbitrary point $x_0\in B_{\epsilon/2}(\bar x)$ satisfying
	$$f(x_0)-\min_{\cX} f\leq \min\left\{\frac{\beta\epsilon^2}{25}, \frac{\mu^2}{2\beta}\right\}.$$
	Then Algorithm~\ref{alg:prox_lin} initialized at $x_0$ generates iterates $x_k$ that always lie in $B_{\epsilon}(\bar x)$ and satisfy
	\begin{align*}
	\dist(x_{k+1},\cX^*)&\leq \frac{\beta}{\mu}\cdot \dist^2(x_k,\cX^*),\\
	f(x_{k+1})-\min_{\cX} f&\leq \frac{\beta}{\mu^2}\left(f(x_k)-\min_{\cX} f\right)^2.
	\end{align*}
	Moreover the iterates converge to some point $x_{\infty}\in\cX^*$ at the quadratic rate
	$$\|x_{k}-x_{\infty}\|_2\leq \frac{2\sqrt{2}\mu}{\beta}\cdot \left(\frac{1}{2}\right)^{2^{k-1}}\qquad \textrm{for all } k\geq 0.$$
\end{thm}

With the above generic results in hand, we can now derive the convergence guarantees for the subgradient and prox-linear methods for asymmetric low-rank matrix recovery problems. To summarize, the prox-linear method converges quadratically, as long as it is initialized within constant relative error of the solution. The guarantees for the subgradient methods are less satisfactory: the size of the region of the linear convergence scales with the condition number of $M_{\sharp}$.  The reason is that the proof estimates the region of convergence using the length of the iterate path, which scales with the condition number.
The dependence on the condition number in general can be eliminated by introducing
regularization  $ \|X^\top X - YY^\top \|_F$, as suggested in the work
\cite{li2018nonconvex_robust}. Still the results we present here are notable even for the subgradient method. For example, we see that for rank $r=1$ instances satisfying RIP (e.g. blind deconvolution), the condition number of $M_{\sharp}$ is always one and therefore regularization is not required at all for subgradient and prox-linear methods.

\begin{corollary}[Convergence guarantees  under RIP (asymmetric)]\label{cor:generic_conv_iso_asym}
	Suppose Assumptions~\ref{assump:RIP} and \ref{assump:outlier} are valid\footnote{ with $\opnorm{\cdot}=\frac{1}{m}\|\cdot\|_1$} and
	consider the optimization problem
	$$\min_{X\in \R^{d_1\times r},~ Y\in \R^{r\times d_2}}~ f(X)=\frac{1}{m}\|\cA(XY) - b\|_1.$$
	Define $X_{\sharp}:=U\sqrt{\Lambda}$ and $Y_{\sharp}=\sqrt{\Lambda}V^\top$, where $M_{\sharp}=U\Lambda V^\top$ is any compact singular value decomposition of $M_{\sharp}$.
	Define also the condition number $\chi:=\sigma_1(M_{\sharp})/\sigma_r(M_{\sharp})$. Then there exists $\eta >0$ depending only on $\ur$, $\ir$, and $\sigma(M_{\sharp})$ such that  the following are true.

	\begin{enumerate}
		\item {\bf (Polyak subgradient)}
		Algorithm~\ref{alg:polyak} initialized at $(X_0,Y_0)$ satisfying
		$\frac{\|(X_0,Y_0)-(X_{\sharp},Y_{\sharp})\|_F}{\sqrt{\sigma_r(M_{\sharp})}}\lesssim \min\{1, \frac{\ir^2}{\ur^2\chi}, \frac{\ir}{\ur}\}$,
		will generate an iterate sequence that converges at the linear rate:
		$$\frac{\dist((X_k,Y_k),\cD^*(M_{\sharp}))}{\sqrt{\sigma_r(M_{\sharp})}}\leq \delta \qquad \textrm{after}\qquad k\gtrsim\frac{\ur^2\chi^2}{\ir^2}\cdot\ln\left(\frac{\eta}{\delta}\right)\qquad \textrm{iterations}.$$

		\item {\bf (geometric subgradient)}
		Algorithm~\ref{alg:geometrically_step} initialized at $(X_0,Y_0)$ satisfying
		$\frac{\|(X_0,Y_0)-(X_{\sharp},Y_{\sharp})\|_F}{\sqrt{\sigma_r(M_{\sharp})}}\lesssim \min\{1, \frac{\ir}{\ur\chi}\}$,
		will generate an iterate sequence that converges at the linear rate:
		$$\frac{\dist((X_k,Y_k),\cD^*(M_{\sharp}))}{\sqrt{\sigma_r(M_{\sharp})}}\leq \delta \qquad \textrm{after}\qquad k\gtrsim\frac{\ur^2\chi^2}{\ir^2}\cdot\ln\left(\frac{\eta}{\delta}\right)\qquad \textrm{iterations}.$$
		\item {\bf (prox-linear)}
		Algorithm~\ref{alg:prox_lin} initialized at $(X_0,Y_0)$ satisfying							  $\frac{f(x_0)-\min_{\cX} f}{\sigma_{r}(M_{\sharp})}\lesssim \min\{\ur ,\ir^2/\ur\}$
		and
		$\frac{\|(X_0,Y_0)-(X_{\sharp},Y_{\sharp})\|_F}{\sqrt{\sigma_r(M_{\sharp})}}\lesssim 1$,
		will generate an iterate sequence that converges at the quadratic rate:
		$$\frac{\dist((X_k,Y_k),\cD^*(M_{\sharp}))}{\sqrt{\sigma_r(M_{\sharp})}}\lesssim \frac{\ir}{\ur}\cdot 2^{-2^{k}}\qquad \textrm{for all } k\geq 0.$$
	\end{enumerate}
\end{corollary}

\section{Examples of $\ell_1/\ell_2$ RIP}\label{sec:all_examples}

In this section, we survey three matrix recovery problems from different fields, including physics, signal processing, control theory, wireless communications, and machine learning, among others. In all cases, the problems satisfy $\ell_1/\ell_2$ RIP and the $\cI$-outlier bounds and consequently, the convergence results in Corollaries~\ref{cor:generic_conv_iso_sym} and~\ref{cor:generic_conv_iso_asym} immediately apply. Most of the RIP results in this section were previously known (albeit under more restrictive assumptions); we provide self-contained arguments in the Appendix~\ref{appendix:RIP} for the sake of completeness.
On the other hand, using nonsmooth optimization in these problems and the corresponding convergence guarantees based on RIP are, for the most part, new. 

For the rest of this section we will assume the following data-generating mechanism.
\begin{definition}[\textbf{Data-generating mechanism}] \label{def:data_generation}
	A random linear map $\cA\colon\R^{d_1\times d_2}\to\R^m$ and a random index set $\cI\subset [m]$ are drawn independently of each other. We assume moreover that the outlier frequency  $\pfail:=|\cI|/m$ satisfies $\pfail\in [0,1/2)$ almost surely. We then observe the corrupted measurements
	\begin{equation}\label{eq:measurements}
	b_i = \begin{cases}
	\cA(M_\sharp) & \text{if }i \notin \cI, \text{ and} \\
	\eta_i & \text{if } i \in \cI,
	\end{cases}
	\end{equation}
	where $\eta$ is an arbitrary vector. In particular, $\eta$ could be correlated with $\cA$.
\end{definition}

Throughout this section, we consider four distinct linear operators $\cA$.

\paragraph{Matrix Sensing.} In this scenario, measurements are generated as follows:
\begin{equation} \label{eq:const_matrix_sens}
\cA(M_\sharp)_i := \dotp{\lM_i, M_\sharp} \qquad \text{for } i = 1, \ldots, m
\end{equation}
where $\lM_i \in \RR^{d_1 \times d_2}$ are fixed matrices.

\paragraph{Quadratic Sensing I .} In this scenario, $M_\sharp \in \RR^{d \times d}$ is assumed to be a PSD rank $r$ matrix with factorization $M_\sharp = X_\sharp X_\sharp ^\top$ and measurements are generated as follows:
\begin{equation} \label{eq:quadratic_measurements_I}
\cA(M_\sharp)_i = \lv_i^\top M_\sharp \lv_i = \|X_\sharp^\top \lv_i\|^2_2\qquad \text{for } i = 1, \ldots, m ,
\end{equation}
where $\lv_i \in \RR^d$ are fixed vectors.

\paragraph{Quadratic Sensing II .} In this scenario, $M_\sharp \in \RR^{d \times d}$ is assumed to be a PSD rank $r$ matrix with factorization $M_\sharp = X_\sharp X_\sharp ^\top$ and measurements are generated as follows:
\begin{equation} \label{eq:quadratic_measurements_II}
\cA(M_\sharp)_i = \lv_i^\top M_\sharp \lv_i - \tilde \lv_i^\top M_\sharp \tilde \lv_i = \|X_\sharp^\top \lv_i\|^2_2 - \|X_\sharp^\top \tilde \lv_i\|^2_2\qquad \text{for } i = 1, \ldots, m ,
\end{equation}
where $\lv_i,\tilde \lv_i \in \RR^d$ are fixed vectors.

\paragraph{Bilinear Sensing.} In this scenario, $M_\sharp \in \RR^{d_1 \times d_2}$ is assumed to be a $r$ matrix with factorization $M_\sharp = XY$ and measurements are generated as follows:\begin{equation}\label{eq:bilinear_measurements}
\cA(M_\sharp)_i = \lv_i^\top M_\sharp \rv_i \qquad \text{for } i = 1, \ldots, m ,
\end{equation}
where $\lv_i \in \RR^{d_1}$ and $\rv_i \in \RR^{d_2}$ are fixed vectors.

~\\
\indent The matrix, quadratic, and bilinear sensing problems have been considered in a number of papers and in a variety of applications. The first theoretical properties for matrix sensing were discussed in~ \cite{faz,guaran,candes2011tight}. Quadratic sensing in its full generality appeared in~\cite{MR3367819} and is a higher-rank generalization of the much older (real) phase retrieval problem~\cite{phase,candes2013phaselift,goldstein2018phasemax}. Besides phase retrieval, quadratic sensing has applications to covariance sketching, shallow neural networks, and quantum state tomography; see for example \cite{li2018nonconvex} for a discussion. Bilinear sensing is a natural modification of  quadratic sensing and is a higher-rank generalization of the blind deconvolution problem~\cite{ahmed2014blind}; it was first proposed and studied in \cite{ROP_matrix_recovery_rank_one}.

The reader is reminded that once $\ell_1/\ell_2$ RIP guarantees, in particular Assumptions~\ref{assump:RIP} and~\ref{assump:outlier}, are established for the above four operators, the guarantees of Corollaries~\ref{cor:generic_conv_iso_sym} and Corollary~\ref{cor:generic_conv_iso_asym} immediately take hold for the problems
$$\min_{X\in \R^{d\times r}}~ f(X)=\frac{1}{m}\|\cA(XX^\top) - b\|_1$$
and
$$\min_{X\in \R^{d_1\times r},~ Y\in \R^{r\times d_2}}~ f(X)=\frac{1}{m}\|\cA(XY) - b\|_1,$$
respectively. Thus, we turn our attention to establishing such guarantees.

\subsection{Warm-up: $\ell_2/\ell_2$ RIP for matrix sensing with Gaussian design}
In this section, we are primarily interested in the $\ell_1/\ell_2$ RIP for the above four linear operators. However, as a warm-up, we first consider the $\ell_2/\ell_2$-RIP property for matrix sensing with Gaussian $P_i$. The following result appears in~\cite{guaran,candes2011tight}.
\begin{thm}[\textbf{$\ell_2/\ell_2$-RIP for matrix sensing}]\label{thm:matrix_sensing_ell_2_rip}
	For any $\delta\in (0,1)$ there exist constants $c,C>0$ depending only on $\delta$ such that if the entries of $\lM_i$ are i.i.d.\ standard Gaussian and $m\geq c r(d_1+d_2)\log(d_1d_2)$, then with probability at least $1-\exp\left(-Cm\right)$, the estimate
	$$(1-\delta)\|M\|_F\leq \frac{1}{\sqrt{m}} \|\cA(M)\|_2 \leq (1+\delta)\|M\|_F,$$
	holds simultaneously for all $M\in\R^{d_1\times d_2}$ of rank at most $2r$.
	Consequently, Assumption~\ref{assump:RIP} is satisfied.
\end{thm}

Following the general recipe of the paper, we see that the nonsmooth formulation
\begin{align}\label{eq:_ell2_matrix_sensing}
\min_{X\in \R^{d_1\times r},~Y\in \R^{r\times d_2}} \frac{1}{\sqrt{m}}\|\cA(XY)-b\|_2=\sqrt{\frac{1}{m}\sum_{i=1}^m \left(\tr( Y\lM_i^\top X)-b_i\right)^2}
\end{align}
is immediately amenable to subgradient and prox-linear algorithms in the noiseless setting $\cI = \emptyset$. In particular, a direct analogue of Corollary~\ref{cor:generic_conv_iso_asym}, which was stated for the penalty function $h = \frac{1}{m} \| \cdot\|_1$, holds; we omit the straightforward details.

\subsection{The $\ell_1/\ell_2$ RIP and $\cI$-outlier bounds: quadratic and bilinear sensing}
We now turn our attention to the $\ell_1/\ell_2$ RIP for more general classes of linear maps than the i.i.d.\ Gaussian matrices considered in Theorem~\ref{thm:matrix_sensing_ell_2_rip}. To establish such guarantees, one must ensure that the linear maps $\cA$ have light tails and are robustly injective on certain spaces of matrices. The first property leads to tight concentration results, while the second yields the existence of a lower RIP constant $\lr$.

\begin{assumption}[Matrix Sensing]\label{assumption:matrix_sensing}
	The matrices $\{\lM_i\}$ are i.i.d.\ realizations of an $\eta$-sub-Gaussian random matrix\footnote{By this we mean that the vectorized matrix $\vect(\lM)$ is a $\eta$-sub-gaussian random vector.} $\lM \in \RR^{d_1 \times d_2}.$ Furthermore, there exists a numerical constant $\co > 0$ such that
	\begin{equation} \label{assump:matrix_sensing} \inf_{\substack{M: \; \rank M \leq 2r \\ \|M\|_F = 1}} \EE |\dotp{\lM,M}|  \geq \co. \end{equation}
\end{assumption}

\begin{assumption}[Quadratic Sensing I]\label{assumption:cov_est_naive}
	The vectors $\{\lv_i\}$ are i.i.d.\ realizations of a $\eta$-sub-Gaussian random variable $\lv \in \RR^d.$ Furthermore, there exists a numerical constant $\co > 0$ such that
	\begin{equation} \label{assump:covariance_estimation1} \inf_{\substack{M\in \cS^d: \; \rank M \leq 2r \\ \|M\|_F = 1}} \EE|\lv^\top M \lv| \geq \co. \end{equation}
\end{assumption}

\begin{assumption}[Quadratic Sensing II]\label{assumption:cov_est_sym}
	The vectors $\{\lv_i\}, \{\tilde \lv_i\}$ are i.i.d.\ realizations of a $\eta$-sub-Gaussian random variable $\lv \in \RR^d.$ Furthermore, there exists a numerical constant $\co > 0$ such that
	\begin{equation} \label{assump:covariance_estimation2}
	\inf_{\substack{M\in\cS^d: \; \rank M \leq 2r \\ \|M\|_F = 1}} \EE |\lv^\top M \lv - \tilde \lv^\top M \tilde \lv| \geq \co .
	\end{equation}
\end{assumption}

\begin{assumption}[Bilinear Sensing]\label{assumption:bilinear}
	The vectors $\{\lv_i\}$ and $\{\rv_i\}$ are i.i.d.\ realizations of $\eta$-sub-Gaussian random vectors $\lv \in \RR^{d_1}$ and $\rv \in \RR^{d_2},$ respectively. Furthermore, there exists a numerical constant $\co > 0$ such that
	\begin{equation} \label{assump:covariance_estimation3} \inf_{\substack{M: \; \rank M \leq 2r \\ \|M\|_F = 1}} \EE|\lv^\top  M \rv |  \geq \co. \end{equation}
\end{assumption}

~\\
The Assumptions~\ref{assumption:matrix_sensing}-\ref{assumption:bilinear} are all valid for i.i.d.\ Gaussian realizations with independent identity covariance, as the following lemma shows. We defer its proof to Appendix~\ref{proof:lem_small_ball}.
\begin{lem}\label{lem:small_ball}
	Assumption~\ref{assumption:matrix_sensing} holds for matrices $P$ with i.i.d.\ standard Gaussian entries. Assumptions~\ref{assumption:cov_est_naive} and~\ref{assumption:cov_est_sym} hold for vectors $\lv, \tilde \lv$ with i.i.d.\ standard Gaussian entries. Assumption~\ref{assumption:bilinear} holds for vectors $\lv$ and $\rv$ with i.i.d.\ standard Gaussian entries.
\end{lem}

We can now state the main RIP guarantees under the above assumptions. Throughout all the results, we fix the data generating mechanism as in Definition~\ref{def:data_generation}. Then, we wish to establish the inequalities
\begin{equation} \label{eqn:RIPM1_thm}
\lr \|M\|_F \leq \frac{1}{m} \|\cA(M)\|_1 \leq \ur \|M\|_F
\end{equation}
and
\begin{equation} \label{eqn:oulier_thm}
\ir \|M\|_F \leq \frac{1}{m}\big( \|\cA_{\cI^c} (M)\|_1 - \|\cA_{\cI} (M) \|_1 \big),
\end{equation}
and, hence, Assumptions~\ref{assump:RIP} and~\ref{assump:outlier}, respectively, for certain constants $\kappa_1, \kappa_2,$ and $\kappa_3$. We defer the proof of this theorem to Appendix~\ref{sec:proof_main_rip_theo}.

\begin{thm}[$\ell_1/\ell_2$ RIP and $\cI$-outlier bounds]\label{thm:main_rip_theorem}
	There exist numerical constants $c_1, \dots, c_6 > 0$ depending only on $\co, \eta$ such that the following hold for the corresponding measurement operators described in Equations~\eqref{eq:const_matrix_sens},~\eqref{eq:quadratic_measurements_I},~\eqref{eq:quadratic_measurements_II}, and~\eqref{eq:bilinear_measurements}, respectively
	\begin{enumerate}
		\item \label{thm:main_rip_theorem:item:matrix_sensing} {\bf (Matrix sensing)} Suppose Assumption~\ref{assumption:matrix_sensing} holds. Then provided $m \geq \frac{c_1}{(1 - 2 \pfail)^2} r(d_1+d_2+1)\ln \left( c_2 + \frac{c_2 }{1 - 2 \pfail}\right)$, we have with probability at least $1-4\exp\left( -c_3 (1-2\pfail)^2 m \right)$ that every matrix $M \in \RR^{d_1 \times d_2}$ of rank at most $2r$ satisfies~\eqref{eqn:RIPM1_thm} and~\eqref{eqn:oulier_thm} with constants $\kappa_1 = c_4, \kappa_2 = c_5$ and $\kappa_3 = c_6(1-2\pfail)$.
		\item\label{thm:main_rip_theorem:item:quadratic_sensing_I} {\bf (Quadratic sensing I)} Suppose Assumption~\ref{assumption:cov_est_naive} holds. Then provided  $m \geq \frac{c_1}{(1 - 2 \pfail)^2} r^2 (2d+1)\ln \left( c_2 + \frac{c_2 }{1 - 2 \pfail}\sqrt{r}\right)$, we have with probability at least $1-4\exp\left( -c_3 (1-2\pfail)^2 m/r \right)$ that every matrix $M \in \RR^{d \times d}$ of rank at most $2r$ satisfies~\eqref{eqn:RIPM1_thm} and~\eqref{eqn:oulier_thm} with constants $\kappa_1 = c_4, \kappa_2 = c_5 \cdot \sqrt{r} $ and $\kappa_3 = c_6(1-2\pfail)$.
		\item\label{thm:main_rip_theorem:item:quadratic_sensing_II} {\bf (Quadratic sensing II)} Suppose Assumption~\ref{assumption:cov_est_sym} holds. Then provided $m \geq \frac{c_1}{(1 - 2 \pfail)^2} r (2d+1)\ln \left( c_2 + \frac{c_2 }{1 - 2 \pfail}\right)$, we have with probability at least $1-4\exp\left( -c_3 (1-2\pfail)^2 m \right)$ that every matrix $M \in \RR^{d \times d}$ of rank at most $2r$ satisfies~\eqref{eqn:RIPM1_thm} and~\eqref{eqn:oulier_thm} with constants $\kappa_1 = c_4, \kappa_2 = c_5$ and $\kappa_3 = c_6(1-2\pfail)$.
		\item\label{thm:main_rip_theorem:item:bilinear_sensing} {\bf (Bilinear sensing)} Suppose Assumption~\ref{assumption:bilinear} holds. Then provided $m \geq \frac{c_1}{(1 - 2 \pfail)^2} r (d_1+d_2+1)\ln \left( c_2 + \frac{c_2 }{1 - 2 \pfail}\right)$, we have with probability at least $1-4\exp\left( -c_3 (1-2\pfail)^2 m \right)$ that every matrix $M \in \RR^{d_1 \times d_2}$ of rank at most $2r$ satisfies~\eqref{eqn:RIPM1_thm} and~\eqref{eqn:oulier_thm} with constants $\kappa_1 = c_4, \kappa_2 = c_5$ and $\kappa_3 = c_6(1-2\pfail)$.
	\end{enumerate}
\end{thm}
The guarantees of Theorem~\ref{thm:main_rip_theorem} were previously known under  stronger assumptions. In particular, item~\eqref{thm:main_rip_theorem:item:matrix_sensing} generalizes the results in \cite{li2018nonconvex_robust} for the pure Gaussian setting. The case $r = 1$  of item~\eqref{thm:main_rip_theorem:item:quadratic_sensing_I} can be found, in a sightly different form, in \cite{eM, duchi_ruan_PR}. Item~\eqref{thm:main_rip_theorem:item:quadratic_sensing_II} sharpens slightly the analogous guarantee in \cite{MR3367819} by weakening the assumptions on the moments of the measuring vectors to the uniform lower bound \eqref{assump:covariance_estimation2}. Special cases of item~\eqref{thm:main_rip_theorem:item:bilinear_sensing} were established in \cite{charisopoulos2019composite}, for the case $r=1$, and \cite{ROP_matrix_recovery_rank_one}, for Gaussian measurement vectors.

We note that all linear mappings require the same number of measurements in order to satisfy RIP and $\cI$ outlier bounds, except for quadratic sensing I operator, which incurs an extra $r$-factor. This reveals the utility of the quadratic sensing II operator, which achieves optimal sample complexity. For larger scale problems, a shortcoming of matrix sensing operator~\eqref{eq:const_matrix_sens} is that $md_1 d_2$ scalars are required to represent the map $\cA$. In contrast, all other measurement operators may be represented with only $m(d_1+d_2)$ scalars.

\section{Matrix Completion}\label{sec:mat_comp}
In the previous sections, we saw that low-rank recovery problems satisfying RIP lead to well-conditioned nonsmooth formulations. We claim, however, that the general framework of sharpness and approximation is applicable even for problems without RIP. We consider two such problems, namely matrix completion in this section and robust PCA in  Section~\ref{sec:robust_pca}, to follow. Both problems will be considered in the symmetric setting.

 The goal of matrix completion problem is to recover a PSD rank $r$ matrix $M_{\sharp} \in \mathcal{S}^d$ given access only to a subset of its entries. Henceforth, let $X_{\sharp}\in \R^{d\times r}$ be a matrix satisfying $M_{\sharp}=X_{\sharp}X_{\sharp}^\top$.
Throughout, we assume incoherence condition, $\|X_{\sharp}\|_{2,\infty}\leq
\sqrt{\frac{\nu r}{d}}$,  for some $\nu>0$. We also make the fairly strong assumption that the singular values of $X_{\sharp}$ are all equal $\sigma_1(X_{\sharp})=\sigma_2(X_{\sharp})=\ldots=\sigma_r(X_{\sharp}) =
1$. This assumption is needed for our theoretical results. 
We let  $\Omega\subseteq [d]\times [d]$ be an index set generated by the Bernoulli model, that is, $ \PP((i,j), (j,i) \in \Omega) = p $ independently for all $ 1\le i\le j \le d $. Let $\Pi_{\Omega}\colon\mathcal{S}^d\to \R^{|\Omega|}$ be the projection onto the entries indexed by $\Omega$. We consider the following optimization formulation of the problem
$$
\min_{X \in \cX}~f(X)=\|\Pi_{\Omega}(XX^\top)-\Pi_{\Omega}(M_{\sharp})\|_2 \quad \text{where $\cX = \left\{X\in \R^{d\times r} \colon ~ \|X\|_{2,\infty}\leq \sqrt{\frac{\nu r}{d}}\right\}$}.
$$
 We will show that both the Polyak subgradient method and an appropriately modified prox-linear algorithm converge linearly to the solution set under reasonable initialization. Moreover, we will see that the linear rate of convergence for the prox-linear method is much better than that for the subgradient method.

To simplify notation, we set
$$\cD^*:=\cD^*(M_{\sharp})=\{X\in \R^{d_1\times r}: XX^\top=M_{\sharp}\}.
$$

We begin by estimating the sharpness constant $\mu$ of the objective function. Fortunately, this estimate follows directly from inequalities (58) and (59a) in \cite{chen2015fast}.

		\begin{lem}[Sharpness \cite{chen2015fast}]\label{lem:sharp_mat_comp}
			There are numerical constant $c_1,c_2>0$ such that the following holds. If $p\geq c_2(\frac{\nu^2 r^2}{d}+ \frac{\log d}{d})$, then with probability $1-c_1d^{-2}$, the estimate
$$\frac{1}{p}	\|\Pi_{\Omega}(XX^\top - X_{\sharp}X_{\sharp}^\top)\|_F ^2 \geq c_1 \|XX^\top- X_{\sharp}X_{\sharp}^\top\|_F^2$$
holds uniformly for all $X\in \cX$ with $\dist(X,\mathcal{D}^{*})\leq c_1$.
			\end{lem}

Let us next estimate the approximation accuracy $|f(Z)-f_{X}(Z)|$, where
	$$
	f_X(Z) = \|\Pi_{\Omega}(XX-M_\sharp+X(Z-X)^{\top}+(Z-X)X^\top) \|_F.
	$$ To this end, we will require the following result.

	\begin{lem}[Lemma 5 in \cite{chen2015fast}]\label{lem:matrix_completion_quad_term}
		There is a numerical constant $c>0$ such that the following holds. If  $p\geq \frac{c}{\epsilon^2}(\frac{\nu^2 r^2}{d}+ \frac{\log d}{d})$ for some $\epsilon \in (0,1)$,  then with probability at least $1-2d^{-4}$, the estimates
\begin{enumerate}
\item $ \frac{1}{\sqrt{p}}\|\Pi_{\Omega}(H H^\top)\|_F \leq
\sqrt{(1+\epsilon)}\|H\|_F^2 + \sqrt{\epsilon} \|H\|_F$; and
\item \label{lem:matrix_completion_quad_term:item:lip} $\frac{1}{\sqrt{p}}\|\Pi_{\Omega}(G H^\top)\|_F \leq \sqrt{\nu r}\|G\|_F$
\end{enumerate}
		hold uniformly for all matrices $H$ with $\|H\|_{2,\infty} \leq 6\sqrt{\frac{\nu r}{d}}$ and $G \in \RR^{d \times r}$.
		\end{lem}

	An estimate of the approximation error $|f(Z)-f_{X}(Z)|$ is now immediate.
\begin{lemma}[Approximation accuracy and Lipschitz continuity]\label{lem:approx_acc_mat_comp}
There is a numerical constant $c>0$ such that the following holds. If  $p\geq \frac{c}{\epsilon^2}(\frac{\nu^2 r^2}{d}+ \frac{\log d}{d})$ for some $\epsilon \in (0,1)$, then with probability at least $1-2d^{-4}$, the estimates

\begin{align*}
\frac{1}{\sqrt{p}}|f(X) - f_Y(X)| &\leq \sqrt{(1+\epsilon)}\|X - Y\|_F^2 + \sqrt{\epsilon} \|X - Y\|_F,\\
|f(X) - f(Y)|&\leq \sqrt{p\nu r}\|X - Y\|_F,
\end{align*}
holds uniformly for all $X, Y\in \cX$.
\end{lemma}
\begin{proof}
The first inequality follows immediately by observing the estimate
\begin{align*}
|f(X) - f_Y(X)| \leq \|\Pi_{\Omega}((X-Y)(X-Y)^\top)\|_F,
\end{align*}
and using Lemma~\ref{lem:matrix_completion_quad_term}. To see the second inequality, observe \begin{align*}
|f(X) - f(Y)| &\leq \|\Pi_{\Omega}(XX^\top - YY^\top)\|_F\\
&= \frac{1}{2}\|\Pi_{\Omega}((X- Y)(X + Y)^\top - (X+Y)(X-Y)^\top)\|_F\\
&\leq \|\Pi_{\Omega}((X- Y)(X + Y)^\top)\|_F\\
&\leq \sqrt{p\nu r}\|X - Y\|_F,
\end{align*}
where the last inequality follows by Part~\ref{lem:matrix_completion_quad_term:item:lip} of Lemma~\ref{lem:matrix_completion_quad_term}.
\end{proof}

Note that the approximation bound in
Lemma~\ref{lem:matrix_completion_quad_term} is not in terms of the square
Euclidean norm.
Therefore the results in Section~\ref{sec:conv_guarant} do not apply directly. Nonetheless, it is straightforward to modify the prox-linear method to take into account the new approximation bound. The proof of the following lemma appears in the appendix.

\begin{lem}\label{lem:basic_conv_guarantee}
Suppose that Assumption~\ref{ass:bb} holds with the approximation property replaced by
$$|f(y)-f_x(y)|\leq a\|y-x\|^2_2+b\|y-x\|_2\qquad \forall x,y\in \cX,$$
for some real $a,b\geq 0$. Consider the iterates generated by the process:
$$ x_{k+1}=\argmin_{x \in\cX}~ \left\{f_{x_k}(x)+a\|x-x_k\|^2_2+b\|x-x_k\|_2\right\}.
$$
Then as long as $x_0$ satisfies $\dist(x_0,\cX^*)\leq \frac{\mu - 2b}{2a}$, the iterates converge linearly:
$$\dist(x_{k+1},\cX^*)\leq \frac{2(b+a\dist(x,\cX^*))}{\mu} \cdot \dist(x_k,\cX^*)\qquad \forall k\geq 0.$$
\end{lem}

Combining Lemma~\ref{lem:basic_conv_guarantee} with our estimates of the sharpness and approximation accuracy, we deduce the following convergence guarantee for matrix completion.
\begin{corollary}[Prox-linear method for matrix completion]
There are numerical constants $c_0,c,C>0$ such that the following holds. If  $p\geq \frac{c}{\epsilon^2}(\frac{\nu^2 r^2}{d}+ \frac{\log d}{d})$ for some $\epsilon \in (0,1)$, then with probability at least $1-c_0d^{-2}$, the iterates  generated by the modified prox-linear algorithm
\begin{equation}\label{eqn:prox_lin_mat_comp}
X_{k+1}=\argmin_{X \in \cX}~ \left\{f_{X_k}(X)+\sqrt{p(1+\epsilon)}\cdot\|X-X_k\|^2_2+\sqrt{p\epsilon}\|X-X_k\|_2\right\}
\end{equation}
satisfy
$$
\dist(X_{k+1},\mathcal{D}^*)\leq \frac{\sqrt{\epsilon}+\sqrt{1+\epsilon}\cdot \dist(X_k,\mathcal{D}^*)}{C} \cdot \dist(X_k,\mathcal{D}^*)\qquad \forall k\geq 0.
$$
In particular, the iterates converge linearly as long as
$\dist(X_0,\cD^*)<\frac{C-2\sqrt{\epsilon}}{2\sqrt{(1+\epsilon)}}$.
\end{corollary}
\begin{proof}
By invoking Proposition~\ref{prop:l2sharpness} and Lemmas~\ref{lem:sharp_mat_comp} and ~\ref{lem:approx_acc_mat_comp} we may appeal to Lemma~\ref{lem:basic_conv_guarantee} with
 $a=\sqrt{p(1+\epsilon)}$, $b=\sqrt{p\epsilon}$, and $\mu=\sqrt{2c_1 p(\sqrt{2}-1)}$. The result follows immediately.
\end{proof}

To summarize, there exist numerical constants $c_0,c_1,c_2,c_3>0$ such that the following is true with probability at least $1-c_0d^{-2}$. In the regime
$$p\geq \frac{c_2}{\epsilon^2}\left(\frac{\nu^2 r^2}{d}+ \frac{\log d}{d}\right) \qquad\textrm { for some } ~\epsilon\in (0,c_1),$$ the prox-linear method will converge at the rapid linear rate,
$$\dist(X_k,\cD^*)\leq \frac{c_2}{2^k},$$ when initialized at $X_0\in\cX$ satisfying $\dist(X_0,\cD^*)<c_2$.

As for the prox-linear method, the results of Section~\ref{sec:conv_guarant} do
not immediately yield convergence guarantees for the Polyak subgradient method.
Nonetheless, it straightforward to show that the standard Polyak subgradient
method still enjoys local linear convergence guarantees. The proof is a straightforward modification of the argument in \cite[Theorem 3.1]{davis2018subgradient}, and appears in the appendix.

\begin{thm}\label{thm:basic_conv_guarantee_sub}
Suppose that Assumption~\ref{ass:bb} holds with the approximation property replaced by
$$|f(y)-f_x(y)|\leq a\|y-x\|^2_2+b\|y-x\|_2\qquad \forall x,y\in \cX,$$
for some real $a,b\geq 0$. Consider the iterates $\{x_k\}$ generated by the Polyak subgradient method in Algorithm~\ref{alg:polyak}.
Then as long as the sharpness constant satisfies $\mu > 2b$ and $x_0$ satisfies $\dist(x_0,\cX^*)\leq \gamma\cdot  \frac{\mu - 2b}{2a}$ for some $\gamma < 1$, the iterates converge linearly
$$\dist^2(x_{k+1},\cX^*)\leq \left(1-\frac{ (1-\gamma)\mu(\mu - 2b)}{L^2}\right) \cdot \dist^2(x_k,\cX^*)\qquad \forall k\geq 0.$$
\end{thm}

Finally, combining Theorem~\ref{thm:basic_conv_guarantee_sub} with our estimates of the sharpness and approximation accuracy, we deduce the following convergence guarantee for matrix completion.
\begin{corollary}[Subgradient method for matrix completion]
There are numerical constants $c_0,c,C>0$ such that the following holds. If  $p\geq \frac{c}{\epsilon^2}(\frac{\nu^2 r^2}{d}+ \frac{\log d}{d})$ for some $\epsilon \in (0,1)$, then with probability at least $1- c_0d^{-2}$, the iterates  generated by the iterates $\{X_k\}$ generated by the Polyak Subgradient method in Algorithm~\ref{alg:polyak}
satisfy
$$
\dist(X_{k+1},\cD^*)^2\leq  \left(1-\frac{C(C - 2\sqrt{\varepsilon})}{2\nu r}\right) \cdot \dist^2(X_k,\cD^*)\qquad \forall k\geq 0.
$$
In particular, the iterates converge linearly as long as
$\dist(X_0,\cD^*)<\frac{C-2\sqrt{\epsilon}}{4\sqrt{(1+\epsilon)}}$.
\end{corollary}
\begin{proof}
First, observe that we have the bound $L\leq \sqrt{p\nu r}$ by Lemma~\ref{lem:approx_acc_mat_comp}.
By invoking Proposition~\ref{prop:l2sharpness} and Lemmas~\ref{lem:sharp_mat_comp} and~\ref{lem:approx_acc_mat_comp} we may appeal to Theorem~\ref{thm:basic_conv_guarantee_sub} with
 $\gamma = 1/2$, $a=\sqrt{p(1+\epsilon)}$, $b=\sqrt{p\epsilon}$, and $\mu=\sqrt{2c_1 p(\sqrt{2}-1)}$. The result follows immediately.
\end{proof}

To summarize, there exist numerical constants $c_0,c_1,c_2,c_3>0$ such that the following is true with probability at least $1-c_0d^{-2}$. In the regime
$$p\geq \frac{c_2}{\epsilon^2}\left(\frac{\nu^2 r^2}{d}+ \frac{\log d}{d}\right) \qquad\textrm { for some } ~\epsilon\in (0,c_1),$$ the Polyak subgradient method will converge at the  linear rate,
$$\dist(X_k,\cD^*)\leq \left(1-\frac{c_3}{\nu r}\right)^{\frac{k}{2}}c_2,$$ when initialized at $X_0\in\cX$ satisfying $\dist(X_0,\cD^*)<c_2$. Notice that the prox-linear method enjoys a much faster linear rate of convergence than the subgradient method---an observation fully supported by numerical experiments in Section~\ref{sec:num_exp}. The caveat is that the per iteration cost of the prox-linear method is significantly higher than that of the subgradient method.

\section{Robust PCA}\label{sec:robust_pca}

The goal of robust PCA is to decompose a given matrix $W$ into a sum of a low-rank matrix $ M_{\sharp} $ and a sparse matrix $ S_{\sharp} $, where $ M_{\sharp} $ represents the principal components, $ S_{\sharp} $ the corruption, and $ W $ the observed data \cite{chand,rob_cand,yi2016rpca}.
In this section, we explore methods of nonsmooth optimization for recovering
such a decomposition, focusing on two different problem formulations. We only consider the symmetric version of the problem.

\subsection{The Euclidean formulation}
\label{sec:l2_robust_pca}

Setting the stage, we assume that the matrix $W\in\R^{d\times d}$ admits a decomposition $W=M_{\sharp}+S_{\sharp}$, where the matrices $M_{\sharp}$ and $S_\sharp$ satisfy the following   for some parameters $\nu>0$ and $k\in \mathbb{N}$:
\begin{enumerate}
	\item The matrix $M_{\sharp}\in\R^{d\times d}$ has rank $ r $ and can be
	factored as $M_{\sharp}=X_{\sharp}X^{\top}_{\sharp}$ for some matrix
	$X_{\sharp}\in\R^{d\times r}$ satisfying $\|X_{\sharp}\|_{{\rm op}}\leq 1$ and
	$\|X_{\sharp}\|_{2,\infty}\leq \sqrt{\frac{\nu r}{d}}$.\footnote{Recall that
		$\|X\|_{2,\infty} = \max_{i \in [d]} \|X_{i \cdot}\|_2$ is the maximum row
		norm.}
	\item The matrix $S_{\sharp}$ is sparse in the sense that it has at most $k$ nonzero entries per column/row. 
\end{enumerate}
The goal is to recover $ M_{\sharp} $ and $ S_{\sharp} $ given $ W $. The first
formulation we consider is the following:
\begin{equation}\label{eqn:l2_formul}
\min_{X\in \mathcal{X},S\in \cS}~ F\big((X,S)\big)=\|XX^{\top} +S -W\|_F,
\end{equation}
where the constraint sets are defined by
$$\cS :=\left\{S\in \mathbb{R}^{d\times d}: \|Se_i\|_1 \leq \|S_{\sharp}e_i\|_1 \; \forall i\right\}, \qquad\mathcal{X} = \left\{X\in \mathbb{R}^{d\times r}: \|X\|_{2,\infty} \leq \sqrt{\frac{\nu r}{d}}\right\}.$$
Note that the problem formulation requires knowing the $\ell_1$ norms of the rows of $S_\sharp$. The same assumption was also made in~\cite{chen2015fast,ge2017unified}.  While admittedly unrealistic, this formulation provides a nice illustration of the paradigm we advocate here.
The following technical lemma will be useful in proving the regularity conditions needed for rapid convergence. The proof is given in Appendix~\ref{appendix:proof_rpca_cross_term}.
\begin{lem}
	\label{lem:rpca_cross_term}
	For all $X\in \mathcal{X}$ and $S  \in \cS$, the estimate holds:
	\begin{align*}
	| \langle S-S_{\sharp}, XX^\top -X_{\sharp}X_{\sharp}^{\top}\rangle  | &  \leq 10 \sqrt{\frac{\nu r k}{d}}\cdot \|S - S_{\sharp}\|_F\cdot \|X-X_{\sharp}\|_F.
	\end{align*}
\end{lem}

Equipped with the above lemma, we can estimate the sharpness and approximation
parameters $\mu, \rho$ for the formulation \eqref{eqn:l2_formul}.

\begin{lem}[Regularity constants]\label{qgup}
	For all $X\in \mathcal{X}$ and $S  \in \cS$, the estimates hold:
	\begin{equation}\label{qg}
	\begin{aligned}
	F((X,S))^2 	& \geq \left(\frac{1}{2}\sigma_r^2(X_\sharp) - 10\sqrt{\frac{\nu rk}{ d}}\right) \cdot \left( \dist(X, \cD^\ast(M_\sharp))^2 + \|S - S_\sharp \|_F^2\right)
	\end{aligned}
	\end{equation}
	and
	\begin{equation} \label{up}
	\begin{aligned}
	|F((X,S))-F_Y((X,S)) | \leq \|X-Y\|_F^2 .
	\end{aligned}
	\end{equation}
	Moreover, for any $X_1,X_2\in \cX$ and $S_1,S_2\in \cS$, the Lipschitz bounds holds:
	$$|F((X_1,S_1))-F((X_2,S_2))|\leq 2\sqrt{\nu r}\|X_1-X_2\|_F + \|S_1 - S_2\|_F.$$
\end{lem}
\begin{proof}
	Let $X_\sharp \in \proj_{\cD^\ast(M_\sharp)}(X)$. To establish the bound~\eqref{qg}, we observe that
	\begin{equation*}
	\begin{aligned}
	\|XX^\top +S -W\|_F^2 & = \|XX^\top - M_{\sharp}\|_F^2 + 2 \langle S-S_{\sharp},XX^\top-M_{\sharp}\rangle + \|S-S_{\sharp}\|_F^2\\
	&\geq  \frac{1}{2}\sigma_r^2(X_{\sharp})\|X-X_{\sharp}\|_F^2  - 20 \sqrt{\frac{\nu r k}{d}} \|S - S_{\sharp}\|_F\|X-X_{\sharp}\|_F +\|S-S_{\sharp}\|_F^2,
	\end{aligned}
	\end{equation*}
	where the first inequality follows from Proposition~\ref{prop:l2sharpness} and Lemma~\ref{lem:rpca_cross_term}. Now set
	\begin{align*}
	a:= 10 \sqrt{\frac{\nu rk }{d}}, && b := \|X - X_\sharp\|_F,  && c := \|S - S_\sharp\|_F,
	\end{align*}
	and $s := \frac{1}{2}\sigma_r^2(X_\sharp)$.
	With this notation, we apply the Fenchel-Young inequality to show that for
	any $\varepsilon > 0$, we have
	$$
	2abc \leq a\varepsilon b^2+ (a/\varepsilon) c^2.
	$$
	Thus, for any $\varepsilon > 0$, we have
	$$
	\|XX^\top +S -W\|_F^2 \geq sb^2 - 2abc + c^2 \geq (s - a \varepsilon) b^2 + (1
	- a/\varepsilon) c^2.
	$$
	Now, let us choose $\varepsilon > 0$ so that
	$
	s - a\varepsilon = 1 - a/\varepsilon.
	$
	Namely set $\varepsilon = \frac{-(1-s) + \sqrt{ (1-s)^2 + 4a^2}}{2a}.$
	With this choice of $\varepsilon$ and the bound $s - a\varepsilon \geq \frac{1}{2}\sigma_r^2(X_\sharp) - 10\sqrt{\nu rk /d}$, the claimed bound \eqref{qg} follows immediately.
	The bound~\eqref{up} follows from the reverse triangle inequality:
	\begin{align*}
	| F((X,S))-F_Y((X,S)) |
	&\le \|XX^\top - Y Y^\top-  (X-Y) Y^\top- Y^\top(X-Y) \|_F \\
	& = \| X X^\top -X Y^\top - Y  X^\top +Y Y^\top \|_F \\
	& = \| (X-Y)  (X-Y)^\top\|_F \\
	&\le \| X-Y \|_F^2.
	\end{align*}

	Finally observe
	\begin{align*}
	|F( (X_1,S_1))-F((X_2,S_2))|&\leq \|X_1X_1^\top-X_2X_2^\top\|_F+\|S_1-S_2\|_F\\
	&\leq \|X_1 + X_2\|_{\mathrm{op}}\|X_1-X_2\|_F + \|S_1-S_2\|_F\\
	&\leq 2\sqrt{\nu r}\|X_1-X_2\|_F + \|S_1-S_2\|_F,
	\end{align*}
	where we use the bound $\|X_i\|_\op \leq \sqrt{d} \|X_i\|_{2, \infty} \leq \sqrt{\nu r}$ in the final inequality. The proof is complete.
\end{proof}

To summarize, there exist numerical constants $c_0,c_1, c_2>0$ such that the following is true. In the regime
$$
\sqrt{\frac{\nu rk}{d}} \leq c_0\sigma_r^2(X_\sharp),
$$
the Polyak subgradient method will converge at the  linear rate,
$$\dist(X_k,\cD^*(M_{\sharp}))\leq \left(1-\frac{c_1\sigma_r^2(X_\sharp)}{\nu r}\right)^{\frac{k}{2}}\cdot c_2\mu,$$ and the prox-linear method will converge quadratically when initialized at $X_0\in\cX$ satisfying $\dist(X_0,\cD^*(M_{\sharp}))< c_2\sigma_r(X_\sharp)$.

\subsection{The non-Euclidean formulation}
\label{sec:l1_robust_pca}

We next turn to a different formulation for robust PCA that does not require knowledge of  $\ell_1$ row norms of  $S_\sharp$.  In particular, we consider the formulation
\begin{equation}\label{eqn:l1_formul}
\min_{X \in \cX}~ f(X)=\|XX^\top-W\|_1 \quad \text{ where $\cX = \{X \in \RR^{d \times r} \mid \|X\|_{2, \infty} \leq C\|X_\sharp\|_{2, \infty}\}$} ,
\end{equation}
for a constant $C > 1$.
Unlike Section~\ref{sec:l2_robust_pca}, here we consider a randomized model for
the sparse matrix $ S_{\sharp} $.
We assume that there are real $\nu,\tau >0$ such that
\begin{enumerate}
	\item $M_{\sharp}\in\R^{d\times d}$ can be factored as $M_{\sharp}=X_{\sharp}X^{\top}_{\sharp}$ for some matrix $X_{\sharp}\in\R^{d\times r}$ satisfying  $\|X_{\sharp}\|_{2,\infty}\leq \sqrt{\frac{\nu r}{d}}\|X_\sharp\|_{\rm op}$.
	\item We assume the random corruption model
	$$
	(S_\sharp)_{ij} = \delta_{ij} \hat S_{ij}  \qquad \forall i,j$$
	where $\delta_{ij}$ are i.i.d.\ Bernoulli  random variables with $\tau = \mathbb{P}(\delta_{ij} = 1)$ and $\hat S$ is an arbitrary and fixed $d\times d$ symmetric matrix.
\end{enumerate}

In this setting, the approximation function at $X$ is given by
\[f_X(Z) = \|XX-W+X(Z-X)^{\top}+(Z-X)X^\top \|_1.\]
We begin by computing an estimate of the approximation accuracy $|f(Z)-f_X(Z)|$.

\begin{lem}[Approximation accuracy]
	The estimate holds:
	$$|f(Z)-f_X(Z)|\leq \|Z-X\|_{2,1}^2\qquad \textrm{ for all }X,Z\in \R^{d\times r}.$$
\end{lem}
\begin{proof}
	As in the proof of Proposition~\ref{prop:approx_acc_sym}, we compute
	\begin{align*}
	|f(Z)-f_X(Z)|&= \Big| \|ZZ^\top-W\|_1-\|XX-W+X(Z-X)^{\top}+(Z-X)X^\top \|_1 \Big|\\
	&\leq \|(Z-X)(Z-X)^\top\|_1=\sum_{i,j} |e_i^\top(Z-X)(e_j^{\top}(Z-X))^{\top}|\\
	&\leq \sum_{i,j} \|e_i^\top(Z-X)\|_2\cdot \|e_j^{\top}(Z-X)\|_2= \|Z-X\|_{2,1}^2,
	\end{align*}
	thereby completing the argument.
\end{proof}
Notice that the error $|f(Z)-f_X(Z)|$ is bounded in terms of the non-Euclidean norm
$\|Z-X\|_{2,1}$. Thus, although in principle one may apply subgradient methods to the formulation~\eqref{eqn:l1_formul},  their convergence guarantees, which fundamentally relied on the Euclidean norm, would yield potentially overly pessimistic performance predictions. On the other hand, the convergence guarantees for the prox-linear method do not require the norm to be Euclidean. Indeed, the following is true, with a proof that is nearly identical as that of Theorem~\ref{thm:prox_lin_loc}.
\begin{thm}
	\label{thm:noneuclidean_prox_linear}
	Suppose that Assumption~\ref{ass:bb} holds where $\|\cdot\|$ is replaced by an arbitrary norm $\opnorm{\cdot}$. Choose any $\beta \geq \rho$ and set $\gamma:=\rho/\beta$ in Algorithm~\ref{alg:prox_lin}. Then Algorithm~\ref{alg:prox_lin} initialized at any point $x_0$ satisfying $\dist_{\opnorm{\cdot}}(x_0,\cX^*)< \frac{\mu}{\rho}$ converges quadratically:
	$$\dist_{\opnorm{\cdot}}(x_{k+1},\cX^*)\leq
	\tfrac{\rho}{\mu}\cdot\dist^2_{\opnorm{\cdot}}(x_{k},\cX^*)\qquad \forall
	k\geq 0.$$
\end{thm}

To apply the above generic convergence guarantees for the prox-linear method,
it remains to show that the objective function $f$ in~\eqref{eqn:l1_formul} is
sharp relative to the norm $\|\cdot\|_{1,2}$. A key step in showing such a
result is to prove that
$$
\|XX^\top - X_{\sharp}X_{\sharp}^\top\|_1 \geq c \cdot  \inf_{R^\top R = I} \|X - X_{\sharp}R\|_{2, 1}
$$
for a quantity $c$ depending only on $X_{\sharp}$. One may prove this
inequality using Proposition~\ref{prop:l2sharpness} together with the
equivalence of the norms $ \| \cdot \|_F $ and $ \| \cdot \|_{1,2} $. Doing so
however leads to a dimension-dependent $c$, resulting in a poor rate of
convergence and region of attraction. We instead seek to directly establish
sharpness relative to the norm $\|\cdot\|_{2,1}$. In the rank one setting, this
can be done using the following theorem.
\begin{thm}[Sharpness (rank one)]\label{sharp_gen_case}
	Consider two vectors $x,\bar x\in \R^d$ satisfying $$\dist_{\|\cdot\|_1}(x,\{\pm\bar x\})\leq (\sqrt{2}-1)\|\bar x\|_1.$$ Then the estimate holds:
	$$\|xx^\top-\bar x\bar x^\top\|_1\geq (\sqrt{2} - 1)\cdot\|\bar x\|_1\cdot \dist_{\|\cdot\|_1}(x,\{\pm\bar x\}).$$
\end{thm}
The proof of this result appears in Appendix~\ref{appendix:proof_sharpness_r1}. We leave as an intriguing open question to determine if an analogous result holds in the higher rank setting.

\begin{conjecture}[Sharpness (general rank)]\label{conj:sharp_ell1}
	Fix a rank $r$ matrix $X_\sharp \in \RR^{d\times r}$ and set $\cD^*:=\{X\in\cX:XX^\top=X_\sharp X_\sharp^\top\}$. Then there exist constants $c,\gamma>0$ depending only on $X_\sharp$ such that the estimate holds:
	$$\|XX^\top-M\|_1\geq c\cdot\dist_{\|\cdot\|_{2,1}}(X,\cD^*),$$
	for all $X\in\cX$ satisfying $\dist_{\|\cdot\|_{2,1}}(X,\cD^*)\leq \gamma$.
\end{conjecture}

Assuming this conjecture, we can then show that the loss function $f$ is sharp under the randomized corruption model. We first state the following technical lemma, whose proof is deferred to Appendix \ref{appendix:proof_technical_lemma}. In what what follows, given a matrix $X \in \RR^{d \times r}$, the notation $X_i$ always refers to the $i$th row of $X$.

\begin{lem}\label{lem:bound_outlier_term}
	Assume Conjecture~\ref{conj:sharp_ell1}. Then there exist
	constants $c_1, c_2, c_3 > 0$ so that for all $d$ satisfying $d \geq
	\frac{c_1\log d}{\tau}$, we have that with probability $1-d^{-c_2}$, the
	following bound holds:
	\begin{align*}
	\sum_{i, j=1}^d \delta_{ij}|\dotp{X_i,X_j} - \dotp{ (X_\sharp)_i, (X_\sharp)_j}| &\leq \left( \tau +  \frac{c_3C \sqrt{\tau \nu r \log d }}{c} \| X_\sharp\|_{\op}\right)\|XX^\top - X_\sharp  X_\sharp^\top \|_1
	\end{align*}
	for all $X \in \cX$ satisfying
	$\dist_{\|\cdot\|_{2,1}}(X,\cD^*)\leq \gamma$.
\end{lem}

We remark that we expect $c$ to scale with $\| X_\sharp\|_{\op}$ in the above bound, yielding a ratio $\| X_\sharp\|_{\op}/c$ dependent on the conditioning of $X_\sharp$. Given the above lemma, sharpness of $f$ quickly follows.

\begin{lem}[Sharpness of Non-Euclidean Robust PCA]\label{lem:low_bound_dirty}
	Assume Conjecture~\ref{conj:sharp_ell1}. Then there exists a
	constants $c_1, c_2, c_3 > 0$ so that for all $d$ satisfying $d \geq
	\frac{c_1\log d}{\tau}$, we have that with probability $1-d^{-c_2}$, the
	following bound holds:
	\begin{align*}
	f(X) - f( X_\sharp) \geq c\cdot \left( 1 - 2\tau- \frac{2c_3C \sqrt{\tau \nu r\log d }}{c} \| X_\sharp\|_{op}\right) \cdot \dist_{\|\cdot\|_{2,1}}(X,\cD^*(M_{\sharp}))
	\end{align*}
	for all $X \in \cX$ satisfying and $\dist_{\|\cdot\|_{2,1}}(X,\cD^*(M_{\sharp}))\leq \gamma$.
\end{lem}
\begin{proof}
	The reverse triangle inequality implies that
	\begin{align*}
	&f(X) - f( X_\sharp) \\
	&= \|X  X^\top - W\|_1- f(X_\sharp)  \\
	&= \|X  X^\top - X_\sharp X_\sharp^\top\|_1 \\
	&\hspace{20pt}+ \sum_{i,j=1}^d\delta_{ij}\left(|\dotp{X_i,X_j} - \dotp{ (X_\sharp)_i, (X_\sharp)_j} - (S_\sharp)_{ij}| - |\dotp{X_i,X_j} - \dotp{ (X_\sharp)_i, (X_\sharp)_j} |  \right) - f(X_\sharp)  \\
	&= \|X  X^\top - X_\sharp X_\sharp^\top\|_1 \\
	&\hspace{20pt}+ \sum_{i,j=1}^d\delta_{ij}\left( |\dotp{X_i,X_j} - \dotp{(X_\sharp)_i,(X_\sharp)_j} - (S_\sharp)_{ij}| - |\dotp{X_i,X_j} - \dotp{(X_\sharp)_i,(X_\sharp)_j} | - |(S_\sharp)_{ij}| \right)\\
	&\ge \|X  X^\top - X_\sharp X_\sharp^\top\|_1 -2\sum_{i,j=1}^d \delta_{ij}
	|\dotp{X_i,X_j} - \dotp{(X_\sharp)_i,(X_\sharp)_j} |.
	\end{align*}
	The result them follows from the the sharpness of the function $\|XX^\top - X_\sharp X_\sharp^\top\|_1$ together with Lemma~\ref{lem:bound_outlier_term}.
\end{proof}

Combining Lemma~\ref{lem:low_bound_dirty} and Theorem~\ref{thm:noneuclidean_prox_linear}, we deduce the following convergence guarantee.
\begin{thm}[Convergence for non-Euclidean Robust PCA]
	Assume Conjecture~\ref{conj:sharp_ell1}. Then there exist  constants $c_1, c_2, c_3 > 0$ so that for all $\tau$ satisfying $1 - 2\tau- 2c_3C \sqrt{\tau \nu r \log d } \| X_\sharp\|_{op}/c > 0$ and $d$ satisfying $d \geq \frac{c_1\log d}{\tau}$, we have that with probability $1-d^{-c_2}$, the iterates generated by the prox-linear algorithm
	\begin{equation}\label{eqn:prox_lin_robust_pca}
	X_{k+1} = \argmin_{x \in \cX} \left\{ f_{X_k}(X) + \frac{1}{2\gamma} \|X - X_k\|_{2, 1}^2 \right\}
	\end{equation}
	satisfy
	$$
	\dist_{\|\cdot\|_{2,1}}(X_{k+1},\mathcal{D}^*(M_{\sharp})) \leq \frac{2}{c\cdot \left( 1 -
		2\tau- \frac{2c_3C \sqrt{\tau \nu r\log d }}{c} \| X_\sharp\|_{op}\right)}
	\cdot \dist_{\|\cdot\|_{2,1}}^2(X_k,\mathcal{D}^*(M_{\sharp})), \qquad \forall k\ge 0.
	$$
	In particular, the iterates converge quadratically as long as the initial iterate $ X_0\in\cX $ satisfies
	$$\dist_{\|\cdot\|_{2,1}}(X_{0},\mathcal{D}^*(M_{\sharp}))<  \min\left\{(1/2)c\cdot \left( 1 - 2\tau- \frac{2c_3C \sqrt{\tau \nu r \log d }}{c} \| X_\sharp\|_{op}\right), \gamma\right\}.$$
\end{thm}

\section{Recovery up to a Tolerance}\label{sec:recovery_tol}

Thus far, we have developed exact recovery guarantees under noiseless or sparsely corrupted measurements. We showed that sharpness together with weak convexity imply rapid local convergence of numerical methods under these settings. In practical scenarios, however, it might be unlikely that any, let alone a constant fraction of measurements, are perfectly observed. Instead, a more realistic model incorporates additive errors that are the sum of a sparse, but otherwise arbitrary vector and a dense vector with relatively small norm. Exact recovery is in general not possible under this noise model. Instead, we should only expect to recover the signal up to an error.

To develop algorithms for this scenario, we need only observe that the previously developed sharpness results all yield a corresponding ``sharpness up to a tolerance" result. Indeed, all problems considered thus far, are convex composite and sharp:
$$
\min_{x \in \cX} f(x) := h(F(x)) \qquad \text{ and } \qquad f(x) - \inf_\cX f \geq \mu \cdot \dist(x, \cX^\ast),
$$
where $h$ is convex and $\eta$-Lipschitz with respect to some norm $\opnorm{\cdot}$,  $F$ is a smooth map, and $\mu > 0$. Now consider a fixed additive error vector $e$, and the perturbed problem
\begin{equation}\label{eqn:perturbed}
\min_{x \in \cX}~ \tilde f(x) := h(F(x) + e).
\end{equation}
The triangle inequality immediately implies that the perturbed problem is sharp up to tolerance $2\eta\opnorm{e}$:
$$\qquad \tilde f(x) - \inf_\cX \tilde f \geq \mu \cdot \dist(x, \cX^\ast) - 2\eta\opnorm{e}\qquad \forall x\in \cX.$$
In particular, any minimizer $x^*$ of $\tilde f$ satisfies
\begin{align}\label{eq:min_tol}
\dist(x^\ast, \cX^\ast) \leq (2\eta/\mu)\opnorm{e},
\end{align}
where as before we set $\cX^*=\argmin_{\cX} f$. In this section, we show that subgradient and prox-linear algorithms applied to the perturbed problem \eqref{eqn:perturbed} converge rapidly up to a tolerance on the order of $\eta\opnorm{e}/\mu$.
To see the generality of the above approach, we note that even the robust recovery problems considered in Section~\ref{subsec:sharpwithnoise}, in which a constant fraction of measurements are already corrupted, may be further corrupted through additive error vector $e$. We will study this problem in detail in Section~\ref{sec:ell1ell2RIP_sparse_dense}.

Throughout the rest of the section, let us define the noise level:
$$
\varepsilon := \eta\opnorm{e}.
$$
Mirroring the discussion in Section~\ref{sec:conv_guarant}, define the annulus:
\begin{align}\label{eq:tube_region_tol}
\widetilde \cT_\gamma := \left\{ z \in \cX : \frac{14\varepsilon}{\mu} < \dist(z, \cX^\ast) <  \frac{\gamma \mu}{4\rho}\right\},
\end{align}
for some $\gamma > 0$. Note that for the annulus $\widetilde \cT_\gamma$ to be nonempty, we must ensure $\epsilon<\frac{\mu^2\gamma}{56\rho}$.
We will see that $\widetilde \cT_\gamma$ serves as a region of rapid convergence for some numerical constant $\gamma$.
As before, we also define subgradient bound and the condition measure:
\begin{equation*}
\tilde L:=\sup\{\|\zeta\|_2:\zeta\in \partial \tilde f(x),x\in \widetilde \cT_1\}\qquad \textrm{and} \qquad \tilde  \tau:=\mu/\tilde L.
\end{equation*}
In all examples considered in the paper, it is possible to show directly that $\tilde L \leq L$ as defined in Assumption~\ref{ass:aa}. A similar result is true in the general case, as well.
Indeed, the following Lemma provides a bound for $\tilde L$ in terms of the subgradients of $f$ on a slight expansion of the tube $\cT_1$ from~\eqref{eq:tube_region}; the proof  appears in the appendix. 
\begin{lem}\label{lem:lip_tol}
	Suppose $\epsilon<\frac{\mu^2}{56\rho}$ so that $\widetilde \cT_1$ is nonempty. Then the following bound holds:
	$$
	\tilde L \leq \sup\left\{\|\zeta\|_2:\zeta\in \partial  f(x), \dist(x, \cX^\ast) \leq \frac{\mu}{\rho}, \dist(x, \cX) \leq 2\sqrt{\frac{\varepsilon}{\rho}}\right\} +  2\sqrt{8\rho\varepsilon}.
	$$
\end{lem}

We will now design algorithms whose basin of attraction is the annulus $\widetilde \cT_\gamma$ for some $\gamma$. To that end, the following modified sharpness bound will be useful for us. The reader should be careful to note the appearance of $\inf_{\cX} f$, not $\inf_\cX \tilde f$ in the following bound.
\begin{lem}[Approximate sharpness]\label{lem:sharp_on_tube_2}
	We have the following bound:
	\begin{align}\label{eq:sharp_on_tube_2}
	\tilde f(x) - \inf_\cX f \geq \mu \cdot \dist(x, \cX^\ast) - \varepsilon \qquad \forall x \in \cX.
	\end{align}
\end{lem}
\begin{proof}
	For any $x \in \cX$, observe
	$
	\tilde f(x) - \inf f \geq  f(x) - \inf f - \varepsilon \geq \mu \cdot \dist(x, \cX^\ast) - \varepsilon,
	$
	as claimed.
\end{proof}

Next, we show that $\tilde f$ satisfies the following approximate subgradient inequality.
\begin{lem}[Approximate subgradient inequality]\label{lem:approximate_subgradient}
	The following bound holds:
	\begin{align*}
	f(y)\geq \tilde f(x)+\dotp{\zeta, y - x} -\frac{\rho}{2}\|x - y\|^2-3\varepsilon   \qquad \forall x, y \text{ and } \zeta \in \partial \tilde f(x).
	\end{align*}
\end{lem}
\begin{proof}
	First notice that $|f_x(y) - \tilde f_x(y)| \leq \varepsilon$ for all $x, y$.  Furthermore, we have $\partial \tilde f(x) = \nabla F(x)^* \partial h(F(x) + e) = \partial \tilde f_x(x)$. Therefore, it follows that for any $\zeta \in \partial \tilde f_x(x)$  we have
	\begin{align*}
	\dotp{\zeta, y - x} &\leq \tilde f_x(y) - \tilde f_x(x) \\
	&\leq f_x(y) - f_x(x) + 2\eta\opnorm{e}\\
	& \leq f(y) - f(x) + \frac{\rho}{2} \|x - y\|^2 + 2\varepsilon \\
	& \leq  f(y) - \tilde f(x) + \frac{\rho}{2} \|x - y\|^2 + 3\varepsilon,
	\end{align*}
	as desired.
\end{proof}

Now consider the following modified Polyak method. It is important to note that the stepsize assumes knowledge of $\min_{\cX} f$ rather than $\min_{\cX} \tilde f$. This distinction is important because it often happens that $\min_{\cX} f = 0$, whereas $\min_{\cX} \tilde f$ is in general unknown; for example, consider any noiseless problem analyzed thus far. We note that the standard Polyak subgradient method may also be applied to $\tilde f$ without any changes and has similar theoretical guarantees. The proof appears in the appendix.

\bigskip
\begin{algorithm}[H]
	\KwData{$x_0 \in \RR^d$}

	{\bf Step $k$:} ($k\geq 0$)\\
	$\qquad$ Choose $\zeta_k \in \partial \tilde f(x_k)$. {\bf If} $\zeta_k=0$, then exit algorithm.\\
	$\qquad$ Set $\displaystyle x_{k+1}=\proj_{\cX}\left(x_{k} - \frac{\tilde f(x_k)-\min_{\mathcal{X}} f}{\|\zeta_k\|^2}\zeta_k\right)$.
	\caption{Modified Polyak Subgradient Method}
	\label{alg:polyak_tol}
\end{algorithm}


\begin{thm}[Polyak subgradient method]\label{thm:qlinear_tol}
	Suppose that Assumption~\ref{ass:aa} holds and suppose that $\varepsilon \leq \mu^2/14\rho$.
	Then  Algorithm~\ref{alg:polyak_tol} initialized at any point $x_0\in \widetilde \cT_1$ produces iterates that converge $Q$-linearly to $\cX^*$ up to tolerance $14\varepsilon/\mu$, that is
	\begin{equation*}
	\dist^2(x_{k+1},\cX^*) \leq \left(1-\frac{13\tilde \tau^2}{56}\right)\dist^2(x_{k},\cX^*)\qquad \forall k\geq 0 \text{ with $\dist(x_k, \cX^\ast) \geq 14\varepsilon/\mu$}.
	\end{equation*}
\end{thm}

Next we provide theoretical guarantees for Algorithm~\ref{thm:geometric}, where one does not know the optimal value $\min_\cX f$. The proof of this result is a straightforward modification of \cite[Theorem 6.1]{davis2018subgradient} based on the Lemmas~\ref{lem:sharp_on_tube_2} and~\ref{lem:approximate_subgradient}, and therefore we omit it.

\begin{thm}[Geometrically decaying subgradient method] \label{thm:geometric_tol} Suppose that Assumption~\ref{ass:aa} holds, fix a real number  $\gamma \in (0,1)$,  and suppose  $\tilde \tau  \le \frac{14}{11}\sqrt{ \frac{1}{2-\gamma} }$. Suppose also $\epsilon<\frac{\mu^2\gamma}{56\rho}$ so that $\widetilde \cT_\gamma$ is nonempty.
	Set $\lambda:=\frac{\gamma \mu^2}{4\rho \tilde L} \textrm{ and } q:=\sqrt{1-(1-\gamma) \tilde\tau^2} $ in Algorithm~\ref{alg:geometrically_step}.
	Then the iterates $x_k$ generated by Algorithm~\ref{alg:geometrically_step} on the perturbed problem \eqref{eqn:perturbed}, initialized at a point $x_0 \in \tilde{\mathcal{T}}_{\gamma}$, satisfy:
	\begin{equation*} 
	\dist^2(x_k;\cX^*) \leq \frac{\gamma^2 \mu^2}{16\rho^2}
	\left(1-(1-\gamma)\tilde \tau^2\right)^{k}\qquad \forall k\geq 0 \text{ with $\dist(x_k, \cX^\ast) \geq 14\varepsilon/\mu$}.
	\end{equation*}
\end{thm}

Finally, we analyze the prox-linear algorithm applied to the problem $\min_\cX \tilde f$. In contrast to the Polyak method, one does not need to know the optimal value $\min_\cX f$. The proof appears in the appendix.
\begin{thm}[Prox-linear algorithm]\label{thm:prox_linear_tol}
	Suppose Assumptions~\ref{ass:bb} holds. Choose any $\beta \geq \rho$ in Algorithm~\ref{alg:prox_lin} applied to the perturbed problem \eqref{eqn:perturbed} and set $\gamma:=\rho/\beta$. Suppose moreover $\epsilon<\frac{\mu^2\gamma}{56\rho}$ so that $\tilde{\mathcal T}_{\gamma}$ is nonempty. Then Algorithm~\ref{alg:prox_lin} initialized at any point $x_0 \in \widetilde \cT_{\gamma}$ converges quadratically up to tolerance $14\varepsilon/\mu$:
	$$\dist(x_{k+1},\cX^*)\leq  \tfrac{7\beta}{6\mu}\cdot\dist^2(x_{k},\cX^*)\qquad \forall k\geq 0 \text{ with $\dist(x_{k+1}, \cX^\ast) \geq 14\varepsilon/\mu$}.$$
\end{thm}

\subsection{Example: sparse outliers and dense noise under $\ell_1/\ell_2$ RIP}\label{sec:ell1ell2RIP_sparse_dense}

To further illustrate the ideas of this section, we now generalize the results of Section~\ref{subsec:sharpwithnoise}, in particular Assumption~\ref{assump:outlier}, to the following observation model.
\begin{assumption}[$\cI$-outlier bounds]\label{assump:outlier_tol}
	There exists vectors $e, \Delta \in \RR^m$, a set $\cI \subset\{1, \ldots, m\}$, and a constant $\ir > 0$ such that the following hold.
	\begin{enumerate}[label = $\mathrm{(C\arabic*)}$]
		\item $b = \cA(M_{\sharp}) + \Delta + e$.
		\item Equality holds $\Delta_i = 0 $ for all $i\notin \cI$.
		\item \label{item:assump:rip_outliers2}  For all matrices $W$ of rank at most $2r$, we have
		\begin{equation*} 
		\ir \|W\|_F \leq  \frac{1}{m}\|\cA_{\mathcal{I}^c}(W)\|_1 -  \frac{1}{m}\|\cA_{\cI}(W)\|_1.
		\end{equation*}
	\end{enumerate}
\end{assumption}

Given these assumptions we follow the notation of the previous section and let
$$
f(X) := \frac{1}{m}\| \cA(XX^\top - M_\sharp) - \Delta\|_1 \qquad \text{and} \qquad \tilde f(X) = \frac{1}{m}\|\cA(XX^\top - M_\sharp) - \Delta - e\|_1.
$$
Then we have the following proposition:
\begin{proposition}\label{prop:appl_inexact_rip}
	Suppose Assumption~\ref{assump:RIP} and~\ref{assump:outlier_tol} are valid. Then the following hold:
	\begin{enumerate}
		\item {\bf (Sharpness)} We have
		\begin{align*}
		f(X) - f(X_{\sharp})  \geq \mu\cdot  \dist \big(X, \cD^\ast(M_{\sharp})\big) \qquad \text{for all $X\in\R^{d\times r}$ and $\mu:= \ir\sqrt{2(\sqrt{2}-1)}\sigma_r(X_{\sharp})$},
		\end{align*}
		\item {\bf (Weak Convexity)} The function $f$ is $\rho := 2\kappa_2$-weakly convex.
		\item {\bf (Minimizers)} All minimizers of $\tilde f$ satisfy
		$$\dist(X^\ast, \cX^\ast) \leq \frac{2\frac{1}{m}\|e\|_1}{\ir\sqrt{2(\sqrt{2}-1)}\sigma_r(X_{\sharp})} \qquad \forall X^\ast \in \argmin_{\cX} \tilde f .$$
		\item {\bf (Lipschitz Bound)} We have the bound
		$$
		\tilde L \leq 2\ur\cdot \left(\frac{\ir\sqrt{2(\sqrt{2}-1)}\sigma_r(X_{\sharp})}{8\kappa_2} + \sigma_1(X_\sharp)\right).
		$$
	\end{enumerate}
\end{proposition}
\begin{proof}
	Sharpness follows from Proposition~\ref{prop:noisy_sharp}, while weak convexity follows from Proposition~\ref{prop:approx_acc_sym}. The minimizer bound follows from~\eqref{eq:min_tol}. Finally, due to Lemma~\ref{lem:Lipschitz_subgradient}, the argument given in Proposition~\eqref{prop:approx_acc_sym}, but applied instead to $\tilde f$, guarantees that
	$$
	\tilde L \leq 2\kappa_2\cdot \sup\left\{ \|X\|_{op} \colon \dist(X, \cD^\ast(M_{\sharp})) \leq \frac{\ir\sqrt{2(\sqrt{2}-1)}\sigma_r(X_{\sharp})}{8\kappa_2}\right\}.
	$$
	In turn the supremum may be bounded as follows: Let $X_\star = X_\sharp R$ denote the closest point to $X$ in $\cD^\ast(M)$. Then
	\begin{align*}
	\|X\|_\op \leq \|X - X_\sharp R\|_\op + \|X_\sharp R\|_\op \leq \frac{\ir\sqrt{2(\sqrt{2}-1)}\sigma_r(X_{\sharp})}{8\kappa_2} + \sigma_1(X_\sharp),
	\end{align*}
	as desired.
\end{proof}

In particular, combining Proposition~\ref{prop:appl_inexact_rip} with the previous results in this section, we deduce the following. As long as the noise satisfies
$$\frac{1}{m}\|e\|_1\leq \frac{c_0\ir^2\sigma_r(M_{\sharp})}{\ur}$$
for a sufficiently small constant $c_0>0$,  the subgradient and prox-linear methods converge rapidly to within tolerance
$$\delta\approx \frac{\frac{1}{m}\|e\|_1}{\ir\sigma_r(X_{\sharp})},$$
when initialized at a matrix $X_0$ satisfying
$$\frac{\dist(X_0,\cD^\ast(M_\sharp))}{\sqrt{\sigma_r(M_{\sharp})}}\leq c_1\cdot\frac{\ir}{\ur},$$
for some small constant $c_1$. The formal statement is summarized in the following corollary.

\begin{corollary}[Convergence guarantees  under RIP with sparse outliers and dense noise (symmetric)]\label{cor:generic_conv_iso_sym_tol}
	Suppose Assumptions~\ref{assump:RIP} is and \ref{assump:outlier_tol} are valid with $\opnorm{\cdot}=\frac{1}{m}\|\cdot\|_1$ and define the condition number $\chi=\sigma_1(M_{\sharp})/\sigma_r(M_{\sharp})$.
	Then there exists numerical constants $c_0, c_1, c_2, c_3, c_4, c_5, c_6 > 0$ such that the following hold.
	Suppose the noise level satisfies
	$$\frac{1}{m}\|e\|_1\leq \frac{2(\sqrt{2}-1)c_0\ir^2\sigma_r(M_{\sharp})}{28\ur}$$
	and define the tolerance
	$$
	\delta = \frac{\frac{14}{m} \|e\|_1}{\ir\sqrt{2(\sqrt{2}-1)\sigma_r(M_{\sharp})}}.
	$$ Then as long as the matrix $X_0$ satisfies
	$$\frac{\dist(X_0,\cD^\ast(M_\sharp))}{\sqrt{\sigma_r(M_{\sharp})}}\leq c_1\cdot\frac{\ir}{\ur},$$
	the following are true.
	\begin{enumerate}
		\item {\bf (Polyak subgradient)} Algorithm~\ref{alg:polyak} initialized at $X_0$  produces iterates that
		converge linearly to $\cD^\ast(M_\sharp)$, that is
		\begin{equation*}
		\frac{\dist^2(X_k,\cD^\ast(M_\sharp))}{\sigma_r(M_{\sharp})}\leq \left(1-\frac{c_2}{1+\frac{c_3\ur^2\chi}{\ir^2}}\right)^{k}\cdot \frac{c_4\ir^2}{\ur^2}\qquad \forall k\geq 0 \text{ with $\dist(X_k, \cX^\ast) \geq \delta$}.
		\end{equation*}
		\item {\bf (geometric subgradient)}
		Algorithm~\ref{alg:geometrically_step} with
		$\lambda=\frac{c_5\ir^2\sqrt{\sigma_r(M_{\sharp})}}{\ur (\ir+2\ur\sqrt{\chi})}$, $q=\sqrt{1-\frac{c_2}{1+c_3\ur^2\chi/\ir^2}}$
		and initialized at $X_0$ converges linearly:
		\begin{equation*}
		\frac{\dist^2(X_k,\cD^\ast(M_\sharp))}{\sigma_r(M_{\sharp})}\leq \left(1-\frac{c_2}{1+\frac{c_3\ur^2\chi}{\ir^2}}\right)^{k}\cdot \frac{c_4\ir^2}{\ur^2}\qquad \forall k\geq 0 \text{ with $\dist(X_k, \cX^\ast) \geq \delta$}.
		\end{equation*}
		\item {\bf (prox-linear)}  Algorithm~\ref{alg:prox_lin} with $\beta = \rho$ and initialized at $X_0$ converges quadratically:
		$$\frac{\dist(X_k,\cD^\ast(M_\sharp)))}{\sqrt{\sigma_r(M_{\sharp})}}\leq 2^{-2^{k}}\cdot \frac{c_6\ir}{\ur}\qquad \forall k\geq 0 \text{ with $\dist(X_k, \cX^\ast) \geq \delta$}.$$
	\end{enumerate}
\end{corollary}

\section{Numerical Experiments}\label{sec:num_exp}

In this section, we demonstrate the theory and algorithms developed in the previous sections on a number of low-rank matrix recovery problems, namely  quadratic and bilinear sensing, low rank
matrix completion, and robust PCA.

\subsection{Robustness to outliers} \label{subsec:phase-transitions}
In our first set of experiments, we empirically test the robustness of our optimization methods to outlying measurements. We generate \textit{phase transition plots}, where each
pixel corresponds to the empirical probability of successful recovery over $50$
test runs using randomly generated problem instances. Brighter pixels represent
higher recovery rates.
All generated instances obey the following:
\begin{enumerate}
	\item The initial estimate is specified reasonably close to
	the ground truth. In particular, given a target symmetric positive semidefinite matrix $X_{\sharp}$,
	we set
	\[
	X_0 := X_{\sharp} + \delta \cdot \norm{X_{\sharp}}_F \cdot \Delta,
	\quad\text{where~~} \Delta = \frac{G}{\norm{G}_F}, \;
	G_{ij} \sim_{\mathrm{i.i.d.}} N(0, I).
	\]
	Here, $\delta$ is a scalar that controls the quality of initialization
	and $\Delta$ is a random unit ``direction''. The asymmetric setting is completely analogous.
	\item When using the subgradient method with geometrically decreasing
	step-size, we set $\lambda = 1.0, \; q = 0.98$.
	\item For the quadratic sensing, bilinear sensing, and matrix completion problems, we mark a test run as a success when the normalized distance $\|M -M_{\sharp}\|_F/\|M_{\sharp}\|_F$ is less than $10^{-5}$. Here we set $M=XX^\top$ in the symmetric setting and $M=XY$ in the asymmetric setting. For the robust PCA problem, we stop when $\|M -M_{\sharp}\|_1/\|M_{\sharp}\|_1< 10^{-5}$.
\end{enumerate}
Moreover, we set the seed of the random number generator at the beginning of each batch of experiments to enable reproducibility.

\paragraph{Quadratic and Bilinear sensing.}
Figures~\ref{fig:bilin-phase-tr} and \ref{fig:symquad-phase-tr} depict the
phase transition plots for bilinear \eqref{eq:bilinear_measurements} and symmetrized quadratic \eqref{eq:quadratic_measurements_II} sensing
formulations using Gaussian measurement vectors. In the experiments, we corrupt a fraction of
measurements with additive Gaussian noise of unit entrywise variance.
Empirically, we observe that increasing the variance of the additive
noise does not affect recovery rates.
Both problems exhibit a sharp
phase transition at very similar scales. Moreover, increasing the rank of the
generating signal does not seem to dramatically affect the recovery rate for
either problem. Under additive noise, we can recover the true signal (up to
natural ambiguity) even if we corrupt as much as half of the measurements.
\begin{figure}[h!]
	\centering
	\includegraphics[width=0.7\linewidth]{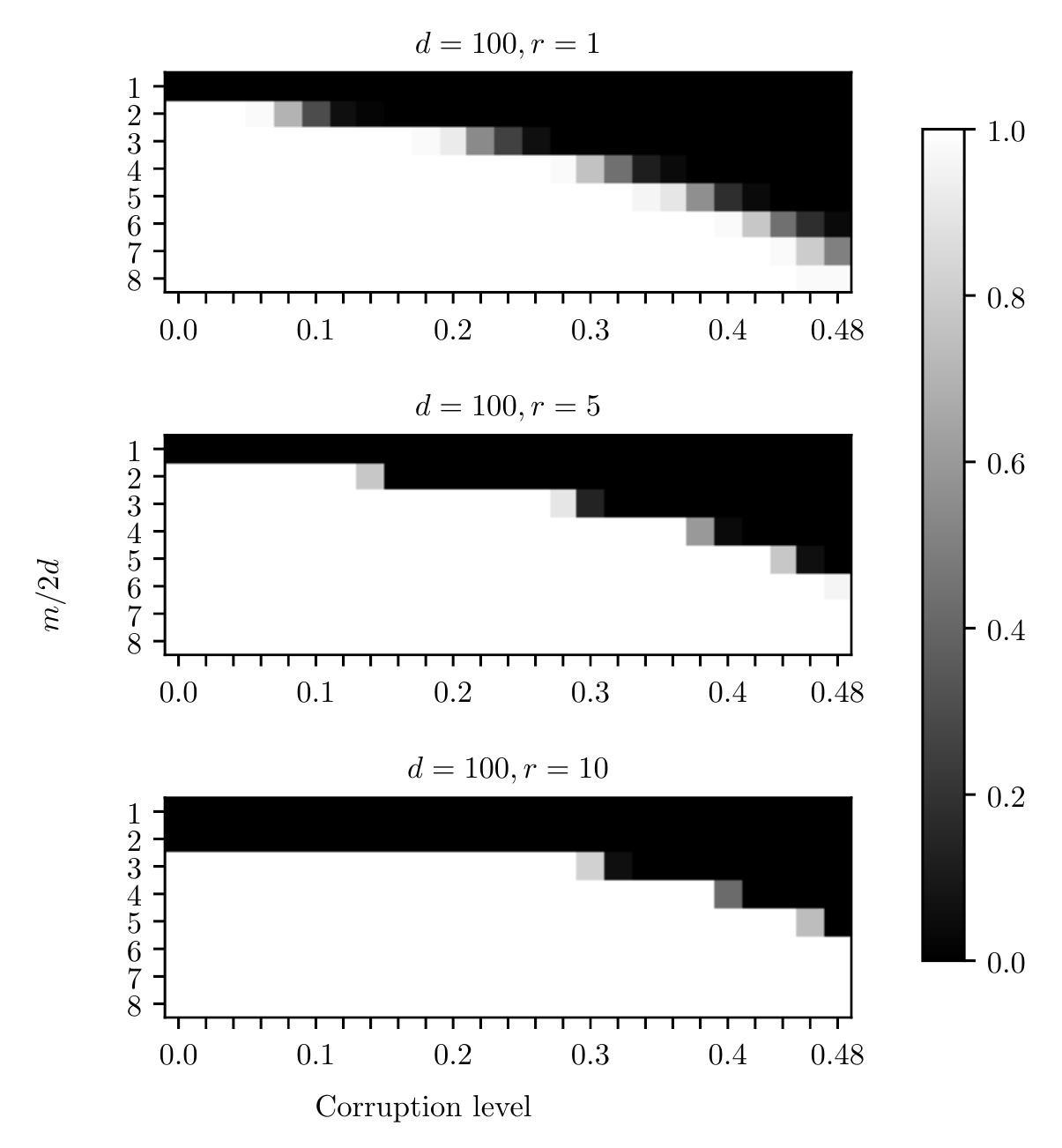}
	\caption{Bilinear sensing with $d_1 = d_2 = d = 100$ using
		Algorithm~\ref{alg:geometrically_step}.}
	\label{fig:bilin-phase-tr}
\end{figure}

\begin{figure}[h!]
	\centering
	\includegraphics[width=0.7\linewidth]{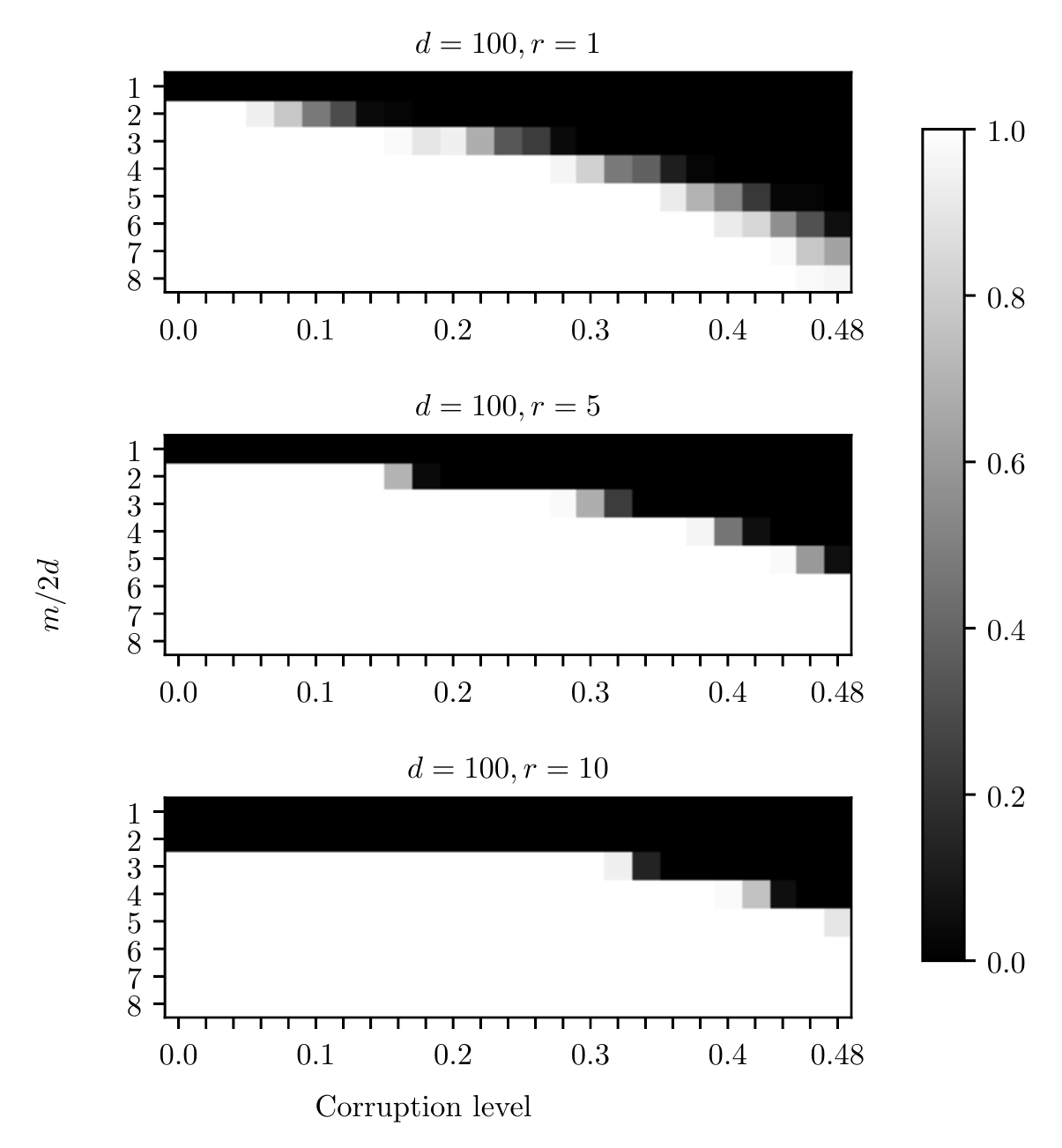}
	\caption{Quadratic sensing with symmetrized measurements using
		Algorithm~\ref{alg:geometrically_step}.}
	\label{fig:symquad-phase-tr}
\end{figure}

\paragraph{Robust PCA.}
We generate robust PCA instances for $d = 80$ and $r \in \{1, 2, 4, 8, 16\}$. The
corruption matrix $S_{\sharp}$ follows the assumptions in
Section~\ref{sec:l1_robust_pca}, where for simplicity we set $\hat{S}_{ij} \sim
\mathsf{N}(0, \sigma^2)$. We observed that increasing or decreasing the variance
$\sigma^2$ did not affect the probability of successful recovery, so our
experiments use $\sigma = 1$. We use the subgradient method, Algorithm~\ref{alg:prox_lin}, and the prox-linear algorithm \eqref{eqn:prox_lin_robust_pca}. Notice that we have not presented any guarantees for the subgradient method on this problem, in contrast to the prox-linear method. The subproblems for the prox-linear method are solved by ADMM with
graph splitting as in~\cite{PariBoyd14}. We set tolerance $\epsilon_k =
\frac{10^{-4}}{2k}$ for the proximal subproblems, which we continue solve for
at most $500$ iterations. We choose $\gamma = 10$ in all subproblems.
The phase transition plots are shown in Figure~\ref{fig:rpca-subgrad-phase-tr}. It appears that the
prox-linear method is more robust to additive sparse corruption, since the
empirical recovery rate for the subgradient method decays faster as the rank increases.

\begin{figure}[h!]
	\centering
	\includegraphics[width=0.6\linewidth]{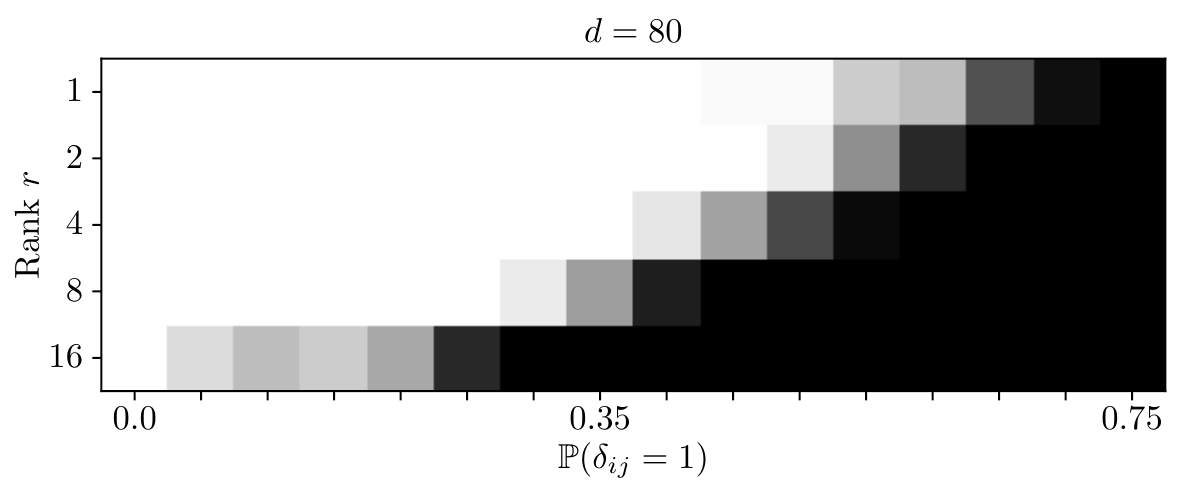}
	\includegraphics[width=0.6\linewidth]{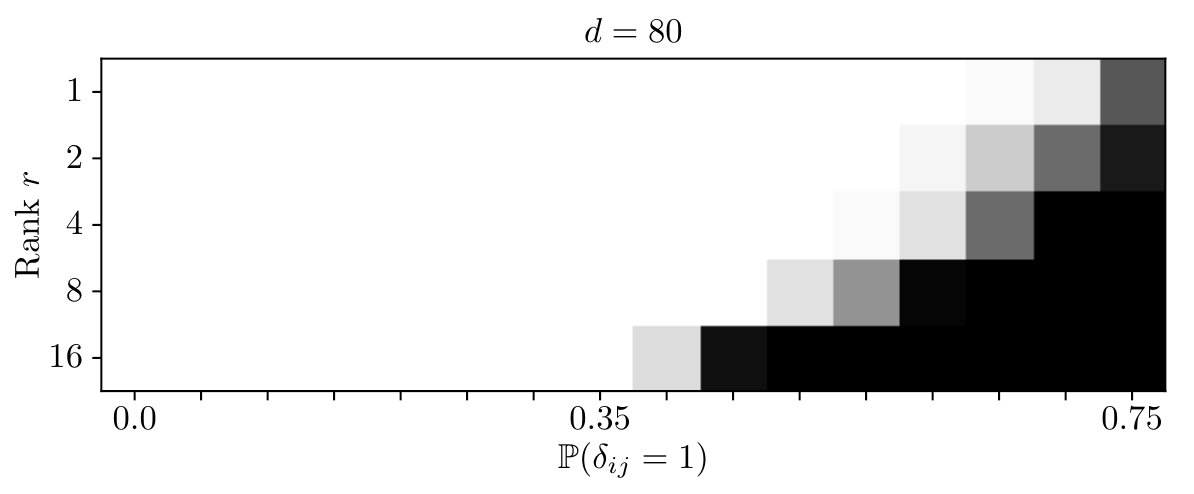}
	\caption{Robust PCA using the subgradient method, Algorithm~\ref{alg:geometrically_step}, (top) and the modified prox-linear method~\eqref{eqn:prox_lin_robust_pca} (bottom).}
	\label{fig:rpca-subgrad-phase-tr}
\end{figure}

\paragraph{Matrix completion.}
We next perform experiments on the low-rank matrix completion problem that test successful recovery against the sampling frequency. We generate random instances of the problem, where we let the probability of observing an entry, $\PP(\delta_{ij} = 1)$, range in
$[0.02, 0.6]$ with increments of $0.02$.
Figure~\ref{fig:matcomp-proxlin-phase-tr} depicts the empirical recovery rate using the Polyak subgradient method and the modified prox-linear algorithm \eqref{eqn:prox_lin_mat_comp}. Similarly to the quadratic/bilinear sensing problems, low-rank matrix completion exhibits
a sharp phase transition. As predicted in
Section~\ref{sec:mat_comp}, the ratio $\frac{r^2}{d}$ appears to be driving the
required observation probability for successful recovery. Finally, we
empirically observe that the prox-linear method can ``tolerate'' slightly
smaller sampling frequencies.

\begin{figure}[h!]
	\centering
	\includegraphics[width=0.7\linewidth]{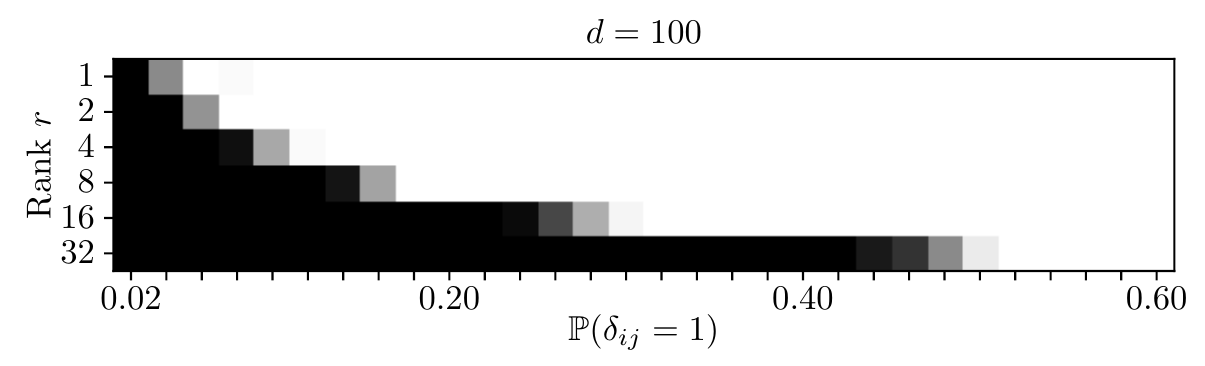}
	\includegraphics[width=0.7\linewidth]{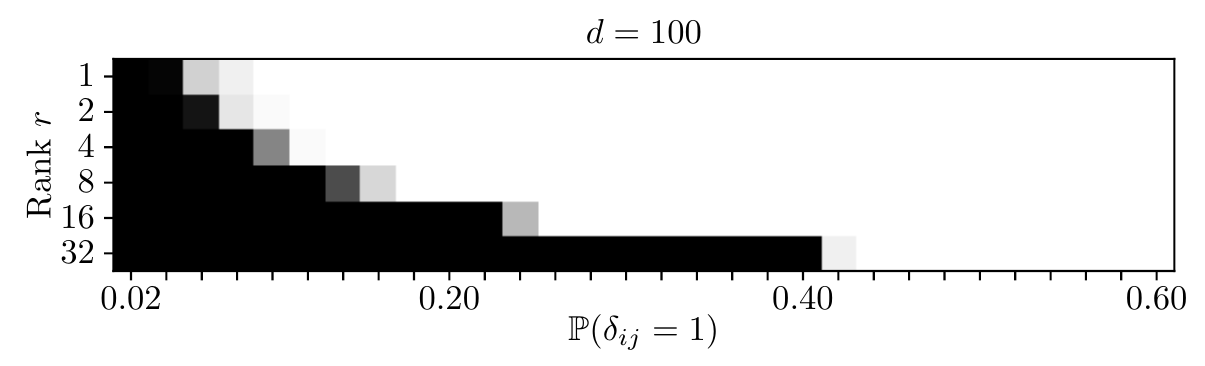}
	\caption{Low-rank matrix completion using the subgradient method, Algorithm~\ref{alg:polyak} (top), and the modified prox-linear method~\eqref{eqn:prox_lin_mat_comp} (bottom).}
	\label{fig:matcomp-proxlin-phase-tr}
\end{figure}

\subsection{Convergence behavior} \label{subsec:convergence-behavior}
We empirically validate the rapid convergence guarantees of the subgradient and prox-linear methods, given a proper initialization.
Moreover, we compare the subgradient method with gradient descent,
i.e. gradient descent applied to a smooth formulation of each problem, using
the same initial estimate in the noiseless setting. In all the cases below, the
step sizes for the gradient method were tuned for best performance. Moreover, we
noticed that the gradient descent method, equipped with the Polyak step size $\eta := \tau \frac{\nabla
	f}{\norm{\nabla f}^2}$ performed at least as well as gradient descent with constant step size. That being said, we were unable to locate any
theoretical guarantees in the literature for gradient descent with the Polyak step-size for the problems we consider here.

\paragraph{Quadratic and Bilinear sensing.}
For the quadratic and bilinear sensing problems, we apply gradient descent on the smooth formulations
\[
\frac{1}{m} \norm{\cA(XX^\top) - b}_2^2 \quad \text{ and } \quad
\frac{1}{m} \norm{\cA(XY) - b}_2^2.
\]
In Figure~\ref{fig:mtxsense-conv}, we plot the performance of
Algorithm~\ref{alg:geometrically_step} for matrix sensing problems with
different rank / corruption level; remarkably, the level of noise does not
significantly affect the rate of convergence. Additionally, the convergence
behavior is almost identical for the two problems for similar rank/noise
configurations.
Figure~\ref{fig:mtxsense-compgd} depicts the behavior of
Algorithm~\ref{alg:polyak} versus  gradient descent with empirically
tuned step sizes. The subgradient method significantly
outperforms gradient descent. For completeness, we also depict the convergence
rate of Algorithm~\ref{alg:prox_lin} for both problems in
Figure~\ref{fig:mtxsense-proxlin}, where we solve the proximal subproblems
approximately.

\begin{figure}[h]
	\begin{minipage}{0.45 \textwidth}
		\includegraphics[width=\linewidth]{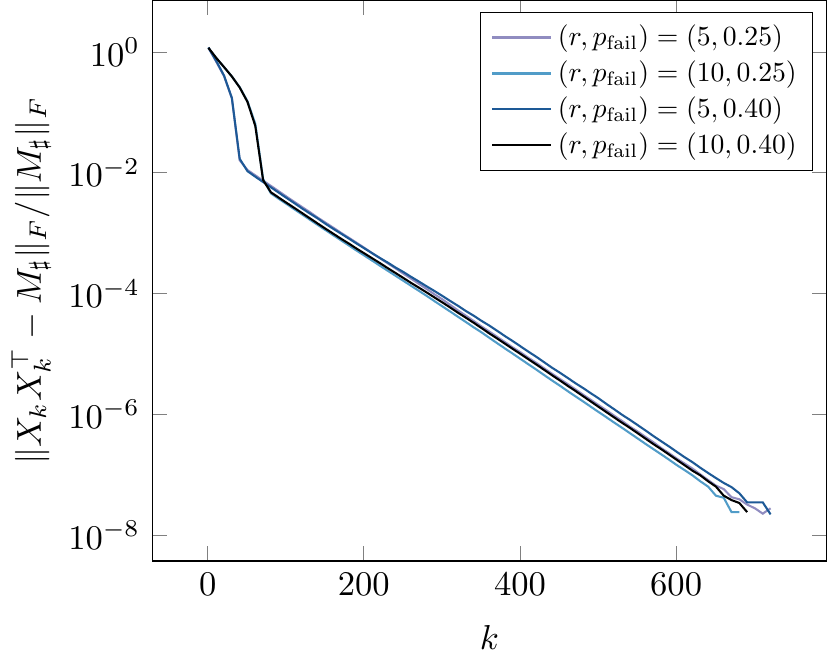}
	\end{minipage}
	\quad
	\begin{minipage}{0.45 \textwidth}
		\includegraphics[width=\linewidth]{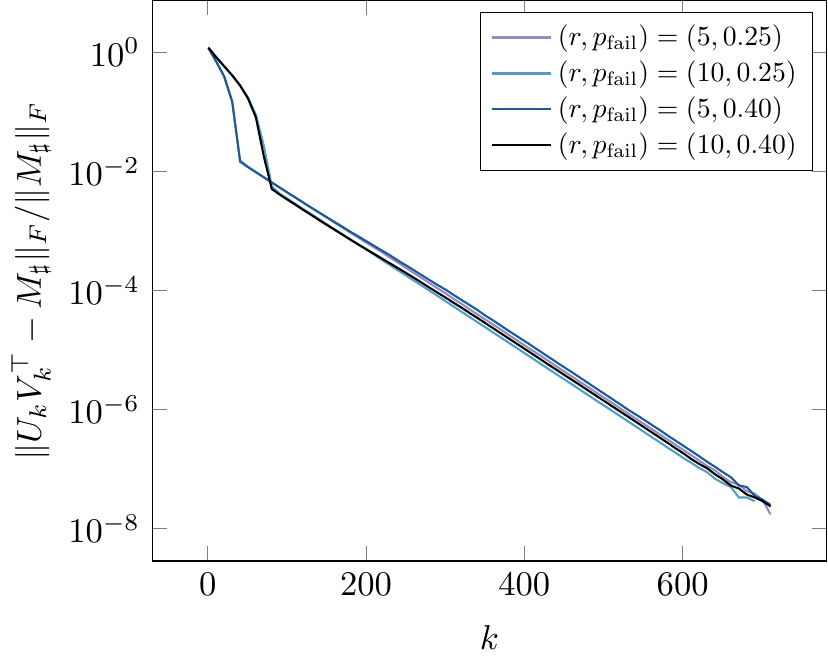}
	\end{minipage}
	\caption{Quadratic (left) and bilinear (right) matrix sensing with $d = 200, m = 8 \cdot rd$, using the subgradient method,
		Algorithm~\ref{alg:geometrically_step}.}
	\label{fig:mtxsense-conv}
\end{figure}

\begin{figure}[h]
	\begin{minipage}{0.45 \textwidth}
		\includegraphics[width=\linewidth]{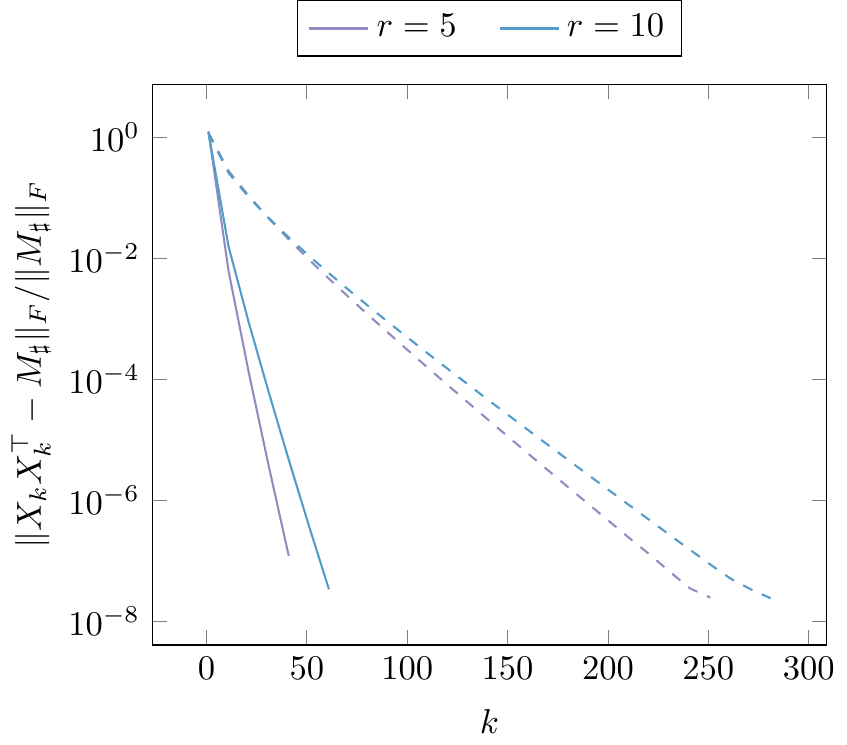}
	\end{minipage}
	\quad
	\begin{minipage}{0.45 \textwidth}
		\includegraphics[width=\linewidth]{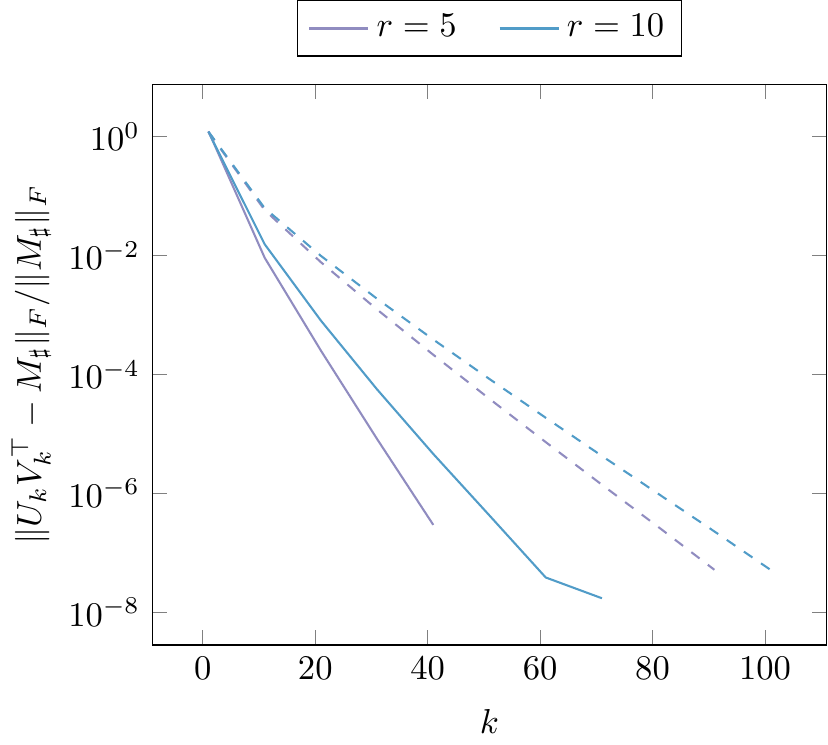}
	\end{minipage}
	\caption{
		Algorithm~\ref{alg:polyak} (solid lines) against gradient descent
		(dashed lines) with step size $\eta$. Left: quadratic sensing, $\eta =
		10^{-4}$. Right: bilinear sensing, $\eta = 10^{-3}$.}
	\label{fig:mtxsense-compgd}
\end{figure}

\begin{figure}[h]
	\begin{minipage}{0.45 \textwidth}
		\includegraphics[width=\linewidth]{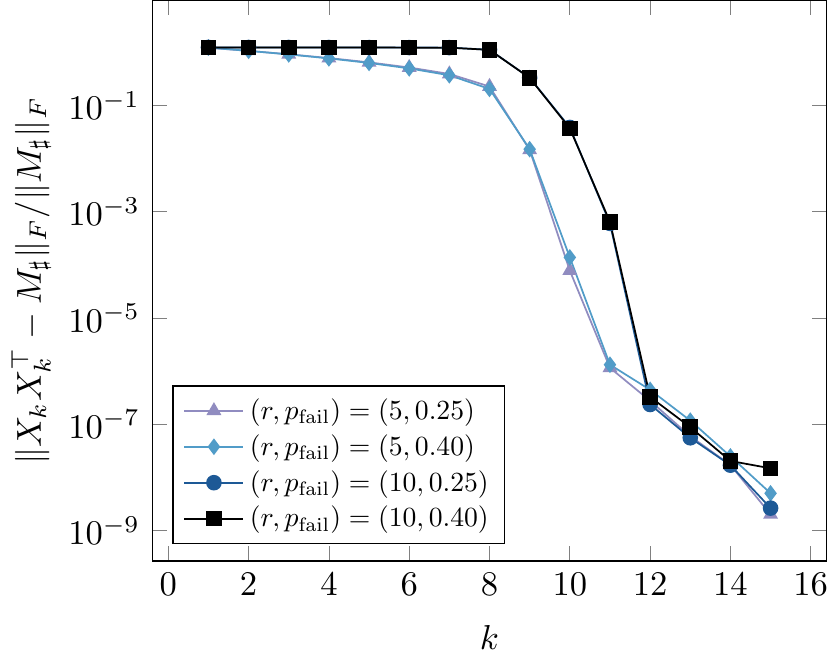}
	\end{minipage}
	\quad
	\begin{minipage}{0.45 \textwidth}
		\includegraphics[width=\linewidth]{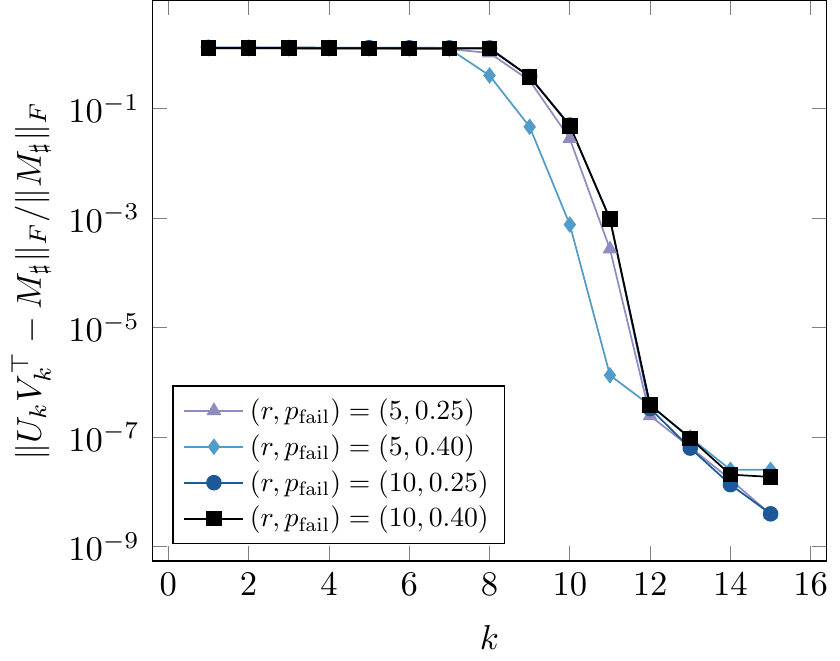}
	\end{minipage}
	\caption{
		Quadratic (left) and bilinear (right) matrix sensing with $d = 100,
		m = 8 \cdot rd$, using the prox-linear method,
		Algorithm~\ref{alg:prox_lin}.}
	\label{fig:mtxsense-proxlin}
\end{figure}

\paragraph{Matrix completion.}
In our comparison with smooth methods, we apply gradient descent on the following minimization
problem:
\begin{equation}
\min_{X \in \RR^{d \times r} : \norm{X}_{2, \infty} \leq C}
\norm{\Pi_{\Omega}(XX^\top) - \Pi_{\Omega}(M)}_F^2.
\label{eq:matcomp_gd}
\end{equation}
Figure~\ref{fig:matcomp-polyak-conv} depicts the convergence behavior of
Algorithm~\ref{alg:polyak} (solid lines) versus  gradient descent applied
to Problem~\eqref{eq:matcomp_gd} with a tuned step size $\eta = 0.004$ (dashed lines),
initialized under the same conditions for low-rank matrix completion instances.
As the theory suggests, higher
sampling frequency implies better convergence rates.  The subgradient method outperforms  gradient descent in all regimes.

\begin{figure}[ht]
	\begin{minipage}{0.45 \textwidth}
		\includegraphics[width=\linewidth]
		{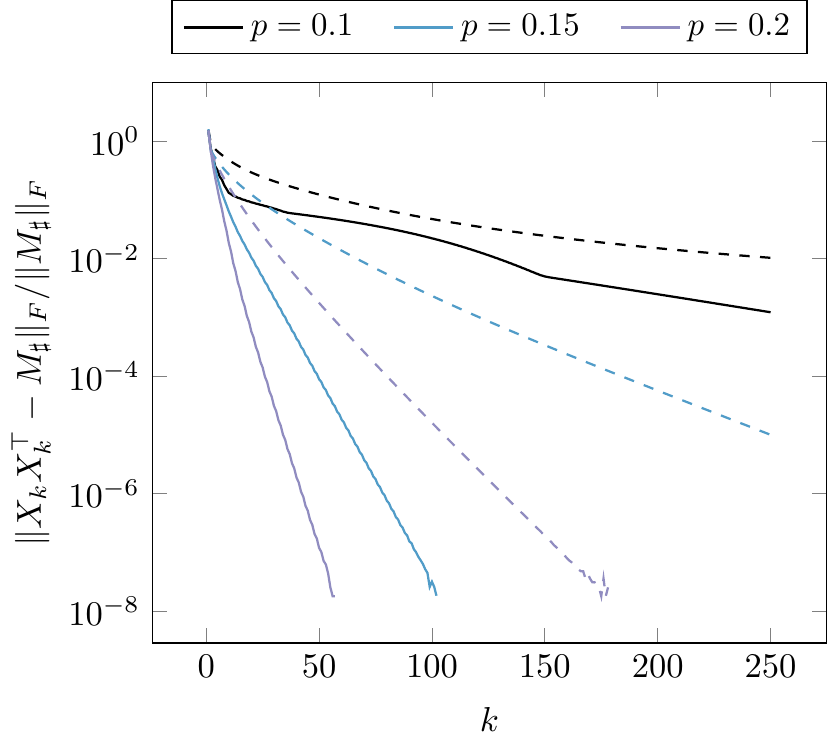}
	\end{minipage}
	\quad
	\begin{minipage}{0.45 \textwidth}
		\includegraphics[width=\linewidth]
		{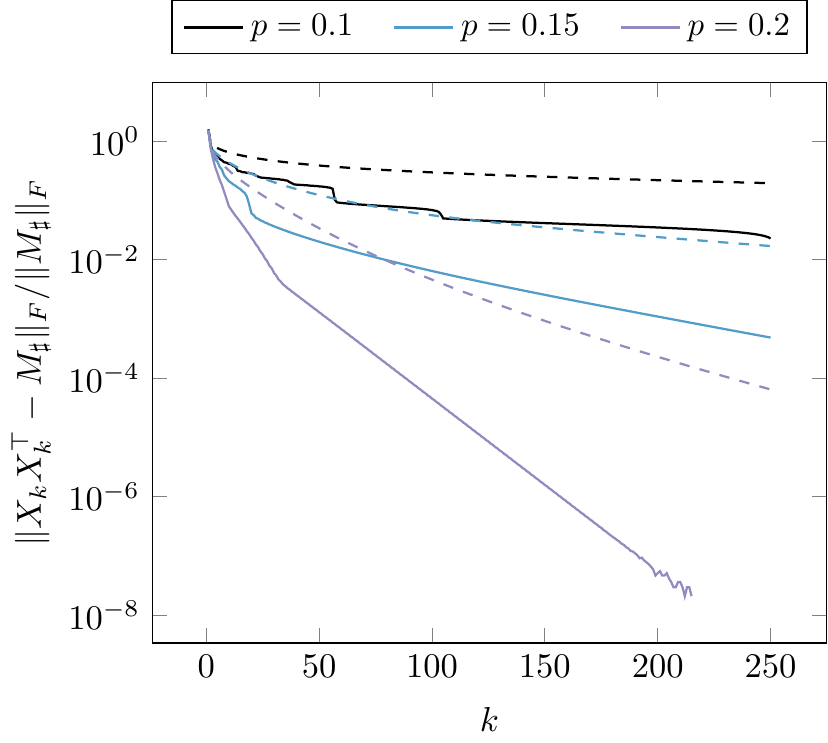}
	\end{minipage}
	\caption{Low rank matrix completion with $d = 100$. Left: $r = 4$, right:
		$r = 8$. Solid lines use Algorithm~\ref{alg:polyak}, dashed lines use
		gradient descent with step $\eta = 0.004$.}
	\label{fig:matcomp-polyak-conv}
\end{figure}

Figure~\ref{fig:matcomp-proxlin-conv} depicts the performance of
the modified prox-linear method \eqref{eqn:prox_lin_mat_comp} in the same setting as
Figure~\ref{fig:matcomp-polyak-conv}. In most cases, the prox-linear algorithm
converges within just $15$ iterations, at what appears to be a rapid linear rate of convergence. Each convex
subproblem is solved using a variant of the
graph-splitting ADMM algorithm~\cite{PariBoyd14}.

\begin{figure}[ht]
	\begin{minipage}{0.45 \textwidth}
		\includegraphics[width=\linewidth]
		{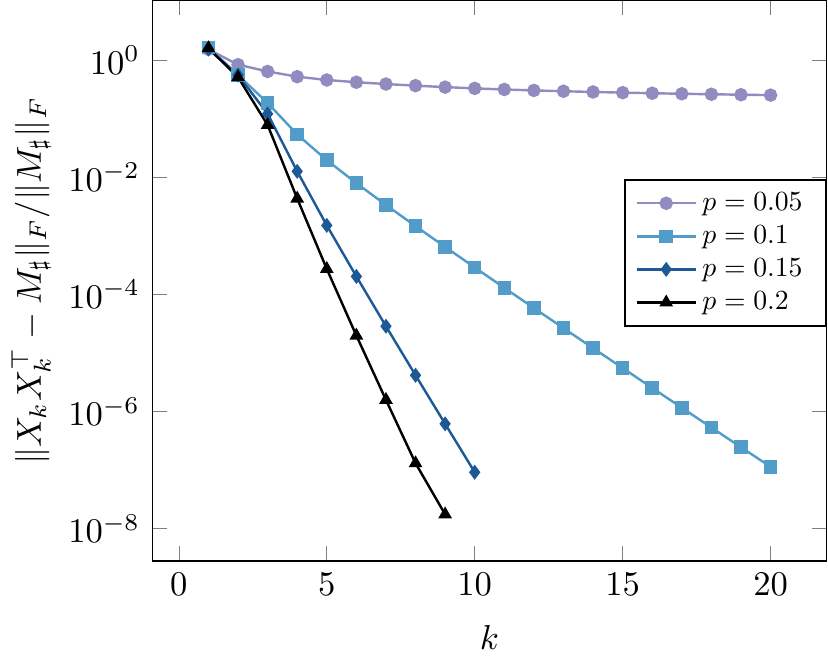}
	\end{minipage}
	\qquad
	\begin{minipage}{0.45 \textwidth}
		\includegraphics[width=\linewidth]
		{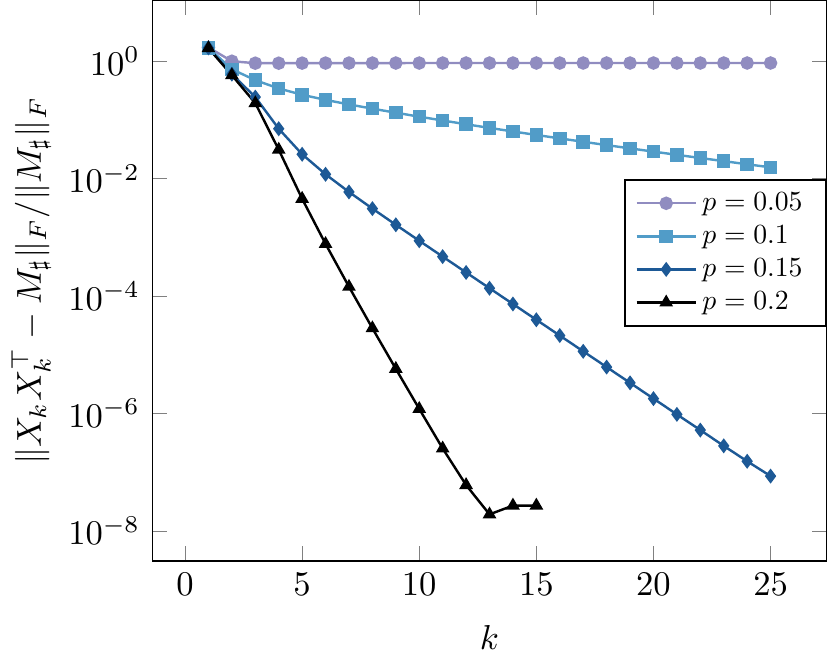}
	\end{minipage}
	\caption{Low rank matrix completion with $d =
		100$ using the modified prox-linear Algorithm \eqref{eqn:prox_lin_mat_comp}. Left: $r = 4$, right: $r = 8$.}
	\label{fig:matcomp-proxlin-conv}
\end{figure}

\paragraph{Robust PCA.}
For the  robust PCA problem, we consider different rank/corruption level configurations to
better understand how they affect convergence for the subgradient and
prox-linear methods, using the non-Euclidean formulation of
Section~\ref{sec:l1_robust_pca}. We depict all configurations in the same plot
for a fixed optimization algorithm to better demonstrate the effect of each
parameter, as shown in Figure~\ref{fig:rpca-conv}. The parameters of the
prox-linear method are chosen in the same way reported in
Section~\ref{subsec:phase-transitions}.  In particular, our numerical experiments appear to support our sharpness Conjecture~\ref{conj:sharp_ell1}  for the robust PCA problem.

\begin{figure}[ht]
	\begin{minipage}{0.45 \textwidth}
		\includegraphics[width=\linewidth]
		{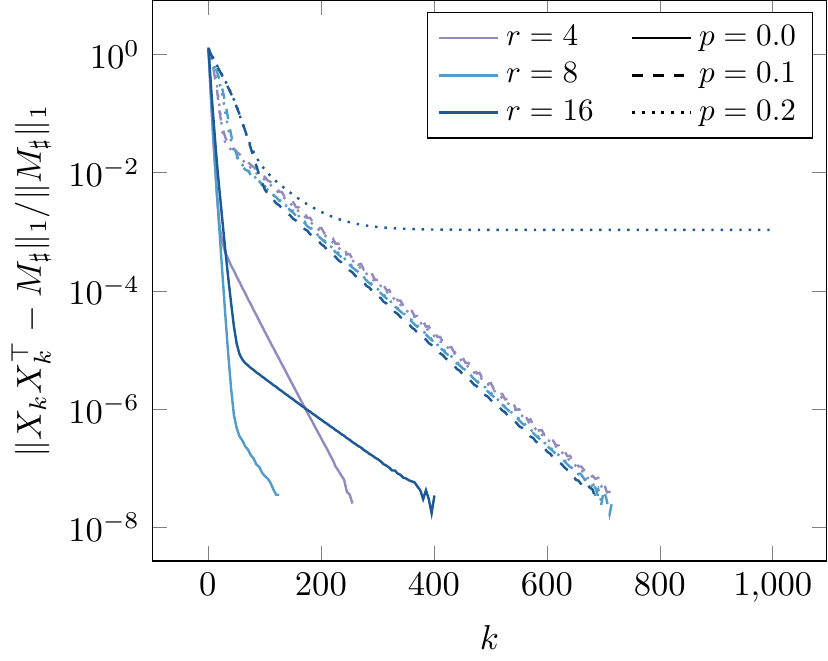}
	\end{minipage}
	\qquad
	\begin{minipage}{0.45 \textwidth}
		\includegraphics[width=\linewidth]
		{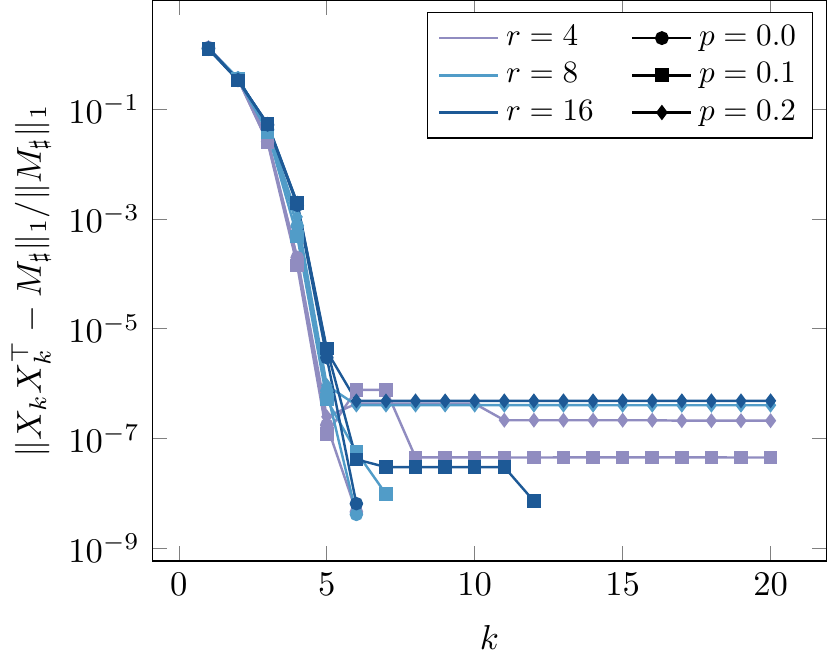}
	\end{minipage}
	\caption{$\ell_1$-robust PCA with $d = 100$ and $p :=
		\mathbb{P}(\delta_{ij} = 1)$. Left: Algorithm~\ref{alg:geometrically_step},
		right:
		Algorithm~\eqref{eqn:prox_lin_mat_comp}.}
	\label{fig:rpca-conv}
\end{figure}

\subsubsection{Recovery up to tolerance}
In this last section, we test the performance of the prox-linear method and the modified Polyak
subgradient method (Algorithm~\ref{alg:polyak_tol}) for the quadratic sensing
and matrix completion problems, under a dense noise model of Section~\ref{sec:recovery_tol}.
In the former setting, we set $p_{\mathrm{fail}} = 0.25$, so 1/4th of our
measurements is corrupted with large magnitude noise. For matrix completion,
we observe $p = 25\% $ of the entries. In both settings, we add Gaussian noise
$e$ which is rescaled to satisfy
\(
\norm{e}_F = \delta \sigma_r(X_{\sharp}),
\)
and test $\delta := 10^{-k}\sigma_r(X_{\sharp}), \; k \in \{1, \dots, 4\}$.
The relevant plots can be found in Figures~\ref{fig:dense-subgrad}
and~\ref{fig:dense-proxlin}. The numerical experiments fully support the developed theory, with the iterates converging rapidly up to the tolerance that is proportional to the noise level.
Incidentally, we observe that the modified prox-linear method~\eqref{eqn:prox_lin_mat_comp} is more robust to
additive noise for the matrix completion problem, with
Algorithm~\ref{alg:polyak_tol} exhibiting heavy fluctuations and failing to
converge for the highest level of dense noise.

\begin{figure}[ht]
	\begin{minipage}{0.45 \textwidth}
		\includegraphics[width=\linewidth]{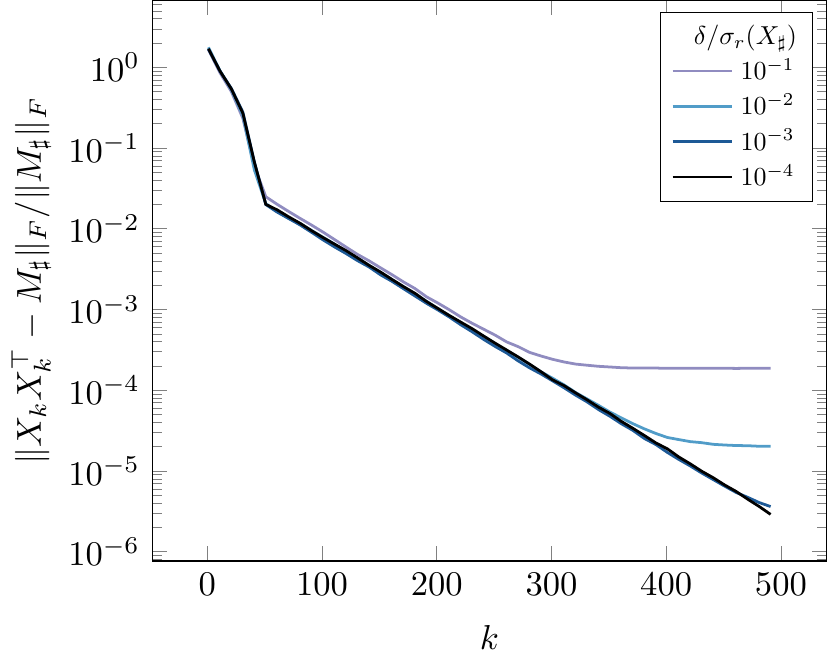}
	\end{minipage}
	\quad
	\begin{minipage}{0.45 \textwidth}
		\includegraphics[width=\linewidth]{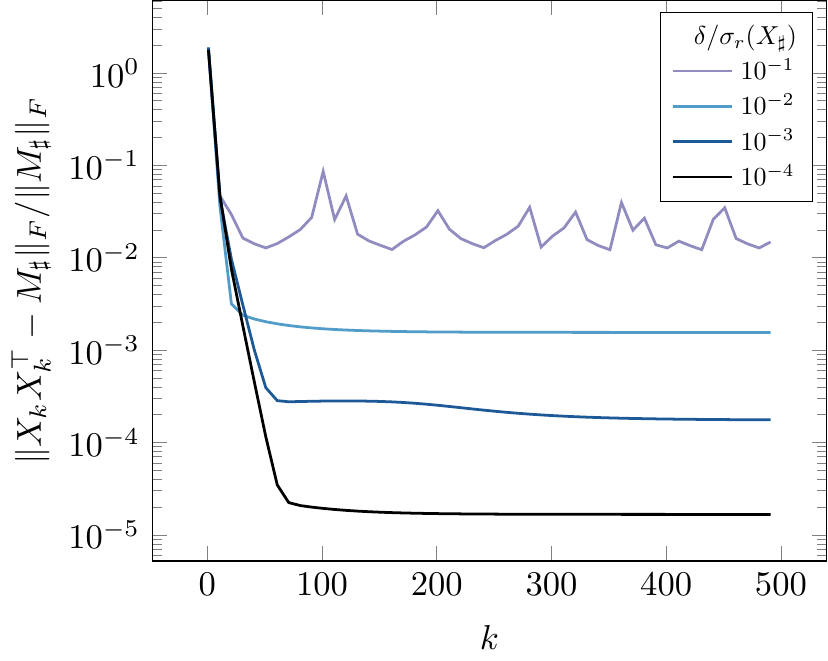}
	\end{minipage}
	\caption{Quadratic sensing with $r = 5$ (left) and matrix completion with
		$r = 8$ (right), $d = 100$, using Algorithm~\ref{alg:polyak_tol}.}
	\label{fig:dense-subgrad}
\end{figure}

\begin{figure}[ht]
	\begin{minipage}{0.45 \textwidth}
		\includegraphics[width=\linewidth]{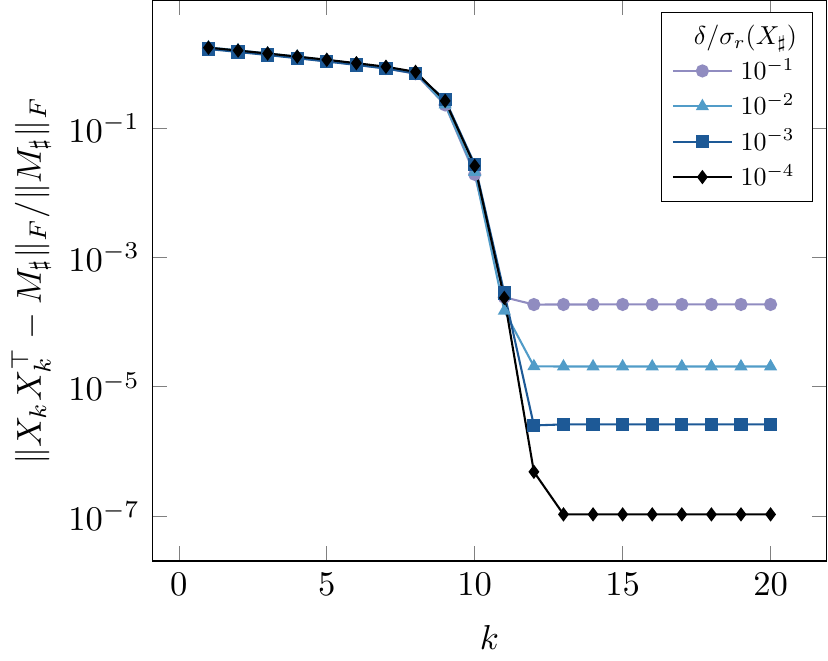}
	\end{minipage}
	\quad
	\begin{minipage}{0.45 \textwidth}
		\includegraphics[width=\linewidth]{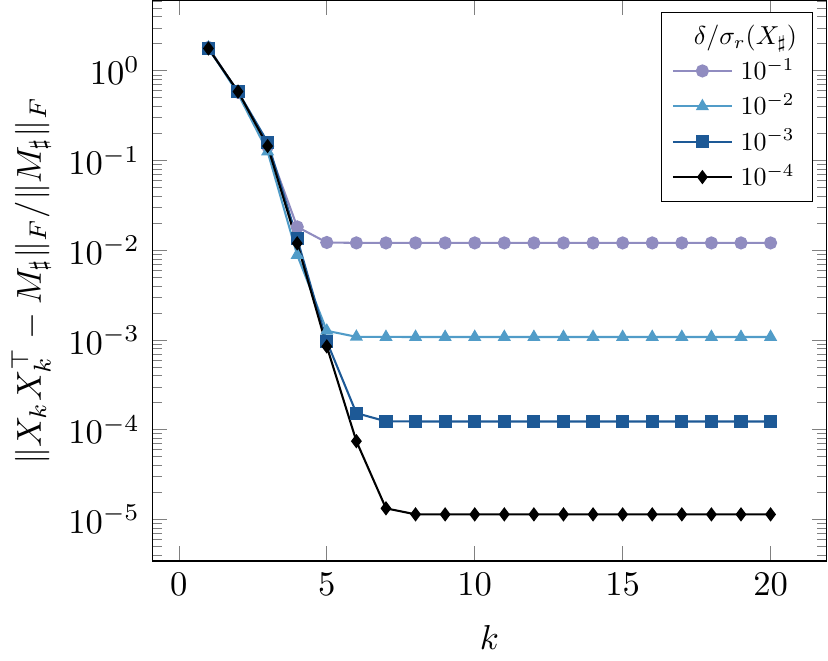}
	\end{minipage}
	\caption{Quadratic sensing with $r = 5$ (left) and matrix completion with
		$r = 8$ (right), $d = 100$, using Algorithm~\eqref{eqn:prox_lin_mat_comp}.}
	\label{fig:dense-proxlin}
\end{figure}

\bibliographystyle{plain}
\bibliography{bibliography}

\addtocontents{toc}{\protect\setcounter{tocdepth}{0}}
\appendix

\section{Proofs in Section~\ref{sec:conv_guarant}}\label{sec:app_convloc}

In this section, we prove rapid local convergence guarantees for the subgradient and prox-linear algorithms under regularity conditions that hold only locally around a particular solution. We will use the Euclidean norm throughout this section; therefore to simplify the notation, we will drop the subscript two. Thus $\|\cdot\|$ denotes the $\ell_2$ on a Euclidean space $\mathbf{E}$ throughout.

We will need the following quantitative version of Lemma~\ref{lem:no_extr_stat}.

\begin{lem}\label{lem:lower_bound_subgrad} Suppose Assumption~\ref{ass:target_weak_conv_loc} holds and let $\gamma \in (0,2)$ be arbitrary. Then for any point $x\in B_{\epsilon/2}(\bar x)\cap \mathcal{T}_{\gamma}\backslash \cX^\ast$, the estimate holds:
	$$\dist\left(0,\partial f(x)\right)\geq \left(1-\tfrac{\gamma}{2}\right)\mu.$$
\end{lem}
\begin{proof}
	Consider any point $x\in B_{\epsilon/2}(\bar x)$ satisfying $\dist(x,\cX^*)\leq \gamma \frac{\mu}{\rho}$. Let $x^*\in \proj_{\cX^*}(x)$ be arbitrary and note $x^*\in B_{\epsilon}(\bar x)$. Thus for any $\zeta\in  \partial f(x)$ we deduce
	$$\mu\cdot \dist(x,\cX^*)\leq f(x)-f(x^*)\leq \langle \zeta,x-x^*\rangle+\frac{\rho}{2}\|x-x^*\|^2\leq \|\zeta\|\dist(x,\cX^*)+\frac{\rho}{2}\dist^2(x,\cX^*).$$
	Therefore we deduce the lower bound on the subgradients
	$\|\zeta\|\geq \mu -\frac{\rho}{2}\cdot\dist(x,\cX^*)\geq\left(1-\tfrac{\gamma}{2}\right)\mu,$
	as claimed.
\end{proof}

\subsection{Proof of Theorem~\ref{thm:polyak_local}}
Let $k$  be the first index (possibly infinite) such that $x_k\notin B_{\epsilon/2}(\bar x)$.
We claim that \eqref{eqn:subgrad_lin_loca} holds for all $i<k$. We show this by induction. To this end, suppose \eqref{eqn:subgrad_lin_loca} holds for all indices up to $i-1$. In particular, we  deduce $\dist(x_{i},\cX^*)\leq \dist(x_{0},\cX^*)\leq \frac{\mu}{2\rho}$. Let $x^*\in \proj_{\cX^*}(x_i)$ and note $x^*\in B_{\epsilon}(\bar x)$, since
$$\|x^*-\bar x\|\leq \|x^*-x_i\|+\|x_i-\bar x\|\leq 2\|x_i-\bar x\|\leq \epsilon.$$ Thus we deduce
\begin{align}
\|x_{i+1} -  x^*\|^2&=\left\|\proj_{\cX}\left(x_{i}-\tfrac{f(x_i)-\min_{\mathcal{X}} f}{\|\zeta_i\|^2} \zeta_i\right)-\proj_\cX(x^*)\right\|^2\notag\\
&\leq \left\|(x_{i} - x^*)- \tfrac{f(x_i)-\min_{\mathcal{X}} f}{\|\zeta_i\|^2} \zeta_i\right\|^2 \label{eqn1:weird_subgrad}\\
&= \|x_{i} -  x^*\|^2 + \frac{2(f(x_i) - \min_{\cX} f)}{\|\zeta_i\|^2}\cdot\dotp{\zeta_i, x^* - x_{i}} + \frac{(f(x_i) - f( x^*))^2}{\|\zeta_i\|^2} \notag\\
&\leq \|x_{i} -  x^*\|^2 + \frac{2(f(x_i) -  \min f)}{\|\zeta_i\|^2}\left( f( x^*) - f(x_i) + \frac{\rho}{2}\|x_i -  x^*\|^2 \right)\notag \\
&\qquad\qquad+ \frac{(f(x_i) - f(x^*))^2}{\|\zeta_i\|^2} \label{eqn2:weird_subgrad}\\
&= \|x_{i} -  x^*\|^2 + \frac{f(x_i) - \min f}{\|\zeta_i\|^2}\left(\rho\|x_i -  x^*\|^2 -  (f(x_i) - f( x^*))  \right)\notag\\
&\leq \|x_{i} -  x^*\|^2 + \frac{f(x_i) -  \min f}{\|\zeta_i\|^2}\left(\rho\|x_i -  x^*\|^2 -  \mu\|x_i -  x^*\|  \right)\label{eqn3:weird_subgrad} \\
&= \|x_{i} -  x^*\|^2 + \frac{\rho (f(x_i) -  \min f)}{\|\zeta_i\|^2}\left( \|x_i -  x^*\| -\frac{\mu}{\rho}\right)\|x_i -  x^*\|\notag\\
&\leq \|x_{i} -  x^*\|^2 - \frac{\mu(f(x_i) -  \min f)}{2\|\zeta_i\|^2}\cdot \|x_i -  x^*\|\label{eqn35:weird_subgrad}\\
&\leq \left(1-\frac{ \mu^2}{2\|\zeta_i\|^2}\right)\|x_i- x^*\|^2\label{eqn4:weird_subgrad}.
\end{align}
Here, the estimate \eqref{eqn1:weird_subgrad} follows from the fact that the
projection $\proj_Q(\cdot)$ is nonexpansive, \eqref{eqn2:weird_subgrad} uses local weak convexity, \eqref{eqn35:weird_subgrad} follow from the estimate $\dist(x_i,\cX^*)\leq \frac{\mu}{2\rho}$, while \eqref{eqn3:weird_subgrad} and \eqref{eqn4:weird_subgrad} use local sharpness.
We therefore deduce
\begin{equation}\label{eqn:intermediate_descent}
\dist^2(x_{i+1};\cX^*)\leq \|x_{i+1} - x^*\|^2\leq \left(1-\frac{\mu^2}{2L^2}\right)\dist^2(x_i,\cX^*).
\end{equation}
Thus \eqref{eqn:subgrad_lin_loca} holds for all indices up to $k-1$. We next show that $k$ is infinite.
To this end, observe
\begin{align}
\|x_k-x_0\|\leq \sum_{i=0}^{k-1} \|x_{i+1}-x_i\|&= \sum_{i=0}^{k-1} \left\|\proj_\cX\left(x_i-\tfrac{f(x_i)-\min_{\cX} f}{\|\zeta_i\|^2}\zeta_i\right)-\proj_\cX(x_i)\right\|\notag\\
&\leq   \sum_{i=0}^{k-1} \frac{f(x_i)-\min_{\cX} f}{\|\zeta_i\|}\quad \notag \\
&\leq   \sum_{i=0}^{k-1} \left\langle \tfrac{\zeta_i}{\|\zeta_i\|},x_i-\proj_{\cX^*}
(x_i) \right\rangle+\frac{\rho}{2\|\zeta_i\|}\|x_i-\proj_{\cX^*}(x_i)\|^2\notag\\
&\leq   \sum_{i=0}^{k-1} \dist(x_i,\cX^*)+\frac{2\rho}{3\mu}\dist^2(x_i,\cX^*)\label{eqn:more0}\\
&\leq \frac{4}{3}\cdot \sum_{i=0}^{k-1} \dist(x_i,\cX^*)\label{eqn:more1}\\
&\leq \frac{4}{3}\cdot \dist(x_0,\cX^*) \cdot \sum_{i=0}^{k-1} \left(1-\frac{\mu^2}{2L^2}\right)^{\frac{i}{2}}\label{eqn:more2}\\
&\leq\frac{16L^2}{3\mu^2}\cdot\dist(x_0,\cX^*)\leq \frac{\epsilon}{4},\notag
\end{align}
where \eqref{eqn:more0} follows by Lemma~\ref{lem:lower_bound_subgrad} with $\gamma = 1/2$, the bound in \eqref{eqn:more1} follows by \eqref{eqn:intermediate_descent} and the assumption on $\dist(x_0, \cX^*),$ finally \eqref{eqn:more2} holds thanks to \eqref{eqn:intermediate_descent}. Thus applying the triangle inequality we get the contradiction $\|x_k-\bar x\|\leq \epsilon/2$. Consequently all the iterates $x_k$ for $k=0,1,\ldots, \infty$ lie in $B_{\epsilon/2}(\bar x)$ and satisfy
\eqref{eqn:subgrad_lin_loca}.

Finally, let $x_{\infty}$ be any limit point of the sequence $\{x_i\}$. We then successively compute
\begin{align*}
\|x_k-x_\infty\|\leq \sum_{i=k}^{\infty} \|x_{i+1}-x_i\|&\leq  \sum_{i=k}^{\infty} \frac{f(x_i)-\min f}{\|\zeta_i\|}\\
&\leq \frac{4L}{3\mu}\cdot \sum_{i=k}^{\infty} \dist(x_i,\cX^*)\\
&\leq \frac{4L}{3\mu}\cdot \dist(x_0,\cX^*) \cdot \sum_{i=k}^{\infty} \left(1-\frac{\mu^2}{2L^2}\right)^{\frac{i}{2}}\\
&\leq \frac{16L^3}{3\mu^3}\cdot \dist(x_0,\cX^*)\cdot \left(1-\frac{\mu^2}{2L^2}\right)^{\frac{k}{2}}.
\end{align*}
This completes the proof.

\subsection{Proof of Theorem~\ref{thm:geo_desc_loc}}
Fix an arbitrary index $k$ and observe
$$\|x_{k+1}-x_k\|=\left\|\proj_Q(x_k)-\proj_Q\left(x_k-\alpha_k\frac{\xi_k}{\|\xi_k\|}\right)\right\|\leq \alpha_k=\lambda\cdot q^k.$$  Hence,  we conclude the uniform bound on the iterates:
$$\|x_{k}-x_0\|\leq \sum_{i=0}^{\infty}\|x_{i+1}-x_i\|\leq \tfrac{\lambda}{1-q}$$ and the R-linear rate of convergence
$$\|x_{k}-x_{\infty}\|\leq \sum_{i=k}^{\infty}\|x_{i+1}-x_i\|\leq \tfrac{\lambda}{1-q}q^k,$$ where $x_{\infty}$ is any limit point of the iterate sequence.

Let us now show that the iterates do not escape $B_{\epsilon/2}(\bar{x})$. To this end, observe $$\|x_k-\bar x\|\leq \|x_k-x_0\|+\|x_0-\bar x\|\leq \tfrac{\lambda}{1-q}+\tfrac{\epsilon}{4}.$$
We must therefore verify the estimate $\tfrac{\lambda}{1-q}\leq \tfrac{\epsilon}{4}$, or equivalently $\gamma\leq \frac{\epsilon\rho L(1-\gamma)\tau^2}{4\mu^2(1+\sqrt{1-(1-\gamma) \tau^2})}.$
Clearly, it suffices to verify
$\gamma\leq \frac{\epsilon\rho (1-\gamma)}{4L},$ which holds by the definition of $\gamma$. Thus all the iterates $x_k$ lie in $B_{\epsilon/2}(\bar x)$. Moreover $\tau\leq \sqrt{\frac{1}{2}} \leq \sqrt{\frac{1}{2-\gamma}}$, the rest of the proof is identical to that in \cite[Theorem 5.1]{davis2018subgradient}.

\subsection{Proof of Theorem~\ref{thm:prox_lin_loc}}
Fix any index $i$ such that $x_i\in B_{\epsilon}(\bar x)$ and let $x\in \cX$ be arbitrary. Since the function $z\mapsto f_{x_i}(z)+\frac{\beta}{2}\|z-x_i\|^2$ is $\beta$-strongly convex and $x_{i+1}$ is its minimizer, we deduce
\begin{equation}\label{eqn:descent_est_prox_linloc}
\left(f_{x_i}(x_{i+1})+\frac{\beta}{2}\|x_{i+1}-x_i\|^2\right)+\frac{\beta}{2}\|x_{i+1}-x\|^2\leq f_{x_i}(x)+\frac{\beta}{2}\|x-x_i\|^2.
\end{equation}
Setting $x=x_i$ and appealing to approximation accuracy, we obtain the descent guarantee
\begin{equation}\label{eqn:descent_loc}	\|x_{i+1}-x_i\|^2\leq \frac{2}{\beta}(f(x_i)-f(x_{i+1})).
\end{equation}
In particular, the function values are decreasing along the iterate sequence.
Next choosing any $x^*\in\proj_{\cX^*}(x_i)$ and setting $x=x^*$ in \eqref{eqn:descent_est_prox_linloc} yields
$$\left(f_{x_i}(x_{i+1})+\frac{\beta}{2}\|x_{i+1}-x_i\|^2\right)+\frac{\beta}{2}\|x_{i+1}-x^*\|^2\leq f_{x_i}(x^*)+\frac{\beta}{2}\|x^*-x_i\|^2.$$
Appealing to approximation accuracy and lower-bounding $\frac{\beta}{2}\|x_{i+1}-x^*\|^2$ by zero, we conclude
\begin{equation}\label{eqn:key_guarant_stayclose}
f(x_{i+1})\leq f(x^*)+\beta\|x^*-x_i\|^2.
\end{equation}
Using sharpness we deduce the contraction guarantee
\begin{equation}\label{eqn:function_down}
\begin{aligned}
f(x_{i+1})-f(x^*)&\leq \beta\cdot \dist^2(x_i,\cX^*)\\
&\leq \frac{\beta}{\mu^2}(f(x_i)-f(x^*))^2\\
&\leq \frac{\beta(f(x_i)-f(x^*))}{\mu^2}\cdot(f(x_i)-f(x^*))\leq \frac{1}{2}\cdot(f(x_i)-f(x^*)),
\end{aligned}
\end{equation}
where the last inequality uses the assumption $f(x_0)-\min_{\cX} f\leq \frac{\mu^2}{2\beta}$.
Let $k>0$ be the first index satisfying $x_{k}\notin B_{\epsilon}(\bar x)$. We then deduce
\begin{align}
\|x_{k}-x_0\|\leq \sum_{i=0}^{k-1} \|x_{i+1}-x_i\|&\leq \sqrt{\frac{2}{\beta}}\cdot \sum_{i=0}^{k-1} \sqrt{f(x_i)-f(x_{i+1})}\label{eqn:use_desc_loc}\\
&\leq	  \sqrt{\frac{2}{\beta}}\cdot \sum_{i=0}^{k-1} \sqrt{f(x_i)- f(x^*)}\notag\\
&\leq \sqrt{\frac{2}{\beta}}\cdot \sqrt{f(x_0)-f(x^*)} \cdot \sum_{i=0}^{k-1} \left(\frac{1}{2}\right)^{\frac{i}{2}}\label{eqn:geo_descent}\\
&\leq \frac{1}{\sqrt{2}-1}\sqrt{\frac{f(x_0)-f(x^*)}{\beta}}\leq \epsilon/2,\notag
\end{align}
where \eqref{eqn:use_desc_loc} follows from \eqref{eqn:descent_loc}	and \eqref{eqn:geo_descent} follows from \eqref{eqn:function_down}.	Thus we conclude $\|x_k-\bar x\|\leq \epsilon$, which is a contradiction. Therefore all the iterates $x_k$, for $k=0,1,\ldots, \infty$, lie in $B_{\epsilon}(\bar x)$.
Combing this with \eqref{eqn:key_guarant_stayclose} and sharpness yields the claimed quadratic converge guarantee
$$\mu\cdot\dist(x_{k+1},\cX^*)\leq f(x_{k+1})-f(\bar x)\leq \beta\cdot \dist^2(x_k,\cX).$$
Finally, let $x_{\infty}$ be any limit point of the sequence $\{x_i\}$. We then deduce
\begin{align}
\|x_{k}-x_\infty\|\leq \sum_{i=k}^{\infty} \|x_{i+1}-x_i\|&\leq \sqrt{\frac{2}{\beta}}\cdot \sum_{i=k}^{\infty} \sqrt{f(x_i)-f(x_{i+1})}\notag\\
&\leq	  \sqrt{\frac{2}{\beta}}\cdot \sum_{i=k}^{\infty} \sqrt{f(x_i)-\min_{\cX} f}\notag\\
&\leq \frac{\mu\sqrt{2}}{\beta}\cdot  \sum_{i=k}^{\infty} \left(\frac{\beta}{\mu^2}(f(x_0)-\min f)\right)^{2^{i-1}}\label{eqn:equad_conv_func}\\
&\leq \frac{\mu\sqrt{2}}{\beta}\cdot  \sum_{i=k}^{\infty} \left(\frac{1}{2}\right)^{2^{i-1}}\notag\\
&\leq \frac{\mu\sqrt{2}}{\beta} \sum_{j=0}^{\infty} \left(\frac{1}{2}\right)^{2^{k-1}+j}\leq \frac{2\sqrt{2}\mu}{\beta}\cdot \left(\frac{1}{2}\right)^{2^{k-1}},\notag
\end{align}
where \eqref{eqn:equad_conv_func} follows from \eqref{eqn:function_down}.
The theorem is proved.

\section{Proofs in Section~\ref{sec:all_examples}}
\label{appendix:RIP}

\subsection{Proof of Lemma~\ref{lem:small_ball}} \label{proof:lem_small_ball}
In order to prove that the assumption in each case, we will prove a stronger  ``small-ball condition'' \cite{smallball1,smallball2}, which immediately implies the claimed lower bounds on the expectation by Markov's inequality. More precisely, we will show that there exist numerical constants $\mu_0,p_0>0$ such that
\begin{enumerate}
	\item {\bf (Matrix Sensing)} \begin{equation*} \inf_{\substack{M: \; \rank M \leq 2r \\ \|M\|_F = 1}} \PP (|\dotp{\lM,M}| \geq \muo) \geq \po, \end{equation*}
	\item {\bf (Quadratic Sensing I)}\begin{equation*} \label{assump:covariance_estimation12} \inf_{\substack{M\in \cS^d: \; \rank M \leq 2r \\ \|M\|_F = 1}} \PP (|\lv^\top M \lv| \geq \muo) \geq \po, \end{equation*}
	\item {\bf (Quadratic Sensing II)}
	\begin{equation*}
	\inf_{\substack{M\in\cS^d: \; \rank M \leq 2r \\ \|M\|_F = 1}} \PP \big( |\lv^\top M \lv - \tilde \lv^\top M \tilde \lv| \geq \muo \big) \geq \po,
	\end{equation*}
	\item {\bf (Bilinear Sensing)}
	\begin{equation*} \label{assump:covariance_estimation32} \inf_{\substack{M: \; \rank M \leq 2r \\ \|M\|_F = 1}} \PP (|\lv^\top  M \rv | \geq \muo) \geq \po. \end{equation*}
\end{enumerate}

These conditions immediately imply Assumptions~\ref{assumption:matrix_sensing}-\ref{assumption:bilinear}. Indeed, by Markov's inequality, in the case of matrix sensing we deduce
\[\EE  |\dotp{\lM,M}| \geq \muo \PP \left( |\dotp{\lM, M} |> \muo \right) \geq \muo \po.\]
The same reasoning applies to all the other problems.

\paragraph{Matrix sensing.}
Consider any matrix $M$ with $\|M\|_F =1.$ Then, since
$g := \dotp{\lM, M}$ follows a standard normal distribution, we may set $\muo$ to be the median of $|g|$ and $\po = 1/2$ to obtain
\[\inf_{\substack{M: \; \rank M \leq 2r \\ \|M\|_F = 1}} \PP (|\dotp{\lM,M}| \geq \muo) =  \PP (|g| \geq \muo) \geq \po.\]

\paragraph{Quadratic Sensing I.}
Fix a matrix $M$ with $\rank M \leq 2r$ and $\|M\|_F=1$. Let $M = UDU^\top$ be an eigenvalue decomposition of $M$. Using the rotational invariance of the Gaussian distribution, we deduce
\[\lv^\top M \lv \eqd \lv^\top D \lv = \sum_{k=1}^{2r} \lambda_k \lv_k^2, \]
where $\eqd$ denotes equality in distribution. Next, let $z$ be a standard normal variable. We will now invoke Proposition~\ref{prop:small_ball}. Let $C>0$ be the numerical constant appearing in the proposition. Notice that the function $\phi\colon\RR_+ \rightarrow \RR$ given by
\[\phi(t) = \sup_{u \in \RR} \PP(|z^2 - u| \leq t) \]
is continuous and strictly increasing, and it satisfies $\phi(0)= 0$ and $\lim_{t \rightarrow \infty} \phi(t) = 1.$ Hence we may set $\muo = \phi^{-1}(\min\{1/2C,1/2\})$. Proposition~\ref{prop:small_ball} then yields
\[\PP (|\lv^\top M \lv| \leq \muo) = \PP\left( \left|\sum_{k=1}^{2r} \lambda_k \lv_k^2 \right| \leq \muo \right) \leq \sup_{u \in \RR} \PP\left( \left|\sum_{k=1}^{2r} \lambda_k \lv_k^2 - u\right| \leq \muo \right) \leq C \phi(\muo) \leq \frac{1}{2}.\]
By taking the supremum of both sides of the inequality we conclude that Assumption~\ref{assumption:cov_est_naive} holds with $\muo$ and $\po = 1/2.$

\paragraph{Quadratic sensing II.}
Let $M = UDU^\top$ be an eigenvalue decomposition of $M$. Using the rotational invariance of the Gaussian distribution, we deduce
\[\lv^\top M \lv - \tilde \lv^\top M \tilde \lv    \eqd \lv^\top D \lv - \tilde \lv^\top D \tilde \lv = \sum_{k=1}^{2r} \lambda_k \left(\lv_k^2 - \tilde \lv_k^2\right) \eqd 2 \sum_{k=1}^{2r} \lambda_k \lv_k \tilde \lv_k, \]
where the last relation follows since $\left(\lv_k - \tilde \lv_k\right),\left(\lv_k + \tilde \lv_k\right)$ are independent standard normal random variables with mean zero and variance two.
We will now invoke Proposition~\ref{prop:small_ball}. Let $C>0$ be the numerical constant appearing in the proposition. Let $z$ and $ \tilde{z} $ be independent standard normal variables. Notice that the function $\phi:\RR_+ \rightarrow \RR$ given by
\[\phi(t) = \sup_{u \in \RR} \PP(|2 z\tilde{z} - u| \leq t) \]
is continuous, strictly increasing, satisfies $\phi(0)= 0$ and approaches one at infinity. Defining $\muo = \phi^{-1}(\min\{1/2C,1/2\})$ and applying Proposition~\ref{prop:small_ball}, we get
\[
\PP\left( \left|2\sum_{k=1}^{2r} \sigma_k \lv_k \tilde \lv_k \right|
\leq \muo \right) \leq \sup_{u \in \RR} \PP\left( \left|2\sum_{k=1}^{2r} \sigma_k \lv_k \tilde \lv_k - u\right| \leq \muo \right)
\leq C \phi(\muo) \leq \frac{1}{2}.
\]
By taking the supremum of both sides of the inequality we conclude that Assumption~\ref{assumption:cov_est_sym} holds with $\muo$ and $\po = 1/2.$

We omit the details for the bilinear case, which follow by similar arguments.

\subsection{Proof of Theorem~\ref{thm:main_rip_theorem}}\label{sec:proof_main_rip_theo}

The proofs in this section rely on the following proposition, which shows that that pointwise concentration imply uniform concentration. We defer the proof to Appendix \ref{sec:proof_epsilon_covering}.
\begin{proposition}\label{prop:epsilon_covering}
	Let $\cA: \RR^{d_1 \times d_2} \rightarrow \RR^m$ be a random linear mapping with property that for any fixed matrix $M \in \RR^{d_1 \times d_2}$ of rank at most $2r$ with norm $\|M\|_F =1$ and any fixed subset of indices $\cI \subseteq \{1, \dots, m\}$ satisfying $|\cI| < m/2$, the following hold:
	\begin{enumerate}[label = $\mathrm{(\arabic*)}$]
		\item \label{assumption:same_expected} The measurements $\cA(M)_1, \dots, \cA(M)_m$ are i.i.d.
		\item RIP holds in expected value:
		\begin{equation}\label{RIP_expected_val} \alpha \leq \EE| \cA(M)_i | \leq \beta(r) \qquad \text{for all } i \in \{1, \dots, m\} \end{equation}
		where $\alpha > 0$ is a universal constant and $\beta$ is a positive-valued function that could potentially depend on the rank of $M$.
		\item \label{assumption:deviations} There exist a universal constant $K>0$ and a positive-valued function $c(m,r)$ such that for any $t \in [0, K]$ the deviation bound
		\begin{equation}\label{ineq:deviation}
		\frac{1}{m}\left| \|\cA_{\cI^c} (M)\|_1 - \|\cA_{\cI} (M)\|_1  - \EE \big[\|\cA_{\cI^c} (M)\|_1 - \|\cA_{\cI} (M)\|_1\big] \right| \leq t
		\end{equation} holds with probability at least $1-2\exp(-t^2c(m,r)).$
	\end{enumerate}
	Then, there exist universal constants $c_1, \dots, c_6 > 0$ depending only on $\alpha$ and $K$ such that if $\cI \subseteq \{1, \dots, m\}$ is a fixed subset of indices satisfying $|\cI| < m/2$ and $$c(m,r) \geq \frac{c_1}{(1-2|\cI|/m)^2}r(d_1+d_2 + 1) \ln \left(c_2 + \frac{c_2\beta(r)}{1- 2|\cI|/m} \right)$$
	then with probability at least $1-4\exp\left(-c_3(1-2|\cI|/m)^2 c(m,r)\right)$ every matrix $M \in \RR^{d_1 \times d_2}$of rank at most $2r$ satisfies
	\begin{equation}\label{eq:rip_proof_bound}
	c_4 \|M\|_F \leq \frac{1}{m}\|\cA(M)\|_1 \leq c_5 \beta(r) \| M\|_F,
	\end{equation}
	and
	\begin{equation}\label{eq:rip_proof_outlier_bound}
	c_6 \left(1 - \frac{2|\cI|}{m} \right) \|M\|_F \leq \frac{1}{m}\left(\|\cA_{\cI^c}(M)\|_1 - \| \cA_{\cI} M\|_1\right).
	\end{equation}

\end{proposition}

Due to scale invariance of the above result, we need only verify its assumptions in the case that $\|M\|_F = 1$. We implicitly use this observation below.

\subsubsection{Part~\ref{thm:main_rip_theorem:item:matrix_sensing} of Theorem~\ref{thm:main_rip_theorem} (Matrix sensing)}\label{sec:RIP_proof_matrix_sen}
\begin{lemma}
	The random variable $|\dotp{\lM, M}|$ is sub-gaussian with parameter $C\eta.$ Consequently,
	\begin{equation} \label{ineq:expected_mat_sen}\co \leq \EE |\dotp{\lM, M} | \lesssim \eta. \end{equation}
	Moreover, there exists a universal constant $c> 0$ such that for any $t \in [0, \infty)$ the deviation bound
	\begin{equation}\label{ineq:deviation_mat_sen}
	\frac{1}{m}\left| \|\cA_{\cI^c} (M)\|_1 - \|\cA_{\cI} (M)\|_1  - \EE \big[\|\cA_{\cI^c} (M)\|_1 - \|\cA_{\cI} (M)\|_1\big] \right| \leq t
	\end{equation}
	holds with probability at least $1-2\exp\left(- \frac{ct^2}{\eta^2}m\right).$
\end{lemma}
\begin{proof}
	Assumption~\ref{assumption:matrix_sensing} immediately implies the lower bound in \eqref{ineq:expected_mat_sen}. To prove the upper bound, first note that by assumption we have
	\[\| \dotp{\lM, M} \|_{\psi_2} \lesssim \eta. \]
	This bound has two consequences, first $\dotp{\lM, M}$ is a sub-gaussian random variable with parameter $\eta$ and second $\EE |\dotp{\lM,M}| \lesssim \eta$ \cite[Proposition 2.5.2]{vershynin2016high}. Thus, we have proved \eqref{ineq:expected_mat_sen}.

	To prove the deviation bound~\eqref{ineq:deviation_mat_sen} we introduce the random variables
	\[Y_i  = \begin{cases}
	|\dotp{\lM_i, M}|  - \EE |\dotp{\lM_i, M}| & \text{if } i \notin \cI \text{, and }\\
	- \left( |\dotp{\lM_i, M}|  - \EE |\dotp{\lM_i, M}|  \right) & \text{otherwise.}
	\end{cases}\]
	Since $|\dotp{\lM_i, M}|$ is sub-gaussian, we have $\|Y_i\|_{\psi_2} \lesssim \eta$ for all $i,$ see \cite[Lemma 2.6.8]{vershynin2016high}. Hence, Hoeffding's inequality for sub-gaussian random variables \cite[Theorem 2.6.2]{vershynin2016high} gives the desired upper bound on $\PP \left( \frac{1}{m} \abs{\sum_{i=1}^m Y_i } \geq t \right).$
\end{proof}
Applying Proposition~\ref{prop:epsilon_covering} with $\beta(r) \asymp \eta$ and $c(m,r) \asymp m/\eta^2$ now yields the result.
\qed

\subsubsection{Part~\ref{thm:main_rip_theorem:item:quadratic_sensing_I} of Theorem~\ref{thm:main_rip_theorem} (Quadratic sensing I)}\label{sec:RIP_proof_cov_est}
\begin{lemma}
	The random variable $|\lv^\top M \lv|$ is sub-exponential with parameter $\sqrt{2r} \eta^2.$ Consequently,
	\begin{equation} \label{ineq:expected_cov}\co \leq \EE |\lv^\top M \lv| \lesssim \sqrt{2r}\eta^2. \end{equation}
	Moreover, there exists a universal constant $c> 0$ such that for any $t \in [0, \sqrt{2r} \eta]$ the deviation bound
	\begin{equation}\label{ineq:deviation_cov}
	\frac{1}{m}\left| \|\cA_{\cI^c} (M)\|_1 - \|\cA_{\cI} (M)\|_1  - \EE \big[\|\cA_{\cI^c} (M)\|_1 - \|\cA_{\cI} (M)\|_1\big] \right| \leq t
	\end{equation}
	holds with probability at least $1-2\exp\left(- \frac{ct^2}{\eta^4}m/r\right).$
\end{lemma}
\begin{proof}
	Assumption~\ref{assumption:cov_est_naive} gives the lower bound in \eqref{ineq:expected_cov}. To prove the upper bound, first note that $M = \sum_{k=1}^{2r} \sigma_k u_k u_k^\top$ where $\sigma_k$ and $u_k$ are the $k$th singular values and vectors of $M$, respectively. Hence
	\begin{align*}
	\| \lv^\top M \lv\|_{\psi_1} & = \left\|\lv^\top \left( \sum_{k=1}^{2r} \sigma_k u_k u_k^\top\right) \lv \right\|_{\psi_1}  = \left\|  \sum_{k=1}^{2r} \sigma_k \dotp{\lv, u_k}^2 \right\|_{\psi_1} \\
	& \hspace{2cm} \leq  \sum_{k=1}^{2r} \sigma_k  \left\|  \dotp{\lv, u_k}^2 \right\|_{\psi_1}  \leq  \sum_{k=1}^{2r}\sigma_k  \left\|  \dotp{\lv, u_k} \right\|_{\psi_2}^2 = \eta^2 \sum_{k=1}^{2r} \sigma_k  \leq \sqrt{2r} \eta^2,
	\end{align*}
	where the first inequality follows since $\| \cdot \|_{\psi_1}$ is a norm, the second one follows since $\|XY\|_{\psi_1} \leq \|X\|_{\psi_2}\|Y\|_{\psi_2}$ \cite[Lemma 2.7.7]{vershynin2016high}, and the third inequality holds since $\|\sigma\|_1 \leq \sqrt{2r}\|\sigma\|_2$. This bound has two consequences, first $\lv^\top M \lv$ is a sub-exponential random variable with parameter $\sqrt{r} \eta^2$ and second $\EE \lv^\top M \lv \leq \sqrt{2r} \eta^2$ \cite[Exercise 2.7.2]{vershynin2016high}. Thus, we have proved \eqref{ineq:expected_cov}.

	To prove the deviation bound~\eqref{ineq:deviation_cov} we introduce the random variables
	\[Y_i  = \begin{cases}
	\lv_i^\top M \lv_i - \EE \lv_i^\top M \lv_i & \text{if } i \notin \cI \text{, and }\\
	- \left(  \lv_i^\top M \lv_i - \EE \lv_i^\top M \lv_i \right) & \text{otherwise.}
	\end{cases}\]
	Since $\lv^\top M \lv$ is sub-exponential, we have $\|Y_i\|_{\psi_1} \lesssim \sqrt{r} \eta^2$ for all $i,$ see \cite[Exercise 2.7.10]{vershynin2016high}. Hence, Bernstein inequality for sub-exponential random variables \cite[Theorem 2.8.2]{vershynin2016high} gives the desired upper bound on $\PP \left( \frac{1}{m} \abs{\sum_{i=1}^m Y_i } \geq t \right).$
\end{proof}
Applying Proposition~\ref{prop:epsilon_covering} with with $\beta(r) \asymp \sqrt{r}\eta^2$ and $c(m,r) \asymp m/{\eta^4}r$ now yields the result.
\qed

\subsubsection{Part~\ref{thm:main_rip_theorem:item:quadratic_sensing_II} of Theorem~\ref{thm:main_rip_theorem} (Quadratic sensing II)}\label{sec:RIP_proof_cov_est_sym}
\begin{lemma}\label{lem:exp_rip_cov_sym}
	The random variable $|\lv^\top M \lv - \tilde \lv^\top M \tilde \lv|$ is sub-exponential with parameter $C\eta^2.$ Consequently,
	\begin{equation} \label{ineq:expected_cov_sym}\co \leq \EE |\lv^\top M \lv - \tilde \lv^\top M \tilde \lv| \lesssim \eta^2. \end{equation}
	Moreover, there exists a universal constant $c> 0$ such that for any $t \in [0,  \eta^2]$ the deviation bound
	\begin{equation}\label{ineq:deviation_cov_sym}
	\frac{1}{m}\left| \|\cA_{\cI^c} (M)\|_1 - \|\cA_{\cI} (M)\|_1  - \EE \big[\|\cA_{\cI^c} (M)\|_1 - \|\cA_{\cI} (M)\|_1\big] \right| \leq t
	\end{equation}
	holds with probability at least $1-2\exp\left(- \frac{ct^2}{\eta^4}m\right).$
\end{lemma}
\begin{proof}
	Assumption \ref{assumption:cov_est_sym} implies the lower bound in \eqref{ineq:expected_cov_sym}. To prove the upper bound, we will show that $\||\lv^\top M \lv - \tilde \lv^\top M \tilde \lv^\top|\|_{\psi_1} \leq \eta^2$. By definition of the Orlicz norm $\||X|\|_{\psi_1} = \|X\|_{\psi_1}$ for any random variable $X,$ hence without loss of generality we may remove the absolute value. Recall that $M = \sum_{k=1}^{2r} \sigma_k u_k u_k^\top$ where $\sigma_k$ and $u_k$ are the $k$th singular values and vectors of $M$, respectively. Hence, the random variable of interest can be rewritten as
	\begin{equation} \label{eq:cov_est_sym_decomp}\lv^\top M \lv - \tilde \lv^\top M \tilde \lv^\top \eqd \sum_{k=1}^{2r} \sigma_k\left(\dotp{u_k, \lv}^2 -  \dotp{u_k, \tilde \lv}^2  \right).\end{equation} By assumption the random variables $\dotp{u_k, \lv}$ are $\eta$-sub-gaussian, this implies that $\dotp{u_k,\lv}^2$ are $\eta^2$-sub-exponential, since $\|\dotp{u_k, \lv}^2\|_{\psi_1} \leq \|\dotp{u_k, \lv}\|_{\psi_2}^2$.

	Recall the following characterization of the Orlicz norm for mean-zero random variables
	\begin{equation}\label{eq:equivalence_psi1_norm}\|X\|_{\psi_1} \leq Q \iff \EE \exp(\lambda X) \leq \exp(\tilde Q^2\lambda^2) \;\; \text{for all }\lambda \text{ satisfying } |\lambda| \leq 1/\tilde Q^2\end{equation}
	where the $Q \asymp \tilde Q,$ see \cite[Proposition 2.7.1]{vershynin2016high}. To prove that the random variable \eqref{eq:cov_est_sym_decomp} is sub-exponential we will exploit this characterization. Since each inner product squared $\dotp{u_k,\lv}^2$ is sub-exponential, the equivalence implies the existence of a constant $c>0$ for which the uniform bound \begin{equation}\label{oh_god_so_many_labels}
	\EE \exp(\lambda \dotp{u_k,\lv}^2) \leq \exp\left(c\eta^4 \lambda^2\right) \qquad \text{for all } k\in[2r]\text{ and }|\lambda|\leq 1/c\eta^4
	\end{equation}  holds. Let $\lambda$ be an arbitrary scalar with $|\lambda|\leq 1/c\eta^4$, then by expanding the moment generating function of \eqref{eq:cov_est_sym_decomp} we get
	\begin{align*}
	\EE \exp\left(\lambda \sum_{k=1}^{2r} \sigma_k \left(\dotp{u_k, \lv}^2 -  \dotp{u_k, \tilde \lv}^2  \right)  \right) &= \EE\prod_{k=1}^{2r}  \exp\left(\lambda  \sigma_k \dotp{u_k, \lv}^2\right) \exp\left( - \lambda \sigma_k \dotp{u_k, \tilde \lv}^2  \right)   \\
	&= \prod_{k=1}^{2r}  \EE\exp\left(\lambda  \sigma_k \dotp{u_k, \lv}^2\right) \EE \exp\left( - \lambda \sigma_k \dotp{u_k, \tilde \lv}^2  \right)  \\
	& \leq \prod_{k=1}^{2r}  \exp\left((c\eta)^2\lambda^2  \sigma_k^2\right) \exp\left(  c\eta^4\lambda^2  \sigma_k^2 \right)  \\
	& =   \exp\left(2c\eta^4\lambda^2 \sum_{k=1}^{2r}  \sigma_k^2\right) = \exp\left(2c\eta^4\lambda^2\right).
	\end{align*}
	where the inequality follows by \eqref{oh_god_so_many_labels} and the last relation follows since $\sigma$ is unit norm.
	Combining this with \eqref{eq:equivalence_psi1_norm} gives
	\[\||\lv^\top M \lv - \tilde \lv^\top M \tilde \lv^\top|\|_{\psi_1}\lesssim \eta^2.\]
	This bound has two consequences, first $|\lv^\top M \lv - \tilde \lv^\top M \tilde \lv^\top|$ is a sub-exponential random variable with parameter $C\eta^2$ and second $\EE |\lv^\top M \lv - \tilde \lv^\top M \tilde \lv^\top| \leq C \eta^2$ \cite[Exercise 2.7.2]{vershynin2016high}. Thus, we have proved \eqref{ineq:expected_cov_sym}.

	To prove the deviation bound~\eqref{ineq:deviation_cov_sym} we introduce the random variables
	\[Y_i  = \begin{cases}
	\cA(M)_i  - \EE \cA(M)_i & \text{if } i \notin \cI \text{, and }\\
	- \left(  \cA(M)_i  - \EE \cA(M)_i  \right) & \text{otherwise.}
	\end{cases}\]
	The sub-exponentiality of $\cA(M)_i$ implies $\|Y_i\|_{\psi_1} \lesssim \eta^2$ for all $i,$ see \cite[Exercise 2.7.10]{vershynin2016high}. Hence, Bernstein inequality for sub-exponential random variables \cite[Theorem 2.8.2]{vershynin2016high} gives the desired upper bound on $\PP \left( \frac{1}{m} \abs{\sum_{i=1}^m Y_i } \geq t \right).$
\end{proof}
Applying Proposition~\ref{prop:epsilon_covering} with $\beta(r) \asymp \eta^2$ and $c(m,r) \asymp m/{\eta^4}$ now yields the result.
\qed
\subsubsection{Part~\ref{thm:main_rip_theorem:item:bilinear_sensing} of Theorem~\ref{thm:main_rip_theorem} (Bilinear sensing)}\label{sec:RIP_proof_bilinear}

\begin{lemma}\label{lem:exp_rip_bilinear}
	The random variable $|\lv^\top M \rv|$ is sub-exponential with parameter $C\eta^2.$ Consequently,
	\begin{equation} \label{ineq:expected_bilinear}\co \leq \EE |\lv^\top M \rv| \lesssim \eta^2. \end{equation}
	Moreover, there exists a universal constant $c> 0$ such that for any $t \in [0,  \eta^2]$ the deviation bound
	\begin{equation}\label{ineq:deviation_bilinear}
	\frac{1}{m}\left| \|\cA_{\cI^c} (M)\|_1 - \|\cA_{\cI} (M)\|_1  - \EE \big[\|\cA_{\cI^c} (M)\|_1 - \|\cA_{\cI} (M)\|_1\big] \right| \leq t
	\end{equation}
	holds with probability at least $1-2\exp\left(- \frac{ct^2}{\eta^4}m\right).$
\end{lemma}
\begin{proof}
	As before the lower bound in \eqref{ineq:expected_bilinear} is implied by Assumption~\ref{assumption:bilinear}. To prove the upper bound, we will show that $\||\lv^\top M \rv|\|_{\psi_1} \leq \eta^2$. By definition of the Orlicz norm $\||X|\|_{\psi_1} = \|X\|_{\psi_1}$ for any random variable $X,$ hence we may remove the absolute value. Recall that $M = \sum_{k=1}^{2r} \sigma_k u_k v_k^\top$ where $\sigma_k$ and $(u_k, v_k)$ are the $k$th singular values and vectors of $M$, respectively. Hence, the random variable of interest can be rewritten as
	\begin{equation} \label{eq:bilinear_decomp}\lv^\top M \rv \eqd \sum_{k=1}^{2r} \sigma_k \dotp{\lv,u_k} \dotp{v_k, \rv}.\end{equation} By assumption the random variables $\dotp{\lv, u_k}$ and $\dotp{v_k,\rv}$ are $\eta$-sub-gaussian, this implies that $\dotp{\lv,u_k}\dotp{v_k,\rv}$ are $\eta^2$-sub-exponential.

	To prove that the random variable \eqref{eq:bilinear_decomp} is sub-exponential we will again use \eqref{eq:equivalence_psi1_norm}. Since each random variable $\dotp{\lv,u_k}\dotp{v_k,\rv}$ is sub-exponential, the equivalence implies the existence of a constant $c>0$ for which the uniform bound \begin{equation}\label{oh_god_so_many_labels2}
	\EE \exp(\lambda \dotp{\lv,u_k}\dotp{v_k,\rv}) \leq \exp\left(c\eta^4 \lambda^2\right) \qquad \text{for all } k\in[2r]\text{ and }|\lambda|\leq 1/c\eta^4
	\end{equation}  holds. Let $\lambda$ be an arbitrary scalar with $|\lambda|\leq 1/c\eta^4$, then by expanding the moment generating function of \eqref{eq:bilinear_decomp} we get
	\begin{align*}
	\EE \exp\left(\lambda \sum_{k=1}^{2r} \sigma_k \dotp{\lv,u_k}\dotp{v_k,\rv}  \right)
	&= \prod_{k=1}^{2r}  \EE\exp\left(\lambda  \sigma_k \dotp{\lv,u_k}\dotp{v_k,\rv} \right)   \\
	& \leq   \exp\left(2c\eta^4\lambda^2 \sum_{k=1}^r  \sigma_k^2\right) = \exp\left(2c\eta^4\lambda^2\right).
	\end{align*}
	where the inequality follows by \eqref{oh_god_so_many_labels2} and the last relation follows since $\sigma$ is unitary. Combining this with \eqref{eq:equivalence_psi1_norm} gives
	\[\||\lv^\top M \rv|\|_{\psi_1}\lesssim \eta^2.\] Thus, we have proved \eqref{ineq:expected_bilinear}.

	Once again, to show the deviation bound~\eqref{ineq:deviation_bilinear} we introduce the random variables
	\[Y_i  = \begin{cases}
	|\lv_i^\top M \rv_i|  - \EE |\lv_i^\top M \rv_i| & \text{if } i \notin \cI \text{, and }\\
	- \left( |\lv_i^\top M \rv_i| - \EE |\lv_i^\top M \rv_i| \right) & \text{otherwise.}
	\end{cases}\]
	and apply Bernstein's inequality for sub-exponential random variables \cite[Theorem 2.8.2]{vershynin2016high} to get the stated upper bound on $\PP \left( \frac{1}{m} \abs{\sum_{i=1}^m Y_i } \geq t \right).$
\end{proof}
Applying Proposition~\ref{prop:epsilon_covering} with $\beta(r) \asymp \eta^2$ and $c(m,r) \asymp m/{\eta^4}$ now yields the result.
\qed

\subsection{Proof of Proposition~\ref{prop:epsilon_covering}}\label{sec:proof_epsilon_covering}

Choose $\epsilon \in (0,\sqrt{2})$ and let $\cN$ be the ($\epsilon/\sqrt{2}$)-net guaranteed
by Lemma~\ref{lemma:eps_net}. Pick some $t \in (0,K]$ so that \eqref{ineq:deviation} can hold, we will fix the value of this parameter later in the proof. Let $\cE$ denote the event that the following two estimates hold for all matrices in $M\in \cN$:
\begin{align}
\frac{1}{m}\Big|\|\cA_{ \cI^c }(M)\|_1 - \|\cA_{ \cI }(M)\|_1 - \EE\left[\|\cA_{ \cI^c }(M)\|_1 - \|\cA_{ \cI }(M)\|_1\right]\Big| &\leq t,\label{ineq:concentration_ind1_first}\\
\frac{1}{m}\Big|\|\cA(M)\|_1  - \EE\left[\|\cA(M)\|_1 \right]\Big| &\leq t. \label{ineq:concentration_ind_first}
\end{align}
Throughout the proof, we will assume that the event $\cE$ holds.
We will estimate the probability of $\cE$ at the end of the proof. Meanwhile, seeking to establish RIP,  define the quantity
$$c_2 := \sup_{M \in S_{2r}} \frac{1}{m}\|\cA(M)\|_1.$$
We aim first to provide a high probability bound on  $c_2$.

Let $M \in S_{2r}$ be arbitrary and let $M_\star$ be the closest point to $M$ in $\cN$. Then we have
\begin{align}
\frac{1}{m}\|\cA(M)\|_1 & \leq \frac{1}{m}\|\cA(M_\star)\|_1 + \frac{1}{m}\|\cA(M- M_\star)\|_1\notag\\
&\leq \frac{1}{m}\EE\|\cA(M_\star)\|_1 + t+ \frac{1}{m}\|\cA(M- M_\star)\|_1  \label{eqn:est_firstusesec}\\
&\leq \frac{1}{m}\EE\|\cA(M)\|_1 + t+ \frac{1}{m}\left(\EE\|\cA(M - M_\star)\|_1+ \|\cA(M- M_\star)\|_1 \right),\label{ineq:wtf_triangle}
\end{align}
where \eqref{eqn:est_firstusesec} follows from \eqref{ineq:concentration_ind_first} and \eqref{ineq:wtf_triangle} follows from the triangle inequality. To simplify the third term in \eqref{ineq:wtf_triangle},  using SVD, we deduce that there exist two orthogonal matrices $M_1, M_2$ of rank at most $2r$ satisfying $M - M_\star = M_1+M_2.$ With this decomposition in hand, we compute
\begin{align}
\frac{1}{m}\|\cA (M - M_\star)\|_1 &\leq \frac{1}{m}\|\cA(M_1)\|_1 + \frac{1}{m}\|\cA(M_2)\|_1\notag
\\& \leq c_2 (\|M_1\|_F+\|M_2\|_F) \leq \sqrt{2}c_2 \|M-M_\star \|_F \leq c_2 \epsilon,\label{eqn:c2bound}
\end{align}
where the second inequality follows from the definition of $c_2$ and the estimate $\|M_1\|_F + \|M_2\|_F \leq \sqrt{2} \|(M_1, M_2)\|_F = \sqrt{2} \|M_1 + M_2\|_F.$
Thus, we arrive at the bound
\begin{equation}\label{ineq:RIP_upper}\frac{1}{m}\|\cA(M)\|_1 \leq \frac{1}{m}\EE\|\cA(M)\|_1 + t+ 2c_2 \epsilon.\end{equation}
As $M$ was arbitrary, we may take the supremum of both sides of the inequality, yielding
$c_2\leq \frac{1}{m}\sup_{M \in S_{2r}}\EE\|\cA(M)\|_1 + t+ 2c_2 \epsilon$. Rearranging yields the bound
$$
c_2 \leq \dfrac{\frac{1}{m}\sup_{M \in S_{2r}}\EE\|\cA(M)\|_1 + t}{1-2\epsilon}.
$$ Assuming that $\epsilon \leq 1/4$, we further deduce that
\begin{equation} \label{sigma_bounded}
c_2 \leq \bar \sigma := \frac{2}{m}\sup_{M \in S_{2r}}\EE\|\cA(M)\|_1 + 2t \leq 2 \beta(r) + 2t,
\end{equation}
establishing that the random variable $c_2$ is bounded by $\bar \sigma$ in the event $\cE$.

Now let $\hat \cI$ denote either $\hat \cI=\emptyset$ or $\hat \cI=\cI$. We now provide a uniform lower bound on $\frac{1}{m}\|\cA_{\hat \cI^c }(M)\|_1 - \frac{1}{m}\|\cA_{\hat \cI }(M)\|_1$. Indeed,
\begin{align}
&\frac{1}{m}\|\cA_{\hat \cI^c }(M)\|_1 - \frac{1}{m}\|\cA_{\hat \cI }(M)\|_1\notag \\
&=\frac{1}{m}\|\cA_{\hat \cI^c }(M_{\star})+\cA_{\hat \cI^c }(M-M_{\star})\|_1 - \frac{1}{m}\|\cA_{\hat \cI }(M_{\star})+\cA_{\hat \cI }(M-M_\star)\|_1\notag \\
&\geq \frac{1}{m}\|\cA_{\hat \cI^c }(M_\star)\|_1 - \frac{1}{m}\|\cA_{\hat \cI }(M_\star)\|_1 -  \frac{1}{m}\|\cA(M- M_\star)\|_1 \label{eqn:doubletriangle}\\
& \geq  \frac{1}{m}\EE\left[\|\cA_{\hat \cI^c }(M_\star)\|_1 - \|\cA_{\hat \cI }(M_\star)\|_1\right] - t -  \frac{1}{m}\|\cA(M- M_\star)\|_1 \label{eqn:allple45}\\
& \geq \frac{1}{m}\EE\left[\|\cA_{\hat \cI^c }(M)\|_1 - \|\cA_{\hat \cI }(M)\|_1\right] - t -\frac{1}{m} \left(\EE \|\cA(M- M_\star)\|_1 + \|\cA(M- M_\star)\|_1\right) \label{eqn:moredoubletriangle}\\
& \geq \frac{1}{m}\EE\left[|\|\cA_{\hat \cI^c }(M)\|_1 - \|\cA_{\hat \cI }(M)\|_1\right] - t  - 2\bar \sigma \epsilon,\label{eqn:finalestthankgod}
\end{align}
where \eqref{eqn:doubletriangle} uses the forward and reverse triangle inequalities,
\eqref{eqn:allple45} follows from~\eqref{ineq:concentration_ind1_first}, the estimate \eqref{eqn:moredoubletriangle} follows from the forward and reverse triangle inequalities, and \eqref{eqn:finalestthankgod} follows from \eqref{eqn:c2bound} and \eqref{sigma_bounded}. Switching the roles of $\cI$ and $\cI^c$ in the above sequence of inequalities, and choosing $\epsilon = t/4\bar \sigma$, we deduce
\[\frac{1}{m}\sup_{M \in S_{2r}}\Big|\|\cA_{\hat \cI^c }(M)\|_1 - \|\cA_{\hat \cI }(M)\|_1 - \EE\left[\|\cA_{\hat \cI^c }(M)\|_1 - \|\cA_{\hat \cI }(M)\|_1\right]\Big| \leq  \frac{3t}{2}.\]
In particular, setting $\hat \cI=\emptyset$, we deduce
\[\frac{1}{m}\sup_{M \in S_{2r}}\Big|\|\cA(M)\|_1 - \EE\left[\|\cA(M)\|_1 \right]\Big| \leq  \frac{3t}{2}\]
and therefore using \eqref{RIP_expected_val}, we conclude the RIP property
\begin{equation}\label{eqn:target_eqn1proof}
\alpha-\frac{3t}{2}\leq  \frac{1}{m}\|\cA(M)\|_1\lesssim \beta(r)+\frac{3t}{2},\qquad \forall X\in S_{2r}.
\end{equation}
Next, let $\hat \cI = \cI$ and note that
$$
\frac{1}{m}\EE\left[\|\cA_{\hat \cI^c }(M)\|_1 - \|\cA_{\hat \cI }(M)\|_1\right] = \frac{|\cI^c| - |\cI|}{m}\cdot\EE |\cA(M)_i | \geq \left(1- \frac{2|\cI|}{m}\right) \alpha,
$$
where the equality follows by assumption \ref{assumption:same_expected}. Therefore every $M\in S_{2r}$ satisfies
\begin{equation}\label{eqn:target_eqn2proof}
\frac{1}{m}\left[\|\cA_{\hat \cI^c }(M)\|_1 - \|\cA_{\hat \cI}(M)\|_1\right]\geq \left(1- \frac{2|\cI|}{m}\right) \alpha-\frac{3t}{2}.
\end{equation}
Setting $t=\frac{2}{3}\min\{\alpha, \alpha(1-2|\cI|/m)/2\} =  \frac{1}{3}\alpha(1-2|\cI|/m)$ in \eqref{eqn:target_eqn1proof} and \eqref{eqn:target_eqn2proof}, we deduce the claimed estimates \eqref{eq:rip_proof_bound} and \eqref{eq:rip_proof_outlier_bound}.
Finally, let us estimate the probability of $\cE$. Using the union bound and Lemma~\ref{lemma:eps_net} yields
\begin{align*}
\PP(\cE^c) & \leq \sum_{M \in \cN} \PP \big\{ \text{\eqref{ineq:concentration_ind1_first} or  \eqref{ineq:concentration_ind_first} fails at }M\big\} \\
&\leq 4|\cN|\exp\left(-t^2 c(m,r)\right)\\
& \leq 4\left(\frac{9}{\epsilon}\right)^{2(d_1+d_2+1)r} \exp\left(-t^2 c(m,r)\right) \\
&=4  \exp\left(2(d_1+d_2+1)r\ln(9/\epsilon)- t^2c(m,r)\right)\end{align*}
where $c(m,r)$ is the function guaranteed by assumption~\ref{assumption:deviations}.

By \eqref{RIP_expected_val} we get $1/\epsilon = 4\bar \sigma/t \lesssim 2 + \beta(r)/(1 - 2|\cI|/m)$. Then we deduce
$$\PP(\cE^c)\leq 4  \exp\left(c_1(d_2+d_2+1)r\ln\left(c_2+\frac{c_2\beta(r)}{1-2|\cI|/m}\right)-\frac{\alpha^2}{9}(1-\frac{2|\cI|}{m})^2c({m},{r})\right).$$
Hence as long as $c(m,r)\geq \frac{9c_1(d_1+d_2+1)r^2\ln\left(c_2+\frac{c_2\beta(r)}{1-2|\cI|/m}\right)}{\alpha^2 \left(1-\frac{2|\cI|}{m}\right)^2}$, we can be sure $$\PP(\cE^c)\leq 4  \exp\left(-\frac{\alpha^2}{18}\left(1-\frac{2|\cI|}{m}\right)^2c({m},{r})\right).$$ Proving the desired result.
\qed

\section{Proof in Section~\ref{sec:mat_comp}}
\subsection{Proof of Lemma~\ref{lem:basic_conv_guarantee}}

Define $P(x,y)=a\|y-x\|^2_2+b\|y-x\|_2$. Fix an iteration $k$ and choose $x^*\in\proj_{\cX^*}(x_k)$. Then the estimate holds:
\begin{align*}
f(x_{k+1})\leq f_{x_k}(x_{k+1})+ P(x_{k+1},x_k)
&\leq f_{x_k}(x^*)+ P(x^*,x_k)\leq f(x^*)+2P(x^*,x_k).
\end{align*}
Rearranging and using the sharpness and approximation accuracy assumptions, we deduce
$$\mu\cdot\dist(x_{k+1},\cX^*)\leq 2(a\cdot\dist^2(x,\cX^*)+b\cdot\dist(x,\cX^*))=2(b+a\dist(x,\cX^*))\dist(x,\cX^*).$$
The result follows.

\subsection{Proof of Theorem~\ref{thm:basic_conv_guarantee_sub}}
First notice that for any $y$, we have $\partial f(y) = \partial
f_y(y)$. Therefore, since $f_y$ is a convex function, we have that for all $x,
y \in \cX$ and $v \in \partial f(y)$, the bound
\begin{align}\label{eq:mc_weak_convexity}
f(y) + \dotp{v , x - y } = f_y(y) + \dotp{v ,x -y } \leq f_y(x)  \leq f(x) + a\|x - y\|_F^2 + b\|x - y\|_F.
\end{align}
Consequently, given that $\dist(x_i,\cX^*)\leq \gamma\cdot  \frac{\mu -
	2b}{2a}$, we have
\begin{align}
\|x_{i+1} -  x^*\|^2&=\left\|\proj_{\cX}\left(x_{i}-\tfrac{f(x_i)-\min_{\mathcal{X}} f}{\|\zeta_i\|^2} \zeta_i\right)-\proj_\cX(x^*)\right\|^2\notag\\
&\leq \left\|(x_{i} - x^*)- \tfrac{f(x_i)-\min_{\mathcal{X}} f}{\|\zeta_i\|^2} \zeta_i\right\|^2 \label{eqn1:weird_subgrad_mc}\\
&= \|x_{i} -  x^*\|^2 + \frac{2(f(x_i) - \min_{\cX} f)}{\|\zeta_i\|^2}\cdot\dotp{\zeta_i, x^* - x_{i}} + \frac{(f(x_i) - f( x^*))^2}{\|\zeta_i\|^2} \notag\\
&\leq \|x_{i} -  x^*\|^2 + \frac{2(f(x_i) -  \min f)}{\|\zeta_i\|^2}\left( f( x^*) - f(x_i) + a\|x_i -  x^*\|^2 + b \|x_i - x^\ast\| \right)\notag \\
&\qquad\qquad+ \frac{(f(x_i) - f(x^*))^2}{\|\zeta_i\|^2} \label{eqn2:weird_subgrad_mc}\\
&= \|x_{i} -  x^*\|^2 + \frac{f(x_i) - \min f}{\|\zeta_i\|^2}\left(2a\|x_i -  x^*\|^2 + 2b \|x_i - x^\ast\| -  (f(x_i) - f( x^*))  \right)\notag\\
&\leq \|x_{i} -  x^*\|^2 + \frac{f(x_i) -  \min f}{\|\zeta_i\|^2}\left(a\|x_i -  x^*\|^2 -  (\mu-2b)\|x_i -  x^*\|  \right)\label{eqn3:weird_subgrad_mc} \\
&= \|x_{i} -  x^*\|^2 + \frac{2a (f(x_i) -  \min f)}{\|\zeta_i\|^2}\left( \|x_i -  x^*\| -\frac{\mu- 2b}{2a}\right)\|x_i -  x^*\|\notag\\
&\leq \|x_{i} -  x^*\|^2 - \frac{(1-\gamma)(\mu - 2b)(f(x_i) -  \min f)}{\|\zeta_i\|^2}\cdot \|x_i -  x^*\|\label{eqn35:weird_subgrad_mc}\\
&\leq \left(1-\frac{ (1-\gamma)\mu(\mu - 2b)}{\|\zeta_i\|^2}\right)\|x_i- x^*\|^2\label{eqn4:weird_subgrad_mc}.
\end{align}
Here, the estimate \eqref{eqn1:weird_subgrad_mc} follows from the fact that
the projection $\proj_\cX(\cdot)$ is nonexpansive,
\eqref{eqn2:weird_subgrad_mc} uses the bound
in~\eqref{eq:mc_weak_convexity}, \eqref{eqn35:weird_subgrad_mc} follow from
the estimate $\dist(x_i,\cX^*)\leq \gamma \cdot \frac{\mu - 2b}{2a}$, while
\eqref{eqn3:weird_subgrad_mc} and \eqref{eqn4:weird_subgrad_mc} use local
sharpness. The result then follows by the upper bound $\|\zeta_i\| \leq L$.
\section{Proofs in Section~\ref{sec:robust_pca}}

\subsection{Proof of Lemma~\ref{lem:rpca_cross_term}}
\label{appendix:proof_rpca_cross_term}

The inequality can be established using an argument similar to that for bounding the $ T_3 $ term in \cite[Section 6.6]{chen2015fast}. We provide the proof below for completeness. Define the shorthand $ \Delta_S := S-S_{\sharp}$ and $ \Delta_X = X- X_{\sharp} $, and let $ e_j \in \mathbb{R}^d$ denote the $ j $-th standard basis vector of $ \mathbb{R}^d $. Simple algebra gives
\begin{align*}
| \langle S-S_{\sharp}, XX^\top -X_{\sharp} X_{\sharp}^\top\rangle |
&= | 2 \langle \Delta_S, \Delta_X  X_{\sharp}^\top \rangle + \langle \Delta_S, \Delta_X \Delta_X^\top \rangle | \\
&\le \Big( 2 \| X^{\top}_{\sharp} \Delta_S \|_F  + \| \Delta_X^\top \Delta_S \|_F \Big) \cdot \| \Delta_X\|_F.
\end{align*}
We claim that $ \| \Delta_S e_j \|_1 \le 2\sqrt{k}  \| \Delta_S e_j \|_2$ for each $ j\in [d] $. To see this, fix any $ j\in [d] $ and let $ v := Se_j $, $ v^* := S_\sharp e_j $, and $ T := \text{support}(v^*). $ We have
\begin{align*}
\| v^*_T \|_1  =  \| v^* \|_1
& \ge \| v \|_1  && S \in \mathcal{S}\\
& = \| v_T \|_1 + \| v_{T^c} \|_1 && \text{decomposability of $ \ell_1 $ norm}\\
& = \| v^*_T + (v - v^*)_T \|_1 + \| (v - v^*)_{T^c} \|_1 &&\\
& \ge \| v^*_T \|_1 - \|  (v - v^*)_T \|_1  +  \| (v - v^*)_{T^c} \|_1. && \text{reverse triangle inequality}
\end{align*}
Rearranging terms gives $ \| (v - v^*)_{T^c} \|_1 \le \|  (v - v^*)_T \|_1  $, whence
\begin{align*}
\| v - v^* \|_1 = \|  (v - v^*)_T \|_1 + \| (v- v^*)_{T^c} \|_1 &\leq 2 \|  (v - v^*)_T \|_1 \\
&\leq 2\sqrt{k} \|  (v- v^*)_T \|_2
\le 2\sqrt{k} \|  v - v^* \|_2,
\end{align*}
where step the second inequality holds because $ |T| \le k $ by assumption.
The claim follows from noting that $ v-v^* = \Delta_S e_j  $.

Using the claim, we get that
\begin{align*}
\| X^{\top}_{\sharp} \Delta_S \|_F
= \sqrt{\sum_{j\in[d]} \| X^{\top}_{\sharp} \Delta_S e_j \|_2^2 }
& \le \sqrt{\sum_{j\in[d]} \| X_{\sharp} \|_{2,\infty}^2 \| \Delta_S e_j \|_1^2 } \\
& \le \| X_{\sharp} \|_{2,\infty} \sqrt{\sum_{j\in[d]} 4k \| \Delta_S e_j \|_2^2 }
\le 2 \sqrt{\frac{\nu r k}{d}} \| \Delta_S \|_F.
\end{align*}
Using a similar argument and the fact that $ \| \Delta_X \|_{2,\infty} \le \| X\|_{2,\infty} + \| X_{\sharp}\|_{2,\infty} \le 3\sqrt{\frac{\nu r}{d}} $, we obtain
\begin{align*}
\| \Delta_X^{\top} \Delta_S \|_F \le 6 \sqrt{\frac{\nu r k}{d}} \| \Delta_S \|_F.
\end{align*}
Putting everything together, we have
\begin{align*}
| \langle S-S^*, XX^\top -X_{\sharp}X_{\sharp}^\top\rangle | \le \left( 2 \cdot 2 \sqrt{\frac{\nu r k}{d}} \| \Delta_S \|_F + 6 \sqrt{\frac{\nu r k}{d}} \| \Delta_S \|_F \right)  \cdot \| \Delta_X \|_F.
\end{align*}
The claim follows.

\subsection{Proof of Theorem~\ref{sharp_gen_case}}
\label{appendix:proof_sharpness_r1}

Without loss of generality, suppose that $x$ is closer to $\bar x$ than to $-\bar x$. Consider the following expression:
\begin{align*}
&\| \bar x(x - \bar x)^\top + (x - \bar x) \bar x^\top \|_1\\
&= \sup_{\|V\|_\infty = 1, V^\top = V} \trace((\bar x(x - \bar x)^\top + (x - \bar x) \bar x^\top)V)\\
&= \sup_{\|V\|_\infty = 1, V^\top = V} \trace(\bar xx^\top V + x \bar x^\top V - 2\bar x\bar x^\top V)\\
&= \sup_{\|V\|_\infty = 1, V^\top = V} \trace(x^\top V\bar x +  \bar x^\top Vx - 2\bar x^\top V\bar x)\\
&= 2\sup_{\|V\|_\infty = 1, V^\top = V} \trace(x^\top V\bar x  - \bar x^\top V\bar x)\\
&= 2\sup_{\|V\|_\infty = 1, V^\top = V} \trace((x- \bar x)^\top V\bar x)\\
&= 2\sup_{\|V\|_\infty = 1, V^\top = V} \trace(\bar x (x - \bar x )^\top V).
\end{align*}
We now produce a few different lower bounds by testing against different $V$. In what follows, we set $a = \sqrt{2} - 1$, i.e., the positive solution of the equation $1-a^2 = 2a$.
\paragraph{Case 1:} Suppose that
$$
|(x - \bar x )^\top \sign(\bar x)| \geq a\|x - \bar x\|_1.
$$
Then set $\bar V = \sign((x - \bar x )^\top \sign(\bar x)) \cdot  \sign(\bar x)\sign(\bar x)^\top$, to get
\begin{align*}
\| \bar x(x - \bar x)^\top + (x - \bar x) \bar x^\top \|_1 &\geq 2\trace(\bar x (x - \bar x )^\top \bar V)\\
&= 2\sign((x - \bar x )^\top \sign(\bar x))\cdot\trace( (x - \bar x )^\top\sign( \bar x)\sign(\bar x)^\top \bar x) \\
&= 2\|\bar x\|_1\sign((x - \bar x )^\top \sign(\bar x))\cdot (x - \bar x )^\top\sign( \bar x) \\
&\geq 2a\|\bar x\|_1 \|x - \bar x\|_1
\end{align*}

\paragraph{Case 2:} Suppose that
$$
|\sign(x - \bar x )^\top \bar x| \geq a \|\bar x\|_1.
$$
Then set $\bar V = \sign(\sign(x - \bar x )^\top \bar x) \cdot  \sign( x - \bar x)\sign( x - \bar x)^\top$, to get
\begin{align*}
\| \bar x(x - \bar x)^\top + (x - \bar x) \bar x^\top \|_1 &\geq 2\trace(\bar x (x - \bar x )^\top \bar V)\\
&= 2\sign(\sign(x - \bar x )^\top \bar x) \cdot\trace( (x - \bar x )^\top \sign( x - \bar x)\sign( x - \bar x)^\top \bar x) \\
&= 2\| x - \bar x\|_1\sign(\sign(x - \bar x )^\top \bar x)\cdot \sign( x - \bar x)^\top \bar x \\
&\geq 2a\|\bar x\|_1 \|x - \bar x\|_1
\end{align*}

\paragraph{Case 3:} Suppose that
$$
|(x - \bar x )^\top \sign(\bar x)| \leq a\|x - \bar x\|_1 \qquad  \text{and} \qquad |\sign(x - \bar x )^\top \bar x| \leq a \|\bar x\|_1
$$
Define  $\bar V = \frac{1}{2}(\sign(\bar x(x - \bar x)^\top ) + \sign((x - \bar x) \bar x^\top))$. Observe that
\begin{align*}
\trace( \bar x(x - \bar x )^\top \sign(\bar x(x - \bar x)^\top ))&=  (x - \bar x )^\top \sign(\bar x) \sign(x - \bar x)^\top\bar x \\
&\geq-  a^2\|\bar x\|_1 \|x - \bar x\|_1
\end{align*}
and
\begin{align*}
\trace( \bar x(x - \bar x )^\top\sign((x - \bar x) \bar x^\top)) &= \trace( \bar x(x - \bar x )^\top\sign(x - \bar x) \sign(\bar x^\top))\\
&= \|\bar x\|_1\|x - \bar x\|_1.
\end{align*}
Putting these two bounds together, we find that
\begin{align*}
\| \bar x(x - \bar x)^\top + (x - \bar x) \bar x^\top \|_1 &\geq  2\trace(\bar x (x - \bar x )^\top \bar V) = (1-a^2)\|\bar x\|_1\|x - \bar x\|_1.
\end{align*}
Altogether, we find that
\begin{align*}
F(x) &= \|xx^\top - \bar x \bar x^\top\|_1\\
&= \| \bar x (x - \bar x)^\top + (x - \bar x) \bar x^\top + (x - \bar x)(x - \bar x)^\top\|_1 \\
&\geq \| \bar x (x - \bar x)^\top + (x - \bar x) \bar x^\top \|_1 - \|(x - \bar x)(x - \bar x)^\top\|_1\\
&\geq 2a\|\bar x\|_1\|x - \bar x\|_1 - \|(x - \bar x)\|_1^2\\
&=   2a\|\bar x\|_1\left( 1 - \frac{\|x - \bar x\|_1}{2a\|\bar x\|_1}\right)\|x - \bar x\|_1,
\end{align*}
as desired.

\subsection{Proof of Lemma~\ref{lem:bound_outlier_term}}\label{appendix:proof_technical_lemma}
We start by stating a claim we will use to prove the lemma. Let us introduce some notation.
Consider the set
$$
S = \left\{(\Delta_+, \Delta_-) \in \RR^{d \times r } \times \RR^{d \times r}\mid \|\Delta_+\|_{2, \infty} \leq (1+C)\sqrt{\frac{\nu r}{ d}}  \|X_\sharp\|_{op}, \|\Delta_-\|_{2, 1} \neq 0\right\}. 
$$
Define the random variable
\begin{align*}
Z &= \sup_{(\Delta_+, \Delta_-) \in S}  \bigg|  \frac{1}{\|\Delta_-\|_{2, 1}}\sum_{i,j=1}^d \delta_{ij} |\dotp{\Delta_{-,i},\Delta_{+,j}} + \dotp{\Delta_{+,i},\Delta_{-,j}}| \\
&\hspace{60pt}- \EE  \frac{1}{\|\Delta_-\|_{2, 1}}\sum_{i,j=1}^d \delta_{ij} |\dotp{\Delta_{-,i},\Delta_{+,j}} + \dotp{\Delta_{+,i},\Delta_{-,j}}|\bigg|.
\end{align*}
\begin{claim}
	There exist constants $ c_2, c_3 > 0$ such that with probability at least  $1-\exp(-c_2 \log d)$
	\[Z \leq c_3 C\sqrt{ \tau \nu r\log d } \norm{X_\sharp}_{op}.\]
\end{claim}
Before proving this claim, let us show how it implies the theorem. Let
$$
R \in \argmin_{\hat R^\top \hat R = I} \|X - X_\sharp \hat R\|_{2, 1}.
$$
Set $\Delta_- = X -  X_\sharp R$ and $\Delta_+ = X + X_\sharp R$. Notice that
$$
\|\Delta_+\|_{2, \infty} \leq \|X\|_{2, \infty} + \|X_\sharp\|_{2, \infty} \leq (1 + C)  \|X_\sharp\|_{2, \infty} \leq \sqrt{\frac{\nu r}{ d}} (1+C) \|X_\sharp\|_{op}.
$$
Therefore, because $(\Delta_+, \Delta_-) \in S$ and
$$
\frac{1}{\|\Delta_-\|_{2, 1}}\sum_{i,j=1}^d \delta_{ij} |\dotp{X_i,X_j} - \dotp{ (X_\sharp)_i, (X_\sharp)_j}| = \frac{1}{\|\Delta_-\|_{2, 1}}\sum_{i,j=1}^d \delta_{ij} |\dotp{\Delta_{-,i},\Delta_{+,j}} + \dotp{\Delta_{+,i},\Delta_{-,j}}|,
$$
we have that
\begin{align*}
\sum_{i,j=1}^d \delta_{ij} |\dotp{X_i,X_j} - \dotp{(X_\sharp)_i,(X_\sharp)_j}| &\leq \tau \|XX^\top - X_\sharp X_\sharp^\top \|_1 +c_3C \sqrt{\tau\nu r\log d } \| X_\sharp\|_{op} \|X - X_\sharp R \|_{2, 1}\\
&\leq  \left( \tau +  \frac{c_3C \sqrt{\tau \nu r\log d }}{c} \| X_\sharp\|_{op}\right)\|XX^\top - X_\sharp X_\sharp^\top \|_1 ,
\end{align*}
where the last line follows by Conjecture~\ref{conj:sharp_ell1}. This proves the desired result.

\begin{proof}[Proof of the Claim]
	Our goal is to show that the random variable $Z$ is highly concentrated around its mean.
	We may apply the standard symmetrization inequality~\cite[Lemma 11.4]{BouLugMas13} to bound the expectation $\EE Z$ as follows:
	\begin{align*}
	\EE Z &\leq 2\EE \sup_{(\Delta_+, \Delta_-) \in S}  \left| \frac{1}{\|\Delta_-\|_{2, 1}}\sum_{i,j=1}^d \varepsilon_{ij}\delta_{ij} |\dotp{\Delta_{-,i},\Delta_{+,j}} + \dotp{\Delta_{+,i},\Delta_{-,j}}| \right|\\
	&\leq \underbrace{2\EE \sup_{(\Delta_+, \Delta_-) \in S}  \left|\frac{1}{\|\Delta_-\|_{2, 1}}\sum_{i,j=1}^d \varepsilon_{ij}\delta_{ij}| \dotp{\Delta_{-,i},\Delta_{+,j}}|  \right|}_{T_1} + \underbrace{2\EE \sup_{(\Delta_+, \Delta_-) \in S}  \left|\frac{1}{\|\Delta_-\|_{2, 1}}\sum_{i,j=1}^d \varepsilon_{ij}\delta_{ij}  |\dotp{\Delta_{+,i},\Delta_{-,j}} |\right|}_{T_2}.
	\end{align*}
	Observing that $T_1$ and $T_2$ can both be bounded by
	\begin{align*}
	\max\{T_1, T_2\} &\leq 2 \sup_{(\Delta_+, \Delta_-) \in S} \frac{1}{\|\Delta_-\|_{2, 1}}\|\Delta_+\Delta_-^\top\|_{2,\infty} \EE\max_{j} \left|\sum_{i=1}^d \varepsilon_{ij}\delta_{ij}  \right|\\
	&\leq 2 \sup_{(\Delta_+, \Delta_-) \in S} \|\Delta_+\|_{2, \infty} \EE\max_{j} \left|\sum_{i=1}^d \varepsilon_{ij}\delta_{ij}  \right|\\
	&\leq 2(1+C) \sqrt{\frac{\nu r}{d}}\|X_\sharp\|_{op}  \EE\max_{j} \left|\sum_{i=1}^d \varepsilon_{ij}\delta_{ij}  \right| \\
	& \lesssim C \sqrt{\frac{\nu r}{d}}\|X_\sharp\|_{op}  (\sqrt{\tau d \log d} + \log d),
	\end{align*}
	where the final inequality follows from Bernstein's inequality and a union bound, we find that
	$$
	\EE Z \lesssim  C\sqrt{\frac{\nu r}{d}}\|X_\sharp\|_{op}  (\sqrt{\tau d \log d} + \log d).
	$$

	To prove that $Z$ is well concentrated around $ \EE Z$, we apply Theorem~\ref{thm:klein}.
	To apply this theorem, we set $\cS = S$ and define the collection $(Z_{ij,s})_{ij, s\in \cS}$, where $s = (\Delta_+, \Delta_-)$ by
	\begin{align*}
	Z_{ij,s} &=  \frac{1}{\|\Delta_-\|_{2, 1}}\delta_{ij} |\dotp{\Delta_{-,i},\Delta_{+,j}} + \dotp{\Delta_{+,i},\Delta_{-,j}}| - \EE  \frac{1}{\|\Delta_-\|_{2, 1}}\delta_{ij} |\dotp{\Delta_{-,i},\Delta_{+,j}} + \dotp{\Delta_{+,i},\Delta_{-,j}}|\\
	&= \frac{(\delta_{ij} - \tau)}{\|\Delta_-\|_{2, 1}}  |\dotp{\Delta_{-,i},\Delta_{+,j}} + \dotp{\Delta_{+,i},\Delta_{-,j}}|.
	\end{align*}
	We also bound
	\begin{align*}
	b &= \sup_{ij, s \in \cS} |Z_{ij,s}| \leq \sup_{ij, (\Delta_+, \Delta_-) \in S}  \left|  \frac{(\delta_{ij} - \tau)}{\|\Delta_-\|_{2, 1}} (\|\Delta_{-,i}\|_F\|\Delta_{+,j}\|_F + \|\Delta_{+,i}\|_F\|\Delta_{-,j}\|_F) \right|\\
	&\leq  (1+C)\sqrt{\frac{\nu r}{d}}\|X_\sharp\|_{op} \sup_{ij,(\Delta_+, \Delta_-) \in S}  \left|  \frac{1}{\|\Delta_-\|_{2, 1}} (\|\Delta_{-,i}\|_F + \|\Delta_{-,j}\|_F) \right| \leq 2C\sqrt{\frac{\nu r}{d}}\|X_\sharp\|_{op}
	\end{align*}
	and
	\begin{align*}
	\sigma^2 &= \sup_{(\Delta_+, \Delta_-) \in S} \EE  \frac{1}{\|\Delta_-\|_{2, 1}^2} \sum_{ij =1}^d(\delta_{ij} - \tau)^2  |\dotp{\Delta_{-,i},\Delta_{+,j}} + \dotp{\Delta_{+,i},\Delta_{-,j}}|^2\\
	&\leq \tau\sup_{(\Delta_+, \Delta_-) \in S}   \frac{1}{\|\Delta_-\|_{2, 1}^2} \sum_{ij =1}^d  (\|\Delta_{-,i}\|_F\|\Delta_{+,j}\|_F + \|\Delta_{+,i}\|_F\|\Delta_{-,j}\|_F)^2\\
	&\leq \tau \sup_{(\Delta_+, \Delta_-) \in S}   \frac{4}{\|\Delta_-\|_{2, 1}^2} \sum_{ij =1}^d  \|\Delta_{-,i}\|_F^2\|\Delta_{+,j}\|_F^2\\
	&\leq  \tau\frac{4(1+C)^2\nu r}{d}\|X_\sharp\|_{op}^2\sup_{(\Delta_+, \Delta_-) \in S}   \frac{2}{\|\Delta_-\|_{2, 1}^2} \sum_{ij =1}^d  \|\Delta_{-,i}\|_F^2\\
	&\leq  \tau\frac{4(1+C)^2\nu r}{d}\|X_\sharp\|_{op}^2\sup_{(\Delta_+, \Delta_-) \in S}   \frac{2d\|\Delta_-\|_F^2}{\|\Delta_-\|_{2, 1}^2} \\
	&\leq  16\tau C^2\nu r\|X_\sharp\|_{op}^2.
	\end{align*}
	Therefore, due to Theorem~\ref{thm:klein} there exists a constant $c_1, c_2, c_3 > 0$ so that with $t = c_2 \log d$, we have that with probability $1-e^{-c_2\log d}$, the bound
	\begin{align*}
	Z &\leq \EE Z + \sqrt{8\left(2b\EE Z+\sigma{}^{2}\right)t}+8bt\\
	&\leq  c_1C\sqrt{\frac{\nu r}{d}}\|X_\sharp\|_{op}  (\sqrt{\tau d \log d} + \log d) \\
	&\hspace{20pt} + \sqrt{8c_2\left(\frac{c_1^2 C^2\nu r}{d}\|X_\sharp\|_{op}^2  (\sqrt{\tau d \log d} + \log d)+16\tau C^2 \nu r\|X_\sharp\|_{op}^2\right)\log d } + 16 c_2C \sqrt{\frac{\nu r}{d}}\|X_\sharp\|_{op}\log(d)\\
	&\leq C \sqrt{\nu r\log d } \|X_\sharp\|_{op} \left( c_1\sqrt{\tau} + c_1\sqrt{\frac{\log d }{d}} + \sqrt{8c_2}\sqrt{ c_1^2\sqrt{\frac{\tau \log d}{d}} + c_1^2\frac{\log d}{d} + 16\tau } + 16c_2 \sqrt{\frac{\log d}{d}}\right) \\
	&\leq c_3C \sqrt{\tau\nu r \log d } \| X_\sharp\|_{op}.
	\end{align*}
	where the last line follows since by assumption $\log d / d \lesssim \tau.$
\end{proof}

\section{Proofs in Section~\ref{sec:recovery_tol}}
\subsection{Proof of Lemma~\ref{lem:lip_tol}}
The proof follows the same strategy as \cite[Theorem 6.1]{davis2017nonsmooth}.
Fix $x \in \widetilde \cT_1$ and let $\zeta \in \partial \tilde f(x)$. Then for all $y$, we have, from Lemma~\ref{lem:approximate_subgradient}, that
$$
f(y) \geq \tilde f(x) + \dotp{\zeta, y - x} - \frac{\rho}{2} \|x - y\|^2_2 - 3\varepsilon.
$$
Therefore, the function
$$
g(y) := f(y) - \dotp{\zeta, y - x} + \frac{\rho}{2} \|x - y\|^2_2 + 3\varepsilon
$$
satisfies
$$
g(x) - \inf g \leq f(x) - \tilde f(x) + 3\varepsilon \leq 4\varepsilon.
$$
Now, for some $\gamma > 0$ to be determined momentarily, define
$$
\hat x = \argmin \left\{ g(x) + \frac{ \varepsilon}{\gamma^2}\|x - y\|^2_2 \right\}.
$$
First order optimality conditions and the sum rule immediately imply that
\begin{align*}
\frac{2\varepsilon}{\gamma^2} (x -\hat x) \in \partial g(\hat x) = \partial f(\hat x) - \zeta + \rho(\hat x - x).
\end{align*}
Thus,
$$
\dist(\zeta, \partial f(\hat x)) \leq \left(\frac{2\varepsilon}{\gamma^2} + \rho\right)\|x -\hat x\|_2.
$$
Now we estimate $\|x - \hat x\|_2$. Indeed, from the definition of $\hat x$ we have
$$
\frac{\varepsilon}{\gamma^2}\| \hat x - x\|^2 \leq g(x) - g(\hat x) \leq g(x) - \inf g \leq 4\varepsilon.
$$
Consequently, we have $\|x - \hat x\| \leq 2\gamma$. Thus, setting $\gamma = \sqrt{2\varepsilon/\rho}$ and recalling that $\varepsilon \leq \mu^2/56\rho$ we find that
$$
\dist(\hat x, \cX^\ast) \leq \|x - \hat x\| + \dist(x, \cX^\ast) \leq 2\sqrt{\frac{2\varepsilon}{ \rho}} + \frac{\mu}{4\rho} \leq  \frac{\mu}{\rho}.
$$
Likewise, we have
$$
\dist(\hat x, \cX) \leq \|x - \hat x\| \leq 2\sqrt{\frac{2\varepsilon}{ \rho}} .
$$
Therefore, setting $L = \sup\left\{\|\zeta\|_2:\zeta\in \partial  f(x), \dist(x, \cX^\ast) \leq \frac{\mu}{\rho}, \dist(x, \cX) \leq 2\sqrt{\frac{\varepsilon}{\rho}}\right\}$, we find that
\begin{align*}
\|\zeta\| \leq L + \dist(\zeta, \partial f(\hat x)) \leq L + \frac{4\varepsilon}{\gamma} + 2\rho\gamma = L + 2\sqrt{8\rho\varepsilon},
\end{align*}
as desired.

\subsection{Proof of Theorem~\ref{thm:qlinear_tol}}
Let $i \geq 0$, suppose $x_i \in \widetilde \cT_1$, and let $x^\ast \in \proj_{\cX^\ast}(x_i)$. Notice that Lemma~\ref{lem:sharp_on_tube_2} implies $\tilde f(x_i)-\min_{\cX}f>0$. We successively compute
\begin{align}
\|x_{i+1} -  x^*\|^2&=\left\|\proj_{\cX}\left(x_{i}-\tfrac{\tilde f(x_i)-\min_{\mathcal{X}} f}{\|\zeta_i\|^2} \zeta_i\right)-\proj_\cX(x^*)\right\|^2\notag\\
&\leq \left\|(x_{i} - x^*)- \tfrac{\tilde f(x_i)-\min_{\mathcal{X}} f}{\|\zeta_i\|^2} \zeta_i\right\|^2 \label{eqn1:weird_subgrad_tol}\\
&= \|x_{i} -  x^*\|^2 + \frac{2(\tilde f(x_i) - \min_{\cX} f)}{\|\zeta_i\|^2}\cdot\dotp{\zeta_i, x^* - x_{i}} + \frac{(\tilde f(x_i) - \min_\cX f)^2}{\|\zeta_i\|^2} \notag\\
&\leq \|x_{i} -  x^*\|^2 + \frac{2(\tilde f(x_i) -  \min_\cX f)}{\|\zeta_i\|^2}\left( \min_\cX f - \tilde f(x_i) + \frac{\rho}{2}\|x_i -  x^*\|^2 + 3\varepsilon \right)\notag \\
&\qquad\qquad+ \frac{(\tilde f(x_i) - \min_\cX f)^2}{\|\zeta_i\|^2} \label{eqn2:weird_subgrad_tol}\\
&= \|x_{i} -  x^*\|^2 + \frac{\tilde f(x_i) - \min_\cX f}{\|\zeta_i\|^2}\left(\rho\|x_i -  x^*\|^2 -  (\tilde f(x_i) - \min_\cX f) + 6\varepsilon  \right)\notag\\
&\leq \|x_{i} -  x^*\|^2 + \frac{\tilde f(x_i) -  \min_\cX f}{\|\zeta_i\|^2}\left(\rho\|x_i -  x^*\|^2 -  \mu\|x_i -  x^*\|  + 7\varepsilon  \right)\label{eqn3:weird_subgrad_tol} \\
&\leq \|x_{i} -  x^*\|^2 + \frac{\rho (\tilde f(x_i) -  \min_\cX f)}{\|\zeta_i\|^2}\left( \|x_i -  x^*\| -\frac{\mu}{2\rho}\right)\|x_i -  x^*\|\label{eqn:tube_lower}\\
&\leq \|x_{i} -  x^*\|^2 - \frac{\mu(\tilde f(x_i) -  \min_\cX f)}{4\|\zeta_i\|^2}\cdot \|x_i -  x^*\|\label{eqn35:weird_subgrad_tol}\\
&\leq \|x_{i} -  x^*\|^2 - \frac{\mu(\mu \|x_i - x^\ast\| - \varepsilon)}{4\|\zeta_i\|^2}\cdot \|x_i -  x^*\|\label{eqn4:weird_subgrad_tol}\\
&\leq \left(1-\frac{13\mu^2}{56\|\zeta_i\|^2}\right)\|x_i- x^*\|^2\notag.
\end{align}
Here, the estimate \eqref{eqn1:weird_subgrad_tol} follows from the fact that the projection $\proj_Q(\cdot)$ is nonexpansive, \eqref{eqn2:weird_subgrad_tol} uses Lemma~\ref{lem:approximate_subgradient}, the estimate \eqref{eqn:tube_lower} follows from the assumption $\epsilon<\frac{\mu}{14}\|x_k-x^*\|$, the estimate \eqref{eqn35:weird_subgrad_tol} follows from the estimate $\|x_i-x^*\|\leq \frac{\mu}{4\rho}$, while \eqref{eqn3:weird_subgrad_tol} and \eqref{eqn4:weird_subgrad_tol} use Lemma~\ref{lem:sharp_on_tube_2}.
We therefore deduce
\begin{equation*} 
\dist^2(x_{i+1};\cX^*)\leq \|x_{i+1} - x^*\|^2\leq \left(1-\frac{13\mu^2}{56L^2}\right)\dist^2(x_i,\cX^*).
\end{equation*}
Consequently either we have $\dist(x_{i+1}, \cX^\ast) < \frac{14\varepsilon}{\mu}$ or $x_{i+1} \in \widetilde \cT_1$. Therefore, by induction, the proof is complete.

\subsection{Proof of Theorem~\ref{thm:prox_linear_tol}}
Let $i \geq 0$, suppose $x_i \in \cT_\gamma$, and let $x^\ast \in \proj_{\cX^\ast}(x_i)$. Then
\begin{align*}
\mu \dist (x_{i+1}, \cX^\ast) \leq  f(x_{i+1}) - \inf_{\cX} f &\leq f_x(x_{i+1}) - \inf_\cX f + \frac{\rho}{2}\|x_{i+1} - x_i\|^2 \\
&\leq \tilde f_x(x_{i+1}) - \inf_\cX f + \frac{\rho}{2}\|x_{i+1} - x_i\|^2  + \varepsilon \\
&\leq  \tilde f_x(x^\ast) - \inf_{\cX} f + \frac{\beta}{2}\|x_i - x^\ast \|^2 + \varepsilon \\
&\leq   f_x(x^\ast) - \inf_{\cX} f + \frac{\beta}{2}\|x_i - x^\ast \|^2 + 2\varepsilon \\
&\leq    f(x^\ast) - \inf_{\cX} f + \beta \|x_i - x^\ast \|^2 + 2\varepsilon \\
&=   \beta\dist^2(x_i, \cX^\ast) + 2\varepsilon .
\end{align*}
Rearranging yields the result.

\section{Auxiliary lemmas}
\begin{lem}[Lemma 3.1 in \cite{candes2011tight}]
	\label{lemma:eps_net}
	Let $S_r := \left\{X \in \RR^{d_1 \times d_2} \mid \rank(X) \leq r,
	\norm{X}_F = 1\right\}$. There exists an $\epsilon$-net $\mathcal{N}$ (with
	respect to
	$\|\cdot\|_F$) of $S_r$ obeying \[|\mathcal{N}| \leq
	\left(\frac{9}{\epsilon}\right)^{(d_1+d_2+1)r}.\]
\end{lem}
\begin{proposition}[Corollary 1.4 in \cite{rudelson2014small}]
	\label{prop:small_ball}
	Consider $X_1, \dots, X_d$ real-valued random variables and let $\sigma \in \SS^{d-1}$ be a unit vector. Let $t, p > 0$ such that
	\[\sup_{u \in \RR} \PP\left( |X_i - u| \leq t \right) \leq p \qquad \text{for all }i = 1, \dots, d.\]
	Then the following holds
	\[\sup_{u \in \RR}  \PP \left( \left|\sum_k \sigma_k X_k - u\right| \leq t \right)\leq Cp,\]
	where $C > 0$ is a universal constant.
\end{proposition}

\begin{thm}[Talagrand's Functional Bernstein for non-identically distributed variables {\cite[Theorem 1.1(c)]{klein2005concentration}}]\label{thm:klein}
	Let $\mathcal{S}$ be a countable index set. Let $Z_{1},\ldots,Z_{n}$
	be independent vector-valued random variables of the form $Z_{i}=(Z_{i,s})_{s\in\mathcal{S}}$.
	Let $Z:=\sup_{s\in\mathcal{S}}\sum_{i=1}^{n}Z_{i,s}$. Assume that
	for all $i\in[n]$ and $s\in\mathcal{S}$, $\EE Z_{i,s}=0$ and $\left|Z_{i,s}\right|\le b$.
	Let
	\[
	\sigma{}^{2}=\sup_{s\in\mathcal{S}}\sum_{i=1}^{n}\EE Z_{i,s}^{2}.
	\]
	Then for each $t>0$,
	we have the tail bound
	\begin{align*}
	P\left(Z-\EE Z\ge\sqrt{8\left(2b\EE Z+\sigma{}^{2}\right)t}+8bt\right) & \le e^{-t}.
	\end{align*}
\end{thm}

\end{document}